\documentclass{article}

\usepackage[utf8]{inputenc}
\usepackage[a4paper]{geometry}
\usepackage[shortlabels]{enumitem}
\usepackage{amssymb,amsmath,amsthm,mathrsfs}
\usepackage{dsfont}
\usepackage{newunicodechar}
\usepackage{todonotes}

\usepackage[backend=bibtex,style=alphabetic,doi=false,isbn=false,url=false,eprint=false,giveninits=true,maxbibnames=5]{biblatex}
\usepackage{graphicx}
\usepackage{xparse}
\usepackage{color}
\usepackage{constants}
\newconstantfamily{s}{symbol=c,
}
\newconstantfamily{m}{symbol=M}

\usepackage{caption}
\usepackage[labelformat=simple]{subcaption}

\usepackage{enumitem}

\usepackage[colorlinks,linkcolor=darkgreen,citecolor=darkgreen,breaklinks=true]{hyperref}
\numberwithin{equation}{section}
\usepackage{soul}

\DeclareMathOperator{\vol}{Vol}
\DeclareMathOperator{\rad}{rad}
\DeclareMathOperator{\Hill}{H}
\DeclareMathOperator{\Moun}{M}
\DeclareMathOperator{\diam}{diam}

\newcommand{\N}{\ensuremath{\mathbb{N}}}
\newcommand{\R}{\ensuremath{\mathbb{R}}}
\newcommand{\Z}{\ensuremath{\mathbb{Z}}}
\renewcommand{\d}{\, \ensuremath{\mathrm{d}}}

\newcommand{\dv}{\ensuremath{d_v}}
\newcommand{\dw}{\ensuremath{d_w}}
\DeclareMathOperator{\CVol}{C_{V\!ol}}
\DeclareMathOperator{\cvol}{c_{vol}}
\DeclareMathOperator{\Cmix}{C_{mix}}
\DeclareMathOperator{\Cpsi}{C_{ψ}}
\newcommand{\displ}{B}
\newcommand{\Conf}{\mathrm{Conf}}

\renewcommand{\tilde}{\widetilde}
\renewcommand{\hat}{\widehat}
\newcommand{\compl}{^{\mathsf{c}}}

\renewcommand{\P}{\ensuremath{\mathbb{P}}}
\newcommand{\given}{\: \big| \:}											
\newcommand{\Given}{\: \Big| \:}										
\newcommand{\as}{\text{a.s.}}

\newcommand{\ind}{\boldsymbol{1}}
\newcommand{\graph}{\ensuremath{\mathbb{G}^d\times \R}}

\newcommand{\Sier}{Sierpi{\'{n}}ski}
\newcommand{\Lip}{Lipschitz cutset}
\newcommand{\unit}{\ensuremath{l_F}}
\newcommand{\remain}{\ensuremath{m_F}}

\newcommand{\ext}{^{\mathrm{{\rm ext}}} }
\newcommand{\base}{^{\mathrm{base}} }
\newcommand{\inbase}{^{\mathrm{inf}} }
\newcommand{\Sup}{^{\mathrm{sup}} }
\newcommand{\Esup}{^{\mathrm{Esup}} }

\theoremstyle{plain}
\newtheorem{deff}{Definition}[section]

\theoremstyle{plain}
\newtheorem{thm}[deff]{Theorem}

\newtheorem{prop}[deff]{Proposition}
\newtheorem{lemma}[deff]{Lemma}
\newtheorem{corol}[deff]{Corollary}

\theoremstyle{remark}
\newtheorem{remark}[deff]{Remark}

\definecolor{darkgreen}{rgb}{0,0.6,0} 
\definecolor{Green}{rgb}{0.1,0.7,0}
\definecolor{viola}{rgb}{0.7,0,1}

\newcommand{\GG}[1]{{
		\color{Green}{
			\ifmmode \text{\textbf{[G: #1]}}
			\else \textbf{[G: #1]	}	\fi
}}}

\newlength{\leftstackrelawd}
\newlength{\leftstackrelbwd}
\def\leftstackrel#1#2{\settowidth{\leftstackrelawd}%
	{${{}^{#1}}$}\settowidth{\leftstackrelbwd}{$#2$}%
	\addtolength{\leftstackrelawd}{-\leftstackrelbwd}%
	\leavevmode\ifthenelse{\lengthtest{\leftstackrelawd>0pt}}%
	{\kern-.5\leftstackrelawd}{}\mathrel{\mathop{#2}\limits^{#1}}}

\newunicodechar{∂}{\ensuremath{\partial}}
\newunicodechar{∞}{\ensuremath{\infty}}
\newunicodechar{∀}{\ensuremath{\forall}}
\newunicodechar{∃}{\ensuremath{\exists}}
\newunicodechar{∇}{\ensuremath{\nabla}}
\newunicodechar{α}{\ensuremath{\alpha}}
\newunicodechar{σ}{\ensuremath{\sigma}}
\newunicodechar{δ}{\ensuremath{\delta}}
\newunicodechar{φ}{\ensuremath{\varphi}}
\newunicodechar{γ}{\ensuremath{\gamma}}
\newunicodechar{η}{\ensuremath{\eta}}
\newunicodechar{ξ}{\ensuremath{\xi}}
\newunicodechar{κ}{\ensuremath{\kappa}}
\newunicodechar{λ}{\ensuremath{\lambda}}
\newunicodechar{ε}{\ensuremath{\varepsilon}}
\newunicodechar{ρ}{\ensuremath{\rho}}
\newunicodechar{τ}{\ensuremath{\tau}}
\newunicodechar{υ}{\ensuremath{\upsilon}}
\newunicodechar{θ}{\ensuremath{\theta}}
\newunicodechar{ι}{\ensuremath{\iota}}
\newunicodechar{π}{\ensuremath{\pi}}
\newunicodechar{ζ}{\ensuremath{\zeta}}
\newunicodechar{χ}{\ensuremath{\chi}}
\newunicodechar{ψ}{\ensuremath{\psi}}
\newunicodechar{ω}{\ensuremath{\omega}}
\newunicodechar{β}{\ensuremath{\beta}}
\newunicodechar{ν}{\ensuremath{\nu}}
\newunicodechar{μ}{\ensuremath{\mu}}
\newunicodechar{Σ}{\ensuremath{\Sigma}}
\newunicodechar{Δ}{\ensuremath{\Delta}}
\newunicodechar{Φ}{\ensuremath{\Phi}}
\newunicodechar{Γ}{\ensuremath{\Gamma}}
\newunicodechar{Ξ}{\ensuremath{\Xi}}
\newunicodechar{Λ}{\ensuremath{\Lambda}}
\newunicodechar{Θ}{\ensuremath{\Theta}}
\newunicodechar{∏}{\ensuremath{\Pi}}
\newunicodechar{Ψ}{\ensuremath{\Psi}}
\newunicodechar{Ω}{\ensuremath{\Omega}}

\title{\Lip{} for fractal graphs\\ and applications to the spread of infections}
\author{		
		Alexander Drewitz 
		\thanks{Universität zu Köln,
			Department Mathematik/Informatik,
			Weyertal 86--90,
			50931 Köln, Germany.
			Emails: \url{adrewitz@uni-koeln.de},
			\url{gioelegj@gmail.com}}
		\and
		Gioele Gallo $ ^* $ 
		\and 
		Peter Gracar 
		\thanks{University of Leeds,
			School of Mathematics,
			Leeds LS2 9JT, U.K.
			Email: \url{P.Gracar@leeds.ac.uk}}
	}
\date{July 22, 2024}

\begin{document}
\maketitle
\begin{abstract}
	We consider the fractal \Sier{} gasket or carpet graph in dimension $d\geq 2,$ denoted by $G$. At time $0$, we place a Poisson point process of particles onto the graph and let them perform independent simple random walks, which in this setting exhibit sub-diffusive behaviour.
	We generalise the concept of particle process dependent Lipschitz percolation to the (coarse graining of the) space-time graph $G\times \R$, where the opened/closed state of space-time cells is measurable with respect to the particle process inside the cell. We then provide an application of this generalised framework and prove the following: if particles can spread an infection when they share a site of $G$, and if they recover independently at some rate $\gamma>0$, then if $\gamma$ is sufficiently small, the infection started with a single infected particle survives indefinitely with positive probability.
\end{abstract}

\tableofcontents

\section{Introduction}\label{SEC-Intro}
We consider the following setting: a collection of particles is placed on an infinite connected graph at time $0$ in such a way that it is in Poissonian equilibrium for simple random walk -- put simply, letting the particles perform simple random walks does not change their distribution on the graph. Then, over time, each particle performs an independent continuous time simple random walk on the graph. Assume that at time $0$ a single additional infected particle is placed somewhere on the graph and consider the infection dynamics to be as follows: whenever a particle shares a vertex with an infected particle, it instantaneously becomes infected itself. Infected particles can also recover and become healthy/susceptible again, which occurs independently for each infected particle at some exponentially distributed random time. Due to the infection mechanism outlined above, a particle can only truly recover when it is the sole particle at a vertex; otherwise it gets reinfected straight away by one of the other particles sharing its location.

This problem has been studied in various forms, and it can be traced back in the literature at least to Kesten and Sidoravicius. In \cite{KestenSid05}, the authors consider the graph to be the nearest neighbour square lattice $\Z^d$ and treat the case where infected particles never recover. They show that for large times and with high probability, the sites of $\Z^d$ that have already been visited by an infected particle contain a ball around the site where the infection started of radius proportional to the time elapsed since the start of the infection. They also prove that these sites are themselves completely contained in a bigger ball of radius that is also proportional to the same time, again with high probability. In \cite{KestenSid08} they refine this result and prove a shape theorem for the infection. In a parallel paper \cite{KestenSid06}, they study the case of infection with recovery on $\Z^d$ and prove the existence of a phase transition with respect to the recovery rate of the particles: for rates higher than a critical threshold, the infection will almost surely go extinct (i.e.\ no infected particle remains after some finite time), whereas for rates below this threshold, the infection will with positive probability survive indefinitely. It should be noted that in \cite{KestenSid06}, infections occur only when an infected particle jumps onto a site, meaning it is possible for healthy and infected particles to share a site. Our main result, stated in Theorem \ref{thm-application} holds for both infection mechanisms -- the one outlined at the beginning and the one from \cite{KestenSid06}.

More recently, Gracar and Stauffer \cite{Pete19,Pete19b} have developed a general framework with which they were able to prove that on the weighted graph $(\Z^d,\lambda)$, with edges equipped with uniformly elliptic conductances $\lambda_{x,y}$, the infection still spreads with positive speed. They also showed that in the case of infection with recovery, the infection not only survives indefinitely with positive probability, but they also derived a lower bound for the furthermost location the infection has reached divided by time---a question that was left unanswered previously. A further application of this framework can be found in \cite{BaldassoStauffer}, where it is shown that in the case of infection with recovery, conditioned on the infection surviving, the origin of $\Z^d$ (i.e.\ where the infection is started) is visited by an infected particle at arbitrarily large times. The key benefit of the framework used in these works is that it can be applied to different variations of the Poisson random walks and infection models, and that the multi-scale analysis which is done in order to set up the framework does not need to be redone from scratch when the type of event studied changes. Given a local, translation invariant and increasing event that has high enough probability, the framework provides the existence of a connected surface in space-time where the event holds. If the local event is chosen to be the successful propagation of the infection to its immediate surroundings in a predetermined amount of time, then the connected surface yields a lower bound on how far the infection can spread over time. This surface also acts as a cutset in space-time, separating the origin from infinity, so that any particle which visits the origin has to intersect the surface at some later time which is a key property that allows the above argument to be applicable. We also refer to \cite{MR2060478, MR2294981, MR3372847} for a non-exhaustive list of further related models.

In this work, we adapt the framework to a new class of graphs, i.e.\ to sub-diffusive fractal lattice graphs, also known as prefactals. In particular, we study the behaviour of a particle system on the \Sier{} gasket and on generalised \Sier{} carpets. Intuitively, these are the graphs of the famous triangle and square based fractals, where instead of repeating the construction recursively inwards, one instead expands outwards, by attaching copies of the current stage of the graph recursively.
A key difference between the standard Euclidean lattice ($\Z^d$ as well as for example the triangle or hexagonal lattice nearest neighbour graph) and the graphs we study is that random walks on the latter  exhibit subdiffusive behaviour. I.e., random walks move through the graph much more slowly than e.g.\ on the Euclidean lattice,  and it takes on average $r^{\dw}$ amount of time to leave a ball of radius $r$, where $\dw> 2$ is a constant that depends on the dimension of the graph, and on which parts of the graph are missing. Compared to Euclidean lattices, where this average is of order $r^2$ regardless of the dimension of the lattice, this shows that on such fractal graphs random walks exhibit a quantitatively different behaviour.
Crucially, this slower movement of the particles makes it unclear whether the dynamics of the infection process remain unchanged or whether the infection has a harder or easier time surviving over time. Our main result, stated in Theorem \ref{thm-application} provides an answer to this question: although the changed dynamics potentially affect the global propagation of the infection, the mechanism by which the infection survives on Euclidean lattices still remains in place and the infection has a positive probability of survival if the recovery rate is not too high.

In order to state the result more precisely, we quickly formalise some of the above concepts. Let $G$ be either the \Sier{} gasket graph or a generalised \Sier{} carpet graph defined precisely in Sections \ref{sec:Sierp} and \ref{SEC-Carpets}. (See also the corresponding Figures \ref{fig-Sierp1} and \ref{fig-SierpCarpet}.)
In our first result we adapt the so-called Lipschitz surface framework from \cite{Pete19} to the fractal graph case. Notably, although the framework retains its main characteristics, it requires substantial changes and new ideas across the board due to the significantly changed geometry of the graph, starting with the analogue of the Lipschitz surface for fractal graphs. While interesting from a purely mathematical point of view, it is also intriguing to understand if and how the introduction of a prefractal---which can be interpreted as containing obstacles on infinitely many scales---leads to a different quantitative and qualitative behaviour when compared to the Euclidean setting. On $\Z^d$, the framework gives rise to a discrete, Lipschitz connected surface in (a coarse-grained) space-time graph $\Z^{d+1}$. On the fractal graphs we study, we cannot hope for such a strong connectivity property. However, as we define in Subsection \ref{subsec-mainres} and prove in Section \ref{SEC-LipSurf}, the corresponding object still acts as a cutset on the coarse-graining of the space-time graph, meaning that any path escaping toward infinity must intersect this cutset (cf.\ Definition \ref{def-LipschitzSurf}). Furthermore, it is in some sense minimal and still retains the Lipschitz connectivity property along the time dimension (cf.\ Corollary \ref{corol-LipschitzInTime}). This object will be referred to as \Lip{}. We prove in  Theorem \ref{thm-main} that such a \Lip{} exists \as{}, and in Theorem \ref{thm-main2} that it surrounds the origin within a finite distance \as{}
The \Lip{} retains the flexibility of the Lipschitz surface and we expect that it can be taken advantage of to derive further interesting consequences. We refer to Section \ref{sec:further} for further details and present now one example. 

For this purpose, consider the infection process with recovery as outlined above, where at the beginning there is an independent Poisson distributed with intensity $μ_0$ number of particles at each vertex of the graph, and $γ$ is the rate at which infected particles recover. We say that the infection survives if for every time  there exists at least one infected particle somewhere on the graph. Our first main result is the following.

\begin{thm}\label{thm-application}
	For any $μ_0>0,$ there exists $γ_0>0$ such that for all $γ \in (0,γ_0),$ the infection with recovery on $G$ survives with positive probability.
\end{thm}



This paper is structured as follows. In Section \ref{SEC-Setting} we define the \Sier{} gasket graph and formalise the definitions and basic properties outlined above. We can then also state the two main technical Theorems \ref{thm-main} and \ref{thm-main2} which give the existence and key properties of the \Lip{}. In Section \ref{SEC-LipSurf} we construct the \Lip{} and provide a sufficient condition for its existence, as well as prove its key geometric properties. Section \ref{SEC-Mixing} covers a tool used in our multi-scale analysis, a decoupling theorem that allows us under the right conditions to resample particles independently. In Section \ref{SEC-MultiScaleSetup} we define the multi-scale tessellation of the space-time graph and its properties which lead to the proof of Theorem \ref{thm-main} in Section \ref{SEC-MultiScaleAnalysis}. We prove Theorem \ref{thm-main2} in Section \ref{SEC-ProofThm2} with an extension of the multi-scale argument developed before. Section \ref{SEC-Carpets} covers the adaptation of the results which are written with the \Sier{} gasket graph in mind to the case of generalised \Sier{} carpet graphs. The paper concludes with Section \ref{SEC-applications}, where Theorem \ref{thm-application} is proven by  applications of Theorems \ref{thm-main} and \ref{thm-main2}.

Throughout this work we will denote constant with $ c_0,c_1,\dots $ and $ C_1,C_2,\dots.$ Important constants that should be kept track of will be denoted differently: this includes $C_λ,\  \Cmix,\  C_ψ $ as well as the constants $ \Cr{Mixing1},\Cr{MixingDelta},\Cr{MixingRoot},\Cr{MixingProb} , Θ $ from  the decoupling Theorem \ref{thm-mixing}.

\section{Settings and definitions}\label{SEC-Setting}
We start by defining the \Sier{} graph and the coarse-graining which we will use throughout the paper. We then proceed to formally define the particle system we will be studying before stating the two main results of this paper.

\subsection{The \Sier{} gasket graph} \label{sec:Sierp}

The \Sier{} gasket is a fractal which was introduced in \cite{Sier}. Here we define the \Sier{} graph or \Sier{} prefractal based on the \Sier{} fractal with a recursive construction as presented in \cite{Delmotte02}.
Consider any of the graphs obtained from the $ d $-dimensional unit side-length regular simplex in $ \R^d $, $d\ge 2,$ by placing one vertex in the origin. Fix such a graph and denote it with
$\bigtriangleup^d $. More precisely, $\bigtriangleup^d :=(V,E)$ where $V$ are the $d+1$ vertices corresponding to the corners of the simplex and $E$ is the set of all undirected pairs of vertices which share an edge in the simplex.
For $d=2,$ this is the graph induced by the equilateral triangle with unit length sides, motivating the notation $\bigtriangleup^d $. In $d=3$, the graph is induced by the equilateral tetrahedron. We furthermore assume the graph to be weighted with conductances $λ:= (λ_{x,y})_{\{x,y\}\in E} ,$ which are positive symmetric and we assume the existence of a constant $ C_λ \in (0,\infty)$ such that the conductances are uniformly elliptic, i.e.
\begin{equation}\label{def-conductances}
	\frac{1}{C_λ}	\le		λ_{x,y}		\le		C_λ.
\end{equation}
Define now $ \bigtriangleup^d_0:=\bigtriangleup^d $ and iteratively the graph of scale $ n $, for $ n\ge 1, $ as
\begin{equation}\label{def-gasketRecursive}
	\bigtriangleup^d_n:= \bigcup_{x\in 2^{n-1} \bigtriangleup^d_0}  \big( x+\bigtriangleup^d_{n-1} \big),
\end{equation}
taking care of identifying overlapping vertices at the junctions; edges carry the same conductance as in $ \bigtriangleup^d_0 $, i.e.\ for any $n\ge 1$, $z\in 2^{n-1} \bigtriangleup^d_0$ and $x,y\in \bigtriangleup^d_0,$ the conductance on the edge $(z+x,z+y)$ is $λ_{x,y}$. The $ d $-dimensional \Sier{} graph $ \mathbb{G}^d $ is the graph obtained by taking the union of $\bigtriangleup^d_n$ over $n \in \N_0.$ We write $ x\sim y $ if there is an edge between $ x $ and $ y $, and let $ λ_x:=\sum_{y\sim x}λ_{x,y} $. We will denote by $ (\mathbb{G}^d,(λ_{x,y})_{x\sim y}) $ the weighted graph $\mathbb{G}^d$ with conductances $(λ_{x,y})_{x\sim y}$.

We introduce  the set
\begin{equation}\label{def-binaryLattice}
	\mathbb{B}^d:= \big\{ι\in \mathbb{G}^d \colon   ι+\bigtriangleup^d_0 \text{ is a subgraph of }\mathbb{G}^d \big\},
\end{equation}
which intuitively contains those vertices in $\mathbb{G}^d$ which are the ``lower left'' corner of some translation of the simplex $\bigtriangleup^d$ in $\mathbb{G}^d$.
Note that this set is stable under multiplication with powers of 2 in the sense that for all $ m\in\N_0 $ and $ ι\in \mathbb{B}^d $,
\begin{equation}\label{eq-PropOfB}
	ι2^m + \bigtriangleup^d_m \text{ is a subgraph of } \mathbb{G}^d.
\end{equation}

We consider the natural graph distance $ d(\cdot,\cdot) $ on $ \mathbb{G}^d $ and define the distance between sets as the usual minimum of the distances between vertices contained therein.
For a finite set $A$ we define the volume $\vol(A):=|A|$ as the cardinality of the set $A$.
Define the ball of radius $ r\ge 0 $ with center $ x\in \mathbb{G}^d$ as $ B_r(x):=\{y\in \mathbb{G}^d\colon d(x,y)\le r\}$, and the volume of such balls $\vol_r(x):=\vol(B_r(x)).$ Note that the conductances do neither affect $d(\cdot,\cdot)$ nor the volume.

\begin{figure}[ht]
	\centering
	\begin{subfigure}[t]{0.4\textwidth}
		\includegraphics[height=\linewidth]{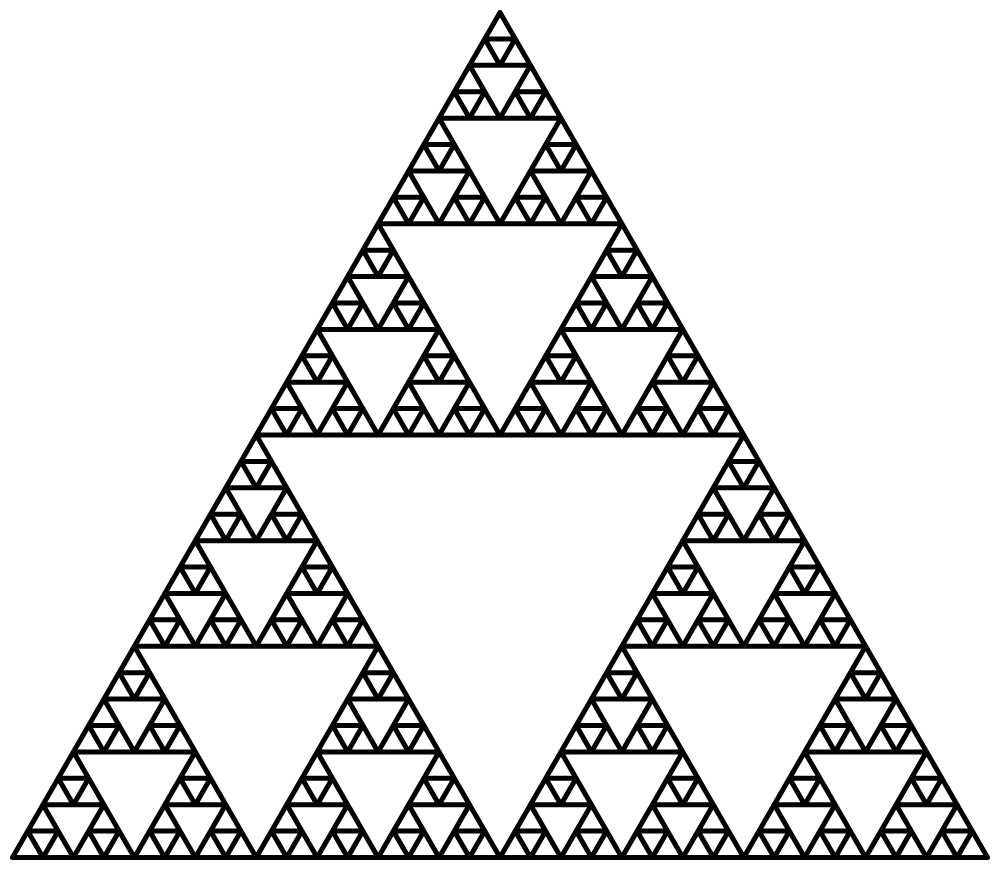}
		\subcaption{$d=2$}
	\end{subfigure}
	\hfill
	\begin{subfigure}[t]{0.4\textwidth}
		\includegraphics[height=\linewidth]{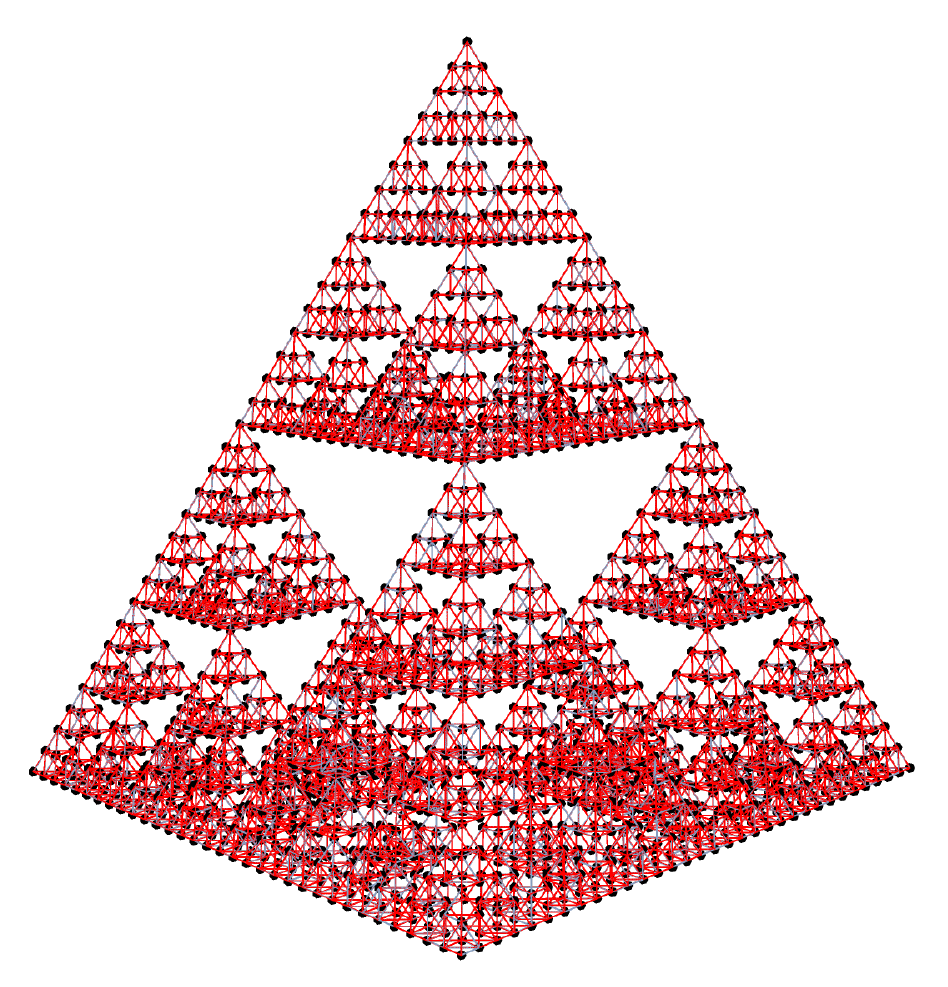}
		\subcaption{$d=3$}
	\end{subfigure}
	\caption{The first six stages of the \Sier{} Gasket.}\label{fig-Sierp1}
\end{figure}

It can be shown that for each $d\ge 2$, there exist constants $ \cvol, \CVol \in (0,\infty) $ (depending on the dimension) such that for all $ x\in \mathbb{G}^d$ and $ r \ge 1 $
\begin{equation}\label{def-dv}
	\cvol \, r^{\dv} \le \vol_r(x)\le \CVol\,r^{\dv},
	\tag{{Vol($ \dv) $}}
\end{equation}
and we call $\dv$ the \emph{volume dimension} of the graph. We refer to the discussion below \eqref{def-dw} for a brief list of the different names of $ \dv $ in the literature.
It is well-known that in dimension two we have $ \dv=\log_2(3) $. To show that \eqref{def-dv} holds in any dimension $ d $, it is not hard to generalise the proof in \cite{BarlowLEC98} in order to obtain that $ \dv=\log_2(d+1). $

We now present a regular coarse graining---referred to as \emph{tessellation}---of the space-time space $\graph$ which we need in order to state the theorems. This definition will be in line with the more complex tessellation presented in Subsection \ref{SUBSEC-Tessel}.

\subsection{First level tessellation} \label{subsec-firstlevel}

\begin{deff}\label{def-Tessel1}
	For a given value $ \ell\in\mathbb{N}_0, $ we tessellate the graph $ \mathbb{G}^d $ into \emph{tiles} $S_1(ι) := ι2^{\ell} +\bigtriangleup^d_{\ell}, $ for $ ι\in \mathbb{B}^d, $ so that each tile is indexed by $ ι $ and has side length equal to $ 2^{\ell}$.

	For a given value $ β>0, $ we tessellate $ \R $ (which will play the role of time) into intervals $T_1(τ)  := [τβ,(τ+1)β), $ indexed by $ τ\in \Z. $

	We then define a \emph{(space-time) cell indexed by $ (ι,τ) \in \mathbb{B}^d \times \Z $} to be the set 
 \begin{equation}\label{eq:cell}
     R_1(ι,τ):=  S_1(ι)\times T_1(τ).
 \end{equation}
\end{deff}
We will consider time starting at $0$ and so one would expect us to work with $\mathbb{R}_+$ instead of $\mathbb{R}$. Our multi-scale framework will however require us to consider ``earlier'' cells relative to a given space-time cell when dealing with larger scales. In particular, this includes cells that lie in the ``past'' relative to the start of the process at time $0$. Due to us considering only collections of particles that are in stationarity (we will formalise this soon) we can therefore work naturally with negative times as well and define the necessary notation already at this stage.
A careful reader might also wonder about the choice of using the index $1$ instead of $\ell$; as we will see later, the above defined tiles, intervals and cells will form the building blocks of our argument at scale 1, with tiles, intervals and cells of larger scales having a corresponding index.
\begin{remark}\label{rem:regions}
When referring to subsets of the spatial graph $\mathbb{G}^d $ in general, such as tiles, unions of tiles or balls on the graph, we will refer to them as \emph{regions} or \emph{subregions} when the distinction between the kind of subset does not play a role.
\end{remark}

Later on, we might refer to generic cells with shorter notation such as simply $ u$ or $v$ when we do not need to specify the indices of the cell. This will usually be in conjunction with some set of cells, where we will write $u\in A$ as a shorthand for $R_1(\iota,\tau)\in A$ (see for example \eqref{def-L1} and the text immediately thereafter).
\begin{deff}\label{def-adjacent}
	We say two  cells $ R_1(ι_1,τ_1)$ $ \ne$ $ R_1(ι_2,τ_2) $ are \emph{adjacent} if either
	$ι_1=ι_2$ and $|τ_1-τ_2|\le 1$ or else if $d(S_1(ι_1),S_1(ι_2))=0$ and $τ_1=τ_2.$
\end{deff}

\begin{remark}
	We could alternatively define $S_1(ι)$ to be ``half-open'' in the sense that only the ``corner'' corresponding to $ι$  is in $S_1(ι)$ while all other corners are not, making the tiles disjoint. This distinction makes no difference for the combinatorial arguments we will use; it could however be important for the lowest level events one could consider (cf.\ Definition \ref{def-E}) in the application of our framework.
\end{remark}

We will use this space-time tessellation in order to define a dependent percolation model where space-time cells will be \emph{good} or \emph{bad} depending on whether a given event dependent on the particle behaviour occurs roughly in the region defined by the corresponding $S_1(\iota)$ during the time interval $T_1(\tau)$. More precisely,
the events we will consider will not necessarily be localised entirely within $ S_1(ι) $. Instead, they will involve larger regions which in particular may intersect for different pairs $(\iota,\tau)$ and $(\iota',\tau')$. To this end we introduce the following extension.

\begin{deff}\label{def-Tessel1η}
	Let $ η\in \N$. For $ι \in \mathbb{B}^d$ we define the \emph{super-tile}
	\begin{equation*}
		S_1^{η}(ι):= \bigcup_{ι' \in \mathbb{B}^d \colon d(ι,ι')\le η} S_1(ι'),
	\end{equation*}
	and for $τ\in\Z$ the \emph{super-interval} $ T_1^η(τ):=[τβ_1,(τ+η)β_1),$ as well as the \emph{super-cell} $ R_1^η(ι,τ) $ as $S_1^{η}(ι)\times T_1^η(τ)$.
\end{deff}

\subsection{Random walks on the \Sier{} graph}


We will study Poisson random walks and for this purpose we start by analysing  properties of the simple random walk on \Sier{} gaskets.
We call a stochastic process $(X_t)_{t \ge 0}$ taking values in $\mathbb{G}^d$ and starting in $x_0 \in \mathbb{G}^d$ a (continuous time simple) random walk on $\mathbb{G}^d$ under the probability measure $P_{x_0}$, if $X_0=x_0$ holds $P_{x_0}$-\as{}, and while at $x \in \mathbb{G}^d$, it jumps to $y \sim x$ with rate $\lambda_{x,y}/\lambda_x.$ The corresponding expectation is then denoted by $E_{x_0}.$
We say that a function $f\colon \mathbb{G}^d\times \R\to \R$ is \emph{caloric} if satisfies the discrete heat equation
\begin{equation*}
	\frac{∂}{∂t} f(x,t) = \sum_{y\sim x} \frac{λ_{x,y}}{λ_x}(f(y,t)-f(x,t))
\end{equation*}
and it is easy to verify that the heat kernel $p_t(x,y):=\frac{1}{λ_y}P_x(X_t=y)$ seen as a function of $y$ and $t,$ with $x$ fixed, satisfies it.

It is well known that the heat kernels for a random walk on $\Z^d$ satisfy Gaussian estimates. Instead, the \Sier{} gasket falls into the class of nested fractals studied in \cite[Corollary 4.13]{HamblyKumagai}, which shows the validity of sharp upper and lower bounds for the heat kernel: denoting by $p_n(x,y):=\frac{1}{λ_y}P_x(X_n=y)$ the heat kernel for the discrete time random walk, it holds that
\begin{align}\label{eq-HKBoundDiscrete}
	p_n(x,y) \asymp n^{- \frac{\dv}{\dw}} \exp		\Big\{		-\Big(	 \frac{d(x,y)^{\dw}}{\C[s]n}		\Big)^{1/(\dw-1)}	\Big\}
	\\ 
	\intertext{for $n > d(x,y),$ with $\dv$ as in \eqref{def-dv} and the \emph{walk dimension} $\dw$ will be motivated in \eqref{def-dw} below. Here, $\asymp$ indicates that the ratios of the two sides are bounded from above and below by positive constants independent of $x,y$ and $n$ in the regime $n > d(x,y).$ This result was first shown on $\mathbb{G}^2$ in \cite{Jones}. Using the fact that the continuous time random walk $X_t$ has jump rate 1, one can extend the proof of \cite[Theorem 2.5.6]{LawlerLimic} to obtain a continuous time version of \eqref{eq-HKBoundDiscrete}: for all $x,y\in\mathbb{G}^d$ and $t>0$ with $d(x,y)< t$ it holds that}
	\label{eq-HKBound}		\tag{{HKB($ \dv,\dw) $}}
	p_t(x,y) \asymp t^{- \frac{\dv}{\dw}} \exp		\Big\{		-\Big(	 \frac{d(x,y)^{\dw}}{\C[s]t}		\Big)^{1/(\dw-1)}	\Big\}.
\end{align}

We say that the parabolic Harnack inequality holds for the graph $\mathbb{G}^d$ if there exists a constant $ \Cl{PHI} >0$ such that for all $ x\in \mathbb{G}^d, $ $ R\ge 1 $ and non-negative $ h\colon \mathbb{G}^d\times \R \to \R  $ 
caloric in $ B_{2R}(x)\times (0,4R^{\dw}) $ satisfies
\begin{equation}\label{eq-PHI}		\tag{PH(\dw)}
	\sup_{B_R(z)\times [R^{\dw},2R^{\dw}]	}	h(x,t)
	\le 	\Cr{PHI}
	\inf_{B_R(z)\times [3R^{\dw},4R^{\dw}]	}	h(x,t)	.
\end{equation}

Next, we introduce the walk dimension, and  for this purpose, for  any subset $B$ of the graph $\mathbb{G}^d$ we write $H_B:=\inf\{t>0\colon X_t\in B\}.$
We say that the graph has \emph{walk dimension} $\dw$, if
\begin{equation}\label{def-dw}\tag{E$(\dw)$}
	E_x[H_{B_r(x)^{\mathsf{c}}}] \asymp r^{\dw}
\end{equation}
for all $x\in \mathbb{G}^d$ and $r \ge 1.$
In the literature, the volume dimension \eqref{def-dv} and walk dimension \eqref{def-dw} are often referred to by different symbols: for example \cite{BarlowLEC98} uses $d_f$ and $d_w$ respectively, \cite{Barlow04} uses $α$ and $β$, \cite{Jones} $\frac{d_s d_w}{2}$ and $d_w$, and \cite{Delmotte02} uses $d_f$ for the volume dimension.

It is proven that the gasket in dimension $d=2$ has walk dimension $\dw= \log_2(5)$ (see for example \cite{BarlowLEC98} or \cite{GrigorYangWD}).

Next, we note that Theorem 3.1 of \cite{Grigor02} establishes the validity of the following implications:
\begin{equation} \label{eq:propsEquiv}
	\eqref{def-dv} + \eqref{eq-HKBound}
	\Longleftrightarrow
	\eqref{eq-PHI}
	\Longrightarrow \eqref{def-dw}.
\end{equation}
Hence, 
having established the validity of \eqref{def-dv} and \eqref{eq-HKBound} on the \Sier{} gasket $\mathbb{G}^d$,
 the validity of \eqref{def-dw} and \eqref{eq-PHI} is obtained from \eqref{eq:propsEquiv} in this setting also.

In this context we also note that volume and walk dimensions are related: indeed, for any graph which satisfies \eqref{def-dv} and \eqref{def-dw}, we have
\begin{equation}\label{eq-dvdw}
	2\le  \dw \le \dv +1;
\end{equation}
see e.g.\ \cite[Theorem 1]{Barlow04} for a proof. 
We will also need the following folklore estimate on the confinement probability, which is a direct consequence of the estimates on the exit probability $ Ψ_n(x,R) $ in \cite[Proposition 7.1]{Grigor01} on a graph with arbitrary random walk dimension.

\begin{lemma}\label{lemma-confining}
	Let $(X_t)$ be a random walk on $(G,λ)$ 
 such that \eqref{def-dw} holds. 
	 Then there exists $ \Cl[s]{conf1} \in (0,\infty)$ such that for all $ \mathbf{Δ}, r>0 $ with $ \mathbf{Δ}>\Cr{conf1} r ,$ the event
	\begin{equation*}
		\Conf(\displ_r, \mathbf{Δ}):= \big\{X_t \in B_r(X_0) \text{ for all } t\in [0,\mathbf{Δ}]\big\}
	\end{equation*}
	satisfies 
	\begin{equation}\label{eq-confining}\tag{{Conf($\dw) $}}
		\inf_{x_0 \in \mathbb{G}^d}P_{x_0}\big(\Conf(\displ_r,\mathbf{Δ})\big)	\ge 1-\Cr{conf1} e^{-\Cr{conf1}^{-1}\big( \frac{r^{\dw} }{\mathbf{Δ}} \big)^{\frac{1}{\dw-1}}			}.
	\end{equation}
    If $\Conf(\displ_r, \mathbf{Δ})$ holds we say that the random walk $X$ is \emph{confined to $B_r = B_r(X_0)$ during $[0,\Delta]$}.
\end{lemma}

\subsection{Poisson particle system}\label{subsec-PRW}

We are now going to define a particle configuration $∏$ as a  function in $(\N_0)^{\mathbb{G}^d}$, where $∏(x)$ is to be interpreted as the number of particles at $x\in\mathbb{G}^d.$ We denote by $φ_x$ the coordinate evaluation of $∏$ defined by $φ_x(∏)= ∏(x)$ and denote by $\mathcal{F}$ the σ-algebra generated by the coordinate maps.

We define a particle system as a family of particle configurations $(\Pi_t)_{t \in \R}\in (Ω,\mathcal{F}'),$ with $Ω:=\{f\colon (-\infty,+∞)\to (\N_0)^{\mathbb{G}^d}\}$ and $\mathcal{F}':=\mathcal{F}^{\otimes \R}$ the product σ-algebra of $\mathcal{F}$ over $\R$ as follows: we define $(∏_t)$  under a probability measure $\P$ a Poisson point process of random walkers with intensity given by $ μ(x):= μ_0 λ_x$ for $ x\in \mathbb{G}^d $ and some $ μ_0>0. $
It is easy to verify that the particle system is stationary (in fact, even reversible) in the sense that at any time $ t\in \R$, the particles remain distributed according to a Poisson point process with intensity $ μ. $ This system is often referred to as \emph{Poisson random walks.}

We say that an event $E\in\mathcal{F}'$ is increasing for the particle system $(∏_t)_{t\in\R}$ if the fact that $E$ holds for $(∏_t)_{t\in\R}$ implies that $E$ holds for all particle systems $(∏'_t)_{t\in\R} $ with $∏'_s\ge ∏_s$ for all $s\ge 0$, where $∏'_s\ge ∏_s$ indicates that $∏'_s(x)\ge ∏_s(x)$ for all $x\in\mathbb{G}^d.$

We now define what it means for an event to be measurable with respect to a particle system. Although one could define this for an arbitrary particle system, we will consider events that are measurable with respect to the more restrictive Poisson random walks particle system from above. In particular, this means that we will consider events that are measurable with respect not only to the locations of the particles at different times, but also their movements over time.

\begin{deff}
	Let $A\subseteq \mathbb{G}^d,$  $t_0\in\R$ and $t_1>0$. For the particle system $(∏_t)_{t \in \R}$ as above we denote by $ P_{x,t_0}:=(p_{x,t_0,i})_{i=1}^{∏_{t_0}(x)}$ the family of particles (including their movements over time) that are located at $x$ at time $t_0$ and with $P_{x,t_0,i}(t)$ the position of particle $p_{x,t_0,i}$ at time $t$,
	we say that an event $E$ is \emph{restricted} to $A$  and a time interval $ [t_0, t_0+t_1] $ if it is measurable with respect to $σ\big\{P_{x,t_0,i}(t), i\in \{1,\dots,∏(x)\}, x\in A,t\in [t_0,t_0+t_1]\big\}$. 
\end{deff}

\begin{deff} Let $r>0$, and $t_0, t_1$ as in the previous definition, we say that a particle is \emph{confined} inside $ B_r$   during $ [t_0, t_0+t_1] $ if during the time interval $[t_0,t_0+t_1]$ it stays inside the ball $ B_r(x) $, where $ x $ is the location of the particle at time $ t_0. $
\end{deff}
Recall that the probability of being confined has been estimated in \eqref{eq-confining}.
We define now the probability associated to an event $ E $.
\begin{deff}\label{def-E}
	For $μ(x)=μ_0 λ_x,$  an increasing event $E$ restricted to $A\subseteq \mathbb{G}^d$ and $[0,t] $, we define, recalling the notation $B_r$ from Lemma \ref{lemma-confining},
 \begin{align*} 
 &ν_E(μ,A,B_r,t) :=\\
 &\P \big(E\given 
    \text{the particles in $A$ at time $0$ with initial density $μ$ are confined inside $B_r$ during } [0,t]
\big).
 \end{align*}
 
\end{deff}

\subsection{Main results}\label{subsec-mainres}

We now provide the final definitions necessary to state the main theorems.
For each $ (ι,τ)\in \mathbb{B}^d\times \Z $ we will  denote by $ E(ι,τ) $ an arbitrary increasing event restricted to the super-cell $ R_1^η(ι,τ) $. We will call the cell $ R_1(ι,τ) $ \emph{bad} if the event $ E(ι,τ) $ does not hold, and \emph{good} otherwise.
We next introduce a \emph{base} of the space-time graph \graph. Recalling the definition of the gasket via $ \bigtriangleup^d_n $ in \eqref{def-gasketRecursive}, we consider the $ (d-1) $-dimensional subgraph $ \bigtriangleup^{d-1}_n $ containing the origin defined in the same way, and taking the union of $ \bigtriangleup^{d-1}_n $ across $ n\in\N$ we obtain the $ (d-1 )$-dimensional \Sier{} gasket $ \mathbb{G}^{d-1}, $ which by construction is a subgraph of $ \mathbb{G}^{d}. $ Intuitively, this corresponds to the Euclidean space identification of the square lattice $\Z^2$ with the subgraph $\Z^2\times\{0\}$ of $\Z^3$ and the origin of $\Z^2$ with the origin in $\Z^3$. Just like in the square lattice case, the choice of which subgraph of $\mathbb{G}^d$ to identify with $\mathbb{G}^{d-1}$ is not unique and can be chosen arbitrarily among the admissible ones.

We now define the \emph{base} of the space-time tessellation as
\begin{equation}\label{def-L0}
	L_0 := \mathbb{G}^{d-1}\times \Z
\end{equation}
seen as a subgraph of $\mathbb{G}^d\times\Z$ as explained above, and the \emph{base of cells}
\begin{equation}\label{def-L1}
	L_1 := \bigcup_{(ι,τ)\in L_0 \cap (\mathbb{B}^d\times \Z)} \{R_1(ι,τ)\}.
\end{equation}
We will often consider the distance
\begin{equation*}
	d(R_1(ι,τ), L_0):=\min_{x\in R_1(ι,τ),\ y\in L_0} d(x,y)
\end{equation*}
between a cell $ R_1(ι,τ) \subseteq\mathbb{G}^d\times\Z$ and the base $ L_0 $, which we will refer to as the \emph{height} of the cell; it may help to visualise the base $ L_0 $ to lie ``horizontally'' as a subgraph of $\mathbb{G}^d\times\Z$. We can now finally define our last central object, the \Lip. We recall the definition of adjacent cells from Definition \ref{def-adjacent}, and from now on we call any sequence $ \{R_1(ι_j,τ_j)\}_{j\in\N} $ inside $\mathbb{G}^d\times\Z$ of adjacent cells a \emph{path}.

\begin{deff}\label{def-LipschitzSurf}
	A \emph{\Lip{}} $ F $ is a set of cells in
	$ \mathbb{G}^d\times \Z $
	such that the following property is fulfilled: any path starting in any cell $v\in L_1$, and such that $ d(R_1(ι_j,τ_j), L_0)\to∞ $ as $j \to \infty$, intersects $ F $.
\end{deff}

Definition \ref{def-LipschitzSurf} is stable under taking unions, and in particular the entire graph $\graph$ seen as a union of all cells satisfies the definition. To prevent such undesired examples, we introduce the following condition.
\begin{deff}
	We say that a \Lip{} $F$ is \emph{minimal} if, for each $F'\subset F$ we have that $F'$ is not a \Lip{}.
\end{deff}

\begin{remark}\label{rmk-LipNoLip}
	The minimal \Lip{} we will end up constructing is the analogue for fractal graphs of the ``Lipschitz surface'' in the lattice settings of $ \Z^d $, see \cite{GrimmettLip1,Drewitz,Pete19}. There, a Lipschitz surface is $ * $-connected, or equivalently, for any point $ (b,0) $ in the base of $ \Z^d $ one finds the corresponding height $ h=F(b) $ of the Lipschitz surface, which satisfies a Lipschitz condition of type $ |F(b_2)-F(b_1)|\le 1  $ whenever $ \|b_2-b_1 \|_1\le 1.$
	For the geometry of the fractal, we cannot hope for such strong connectivity properties. Seeing the fractal graph as a subset of the triangular lattice, we could start by defining the height $h$ as one of the dimensions of the lattice and the base $b$ as the remaing $d-1$ dimensions spanning $L_0$ together with the time dimension; in this case however, not every cell $(b,h)$ (written in this base-height notation) in the triangular lattice would belong to the fractal graph $\mathbb{G}^d$, since it may lie in one of the ``holes'' of the fractal graph. In particular we cannot require for any $b_1,b_2 \in L_0$  such that $\|b_2-b_1\|_1 \le 1,$ that the corresponding $h_1$ and $h_2$ satisfy $|h_2-h_1|\le 1$, as either could be ill defined.  However, the key property which remains true is that an appropriately\footnote{See Proposition \ref{prop-RadHill} and Corollary \ref{corol-LipschitzInTime}.} constructed minimal Lipschitz cutset $ F $ separates the origin $ (0,0)\in\graph $ from infinity in the sense of Definition \ref{def-LipschitzSurf} in the fashion of a cutset and it retains some mild Lipschitz continuity properties, so we opted to use the name \emph{\Lip{}} instead.
\end{remark}

We can now state our first technical result, and for this purpose, recall the constant $\Cr{conf1}$ from Lemma \ref{lemma-confining}.
\begin{thm}\label{thm-main}
	Let $ \mathbb{G}^d $ be the $ d- $dimensional \Sier{} gasket with conductances satisfying \eqref{def-conductances}. 
	Let $\ell\in \N$ and let $β\in\N$ be large enough. Furthermore, let $η\in\N$, $ ε\in(0,1)$ and $ζ \in (0,\infty)$ such that 
	\begin{equation*}
		ζ \ge	\frac{1}{\ell}		\Big(
            \big(   \Cr{conf1}\log \big( \tfrac{8\Cr{conf1}}{3ε} \big) \big)^{\dw-1}	ηβ	
            \Big)^\frac{1}{\dw}
            ,
	\end{equation*}
	and tessellate $ \graph $ into space-time cells as described above. Let $ E:=E(ι,τ) $ be an increasing event restricted to the super cell $R_1^η(ι,τ)$ whose associated probability $ν_E\big((1-ε)μ, S_1^η(ι,τ), \displ_{ζ\ell}, ηβ\big)$ (cf.\ Definition \ref{def-E}) has a uniform lower bound across all $(ι,τ)\in \mathbb{B}^d\times \Z$ denoted by
	\begin{equation*}
		ν_E\big((1-ε)μ, S_1^η, \displ_{ζ\ell}, ηβ\big).
	\end{equation*}
	Then there exists $ α_0 \in (0,\infty)$ such that if
	\begin{equation*}
		ψ_1 (ε,μ_0,\ell,η) :=	\min \Big\{	\frac{ε^2μ_0 2^{\dv \ell}	} {C_λ}				,		-\log \Big(	1- ν_E\big( (1-ε)λ, S_1^η, \displ_{ζ\ell}, ηβ	 \big)
		\Big)	\Big\}  \ge α_0,
	\end{equation*}
	there exists almost surely a minimal \Lip{} $ F $ with the property that  $ E(ι,τ) $ occurs for all $(ι,τ)$ such that $R_1 (ι,τ)\in F $.
\end{thm}

\begin{remark}\label{rem-order}
    Before proceeding, we quickly outline how the conditions of Theorem \ref{thm-main} (and also the following Theorem \ref{thm-main2}) can be established. One usually fixes the ratio $2^\ell/\beta$  to be an arbitrary, but large constant. Once $2^\ell/\beta$ is fixed, we fix $\varepsilon\in(0,1)$ according to how much of the initial Poisson point process intensity we require in order to make the event $E$ still serve our goal. More precisely, we want that $E$ still lets us draw desired conclusions under the intensity $(1-\varepsilon)\mu$ which is slightly smaller than the actual intensity of the particles. This slack is needed to restrict our attention to the particles that are ``behaving well''. When dealing with events that become increasingly probable at larger scales, $\varepsilon$ can usually be chosen arbitrarily, but it is good to think of it as small. Next, the value $\eta$ determines the size of the super-tiles, which controls how much overlap we need and allow between the cells of the tessellation (this is usually done in order to enable information to propagate from cell to cell).
    The lower bound on $\zeta$ is to guarantee that, as particles move in $S_1^η$ for time $\beta$, with high probability they do not travel away from their starting position further than $\zeta\ell$, which gives us better control of dependencies between neighbouring cells of the tessellation.
    
    With $2^\ell/\beta$, $\varepsilon$ and $\zeta$ fixed, there exists a constant $\alpha_0>0$ for which the statement of the theorem holds. This constant is defined purely implicitly and ensures various expressions in the proof remain sufficiently small throughout the calculations.
    We now want to satisfy the condition that $\psi_1(\varepsilon,\mu_0,\ell,\eta)\geq\alpha_0$. This first means that $\frac{ε^2μ_0 2^{\dv \ell}	} {C_λ}\geq\alpha_0$. This can be satisfied by either making $\ell$ large enough (thus making the tessellation very coarse) or by assuming that the intensity of the particles is large enough by making $\mu_0$ large. Next, we need to satisfy $ν_E\big( (1-ε)λ, S_1^η, \displ_{ζ\ell}, ηβ	 \big)\geq 1-\exp\{-\alpha_0\}$. Usually, the event $E$ is a local event that becomes more likely as $\ell$ is made larger, so setting $\ell$ large satisfies this condition as well.
\end{remark}

\vspace{12 pt}

We can prove a further property of the \Lip{}, which gives us control on the distance of $F$ from any cell $R_1(\iota,\tau) \in L_1 $, without loss of generality  and in particular from $R_1(0,0)$, the cell containing the origin: for a fixed radius $r$ we investigate if the \Lip{} $F$ at distance $r$ surrounds the origin. More precisely, for a \Lip{} $F$ and $r>0$, we say that the event $ S(F,r) $ holds if any path $\{v_j\}_{j=1}^{n}$ of adjacent cells from $R_1(0,0)$ with $d\big(v_n,R_1(0,0)\big)> r$ intersects $ F $.
Note that this event is considerably more restrictive  than the one in Definition \ref{def-LipschitzSurf}: if $S(F,r)$ holds, it implies in particular that the \Lip{} does not only have finite distance from $L_0$, but essentially ``surrounds'' the cell $R_1(0,0)$ and prevents paths from obtaining arbitrary lengths while keeping their distance to $L_0$ small.

\begin{thm}\label{thm-main2}
	Under the conditions of Theorem \ref{thm-main}, let $F$ be the \Lip{} from Theorem \ref{thm-main} on which, in particular, the event $E$ holds. Then for each
  $c_s \in (0,\tfrac{\dv}{\dv+1}-\frac{1}{2})$  there exists $ \Cl[s]{thm2a}>0 $ such that for $ r_0 $ large enough we have
	\begin{equation*}
		\mathbb{P} \big( S(F,r_0)^{\mathsf{c}} \big)		\le \sum_{r\ge r_0}  r^{\dv+1} \exp\big\{		-\Cr{thm2a}  r^{c_s} \big\}.
	\end{equation*}
\end{thm}
The theorem  in particular entails that the \Lip{} surrounds $R_1(0,0)$ at an almost surely finite distance. 

\paragraph{Strategy of the proofs.}
The existence of the \Lip{} of good cells constructed in Theorem \ref{thm-main} is essentially equivalent to  all paths starting in $L_1$ of bad cells having only finite lengths. However, simply estimating the number and the probability of bad paths does not work;
even in the simplest case where $\eta=0$ (i.e.\ the super-cells would be just the cells themselves and therefore non-intersecting), two events $E(ι,τ)$ and $E(ι',τ')$ can be heavily correlated whenever $τ\neq τ'$ and especially if the two tiles corresponding to $ι'$ and $ι$   are close to each other. As an example, knowing that there were no particles present in the tile $ι$ during the time interval $τ$ increases the probability that all spatially close tiles will have fewer than expected particles for some time to come.
On the other hand, as long as the occurrence of $E(ι,τ)$ depends principally on the particle system behaving ``typically'', it  becomes more probable that the event will occur if the cells are all made bigger. Just blowing everything up is not enough however, since this would not resolve the correlation and combinatorial issues, so we adopt a multi-scale argument. For each scale we estimate the probability of a cell of that scale to be ``multi-scale bad'', knowing that at a larger scale the particles were behaving typically up until shortly before; this property is defined precisely in \eqref{def-multiscalebad}. For a given time horizon we choose a maximal scale $κ$, the largest scale that we will consider, and show that the probability to be ``multi-scale good'' is exponentially close to 1 at this large scale $κ$ and consequentially, as long as there are only sub-exponentially many cells of scale $κ$ within the space-time region we consider, we have that at this largest scale, all cells are ``multi-scale good'' with arbitrarily large probability.
By partitioning space-time into cells of ever smaller scale until reaching scale one, this gives rise to a dependent space-time fractal percolation problem on which we want to count the number of paths of bad cells. Using the fractal percolation nature of the setup and the alluded property that large cells are much less likely to be bad than even all of their ``descendant'' cells being bad at once, we consider paths of bad cells across multiple scales. This makes the combinatorial arguments more involved, but provides much better bounds on the probabilities of individual paths existing.
We also consider only the most important (i.e. largest in their part of the path) cells of a path and use a decoupling result to decouple the remaining space-time cells of a path. Then, using a clever union bound for the probability of finding a path of cells of various scales yields the result.

\section{Constructing the \Lip{}}\label{SEC-LipSurf}

Recall the definitions of adjacent cells from Definition \ref{def-adjacent}, of $ L_0 $ and $L_1$ from \eqref{def-L0}, \eqref{def-L1}, and of bad cells at the very start of Subsection \ref{subsec-mainres}, where we considered a cell $R_1(ι,τ)$ \emph{bad} if a certain increasing event $E(ι,τ)$ does not occur. To construct the \Lip{} we will make use of the concept of d-paths of cells, hills and mountains which we now define.
\begin{deff}[d-path]\label{def-dpath}
	A \emph{d-path} in $ \graph$ is a (possibly finite) sequence $\{u_k\}_{k\in\N}$ of adjacent cells starting with a bad cell $u_0\in L_1$ such that for each $k \in \N$ one of the following holds:
	\begin{itemize}
		\item $u_{k+1}$ is bad and $d(L_0,u_{k+1})\ge d(L_0,u_{k})$  \quad(\emph{increasing move});
		\item  $d(L_0,u_{k+1})< d(L_0,u_{k})$ \quad(\emph{diagonal move}).
	\end{itemize}
\end{deff}

A d-path is defined in a way that it can increase or maintain the distance to the base $L_0$ only by moving to a bad cell in the next step, and otherwise can go ``down'' towards $L_0$ with the so-called \emph{diagonal move}, independently of the state of the cell it is moving to.

\begin{remark}
	We kept the name \emph{diagonal move} as in the lattice setting of \cite{Pete19} for consistency and in order to distinguish a connection in the path that can only go toward $L_0$ regardless of the state of the cell. Furthermore, in the carpet setting it will revert to a $\ast-$neighbours connection (cf.\ Definition \ref{def-dpathCarpet}), thus rendering the term \emph{diagonal} more meaningful.
\end{remark}
To describe the set of cells which can be reached via d-paths we introduce hills and mountains.

\begin{deff}[Hill and Mountain]\label{def-HillMountain}
	For any two cells $u,v\subseteq\graph$, we write $ u\to v $ if $u$ is a \emph{bad} cell and there is a d-path from $ u $ to $ v. $ For a cell $ u \in L_1$ define the hill $\Hill_u$ and mountain $\Moun_u$ around $ u \in L_1$ as
	\begin{align*}
		\Hill_u:= \bigcup_{v\colon 
			u\to v} \{v\}
		\quad \text{ and } \quad
		\Moun_u:=\bigcup_{v\in L_1\colon u\in \Hill_v}\Hill_v,
	\end{align*}
	with the convention that if $u$ is good, then the  hill $H_u$ is defined to be the empty set.
\end{deff}
For a set of cells $S$, i.e.\ of the form $S=\bigcup_{i\in I} \{R_1(ι_i,τ_i)\}$  for some index set $I$, define for $ u\in S $ the sets
\begin{equation*}
	\rad_u(S):= \sup\{d(u,v)\colon v\in S\}
\end{equation*}
and
\[ ∂_{{\rm ext}}S:=\bigcup_{\substack{
			u\in S^{\mathsf{c}}     \, : \,
			\exists v\in S           \\
			v \text{ adjacent to } u}}
	\{u\},\]
where $S\compl$ is the set of all cells not belonging to $S$.
We then obtain the following result.
\begin{prop}\label{prop-RadHill}
	If for all $ u\in L_1$ we have
	\begin{equation} \label{eq:cutsetAssumption}
		\sum_{r\ge 1} r^{\dv+1} \P(\rad_u(\Hill_u) > r) <∞,
	\end{equation}
	then the set
	$$F:=∂_{{\rm ext}} \Big( \bigcup_{u\in L_1} \Moun_u \Big)
		\quad \cup \quad
		L_1\setminus (\cup_{u\in L_1} \Moun_u)$$
	is $\P$-\as{} within a finite distance from $L_0$ (which is in fact equivalent to $F$ being non-empty), is a \Lip{} and all cells $u\subseteq F$ are good.
\end{prop}

\begin{proof}

	If $L_1\setminus (\cup_{u\in L_1} \Moun_u) \ne \emptyset,$ then it is trivially within finite distance from $L_0$
	and  the cells contained in it 
	are good since they would otherwise be contained in some hills and therefore not in $L_1\setminus (\cup_{u\in L_1} \Moun_u)$.

	Next, we prove that cells in $ ∂_{{\rm ext}} ( \bigcup_{u\in L_1} \Moun_u ) $ are good. Suppose by contradiction that for some $u\in L_1$, a cell $ v\in ∂_{{\rm ext}}\Moun_u $ is bad. By definition of $∂_{{\rm ext}}\Moun_u $ there exists a cell $ v' \in \Moun_u$ adjacent to $ v, $ and $v'$ can be reached by some d-path since it lies in the mountain $\Moun_u$. If $d(L_0,v)\ge d(L_0,v')$, since $ v $ is bad, the d-path reaching $ v' $ can be extended to $ v $ with an \emph{increasing move}. Otherwise, if $d(L_0,v)<d(L_0,v')$, $ v $ can be reached by a \emph{diagonal move} from $v'$ (independently of the state of $v$), and in both cases therefore $ v\notin ∂_{{\rm ext}}\Moun_u. $

	To prove that $∂_{{\rm ext}} ( \bigcup_{u\in L_1} \Moun_u )$ is within a finite distance from $L_0$,  it is sufficient to show that for any cell $ u\in L_1 $ we have $ \rad_u(M_u)<∞,$ since, by construction of mountains with the diagonal moves, if the radius of a mountain was infinite, then it would be infinite for all mountains. We can therefore upper bound

	\begin{align*}
		\P(\rad_u(\Moun_u)>r) \le & \sum_{v\in L_1}\P\big(u\in \Hill_v,\ \rad_v (\Hill_v)> r-d(u,v)\big)                                               \\
		=                        & \sum_{\substack{v\in L_1 \, :                                                                                     \\d(u,v)\le r/2}}   \P\big(u\in \Hill_v,\ \rad_v (\Hill_v)> r-d(u,v)\big) +
		\sum_{\substack{v\in L_1 \, :                                                                                                                \\ d(u,v)\ge r/2}}\P(u\in \Hill_v).
  \end{align*}
We can upper bound the previous by
    \[
		                         \vol ({B}_{r/2}(u)\cap L_0) \P(\rad_v (\Hill_v)> r/2) +\sum_{s\ge r/2}\vol(∂({B}_s(u)\cap L_0))\P(\rad_v (\Hill_v)> s).
	\]
	Since by \eqref{def-dv} the volume of a ball in $\graph$ can be upper bounded by $\CVol r^{d+1},$ by the assumption in the proposition both summands tend to $ 0 $ as $ r $ increases.

	It remains to show that $F$ is a \Lip{}, i.e.\ it intersects any path $ \{u_j\}_{j\in\N} $ of cells starting from $L_1$ with $ d(u_j,L_0)\to ∞ $. Note that $L_1\setminus (\cup_{u\in L_1} \Moun_u)$ and a fortiori $F$ intersects any path that starts in a cell contained in $L_1\setminus (\cup_{u\in L_1} \Moun_u),$ so it remains to argue the case of paths that start in $L_1\cap (\cup_{u\in L_1} \Moun_u)$. The claim is a consequence of the definition of external boundary. Since $F$ is \as\ within finite distance from $L_0$, a path starting in a cell in $L_1$ and distance from $L_0$ going to infinity contains a cell $ u_j $ which is the first cell outside $\cup_{u\in L_1} \Moun_u. $ In particular, for some $ v\in L_1 $,  $ u_{j-1}\in \Moun_v$, $ u_j \notin \cup_{u\in L_1} \Moun_u$,  and $u_j \sim u_{j-1}$ so $u_j\in ∂_{{\rm ext}}  (\bigcup_{u\in L_1}\Moun_u),$ i.e.\ the path intersects the \Lip{} $ F. $
\end{proof}

Before turning to the multi-scale arguments, we prove a further property of the \Lip{}. We already highlighted in Remark \ref{rmk-LipNoLip} that on a fractal graph we cannot hope for a general Lipschitz condition. However, a Lipschitz connectivity property holds in the ``time dimension'' in the following sense.
\begin{corol}\label{corol-LipschitzInTime}
	Suppose that \eqref{eq:cutsetAssumption} is satisfied and let ${F} $ be as in Proposition \ref{prop-RadHill}. Consider
	\begin{align*}
		 & F^{\mathrm{o}}:= \bigcap_{\substack{ F' \subseteq F\, : \\ F' \text{ is a \Lip{}}    }}F'.
	\end{align*}
	Then $F^\mathrm{o}$ is a minimal \Lip{} and for all $ R_1(ι,τ)\in F^{\mathrm{o}}, $ there exist $ ι_{-1},ι_{+1} \in \mathbb{B}^d$ such that $ S_1(ι_{-1})  $ and $ S_1(ι_{+1}) $ are individually either adjacent or equal to $ S_1(ι) $, and such that
	\[  R_1(ι_{-1},τ-1), R_1(ι_{+1},τ+1) \in F^{\mathrm{o}}\]
\end{corol}

An example of cells of $F$ which were removed in the transition from $F$ to $F^\mathrm{o}$ is depicted in Figure \ref{fig-LipMinimal}. The Lipschitz continuity in the time dimension is illustrated in Figure \ref{fig-sierpTime}.

\begin{figure}[!h]
	\centering
	\begin{subfigure}[t]{0.48\linewidth}
		\includegraphics[width=\linewidth]{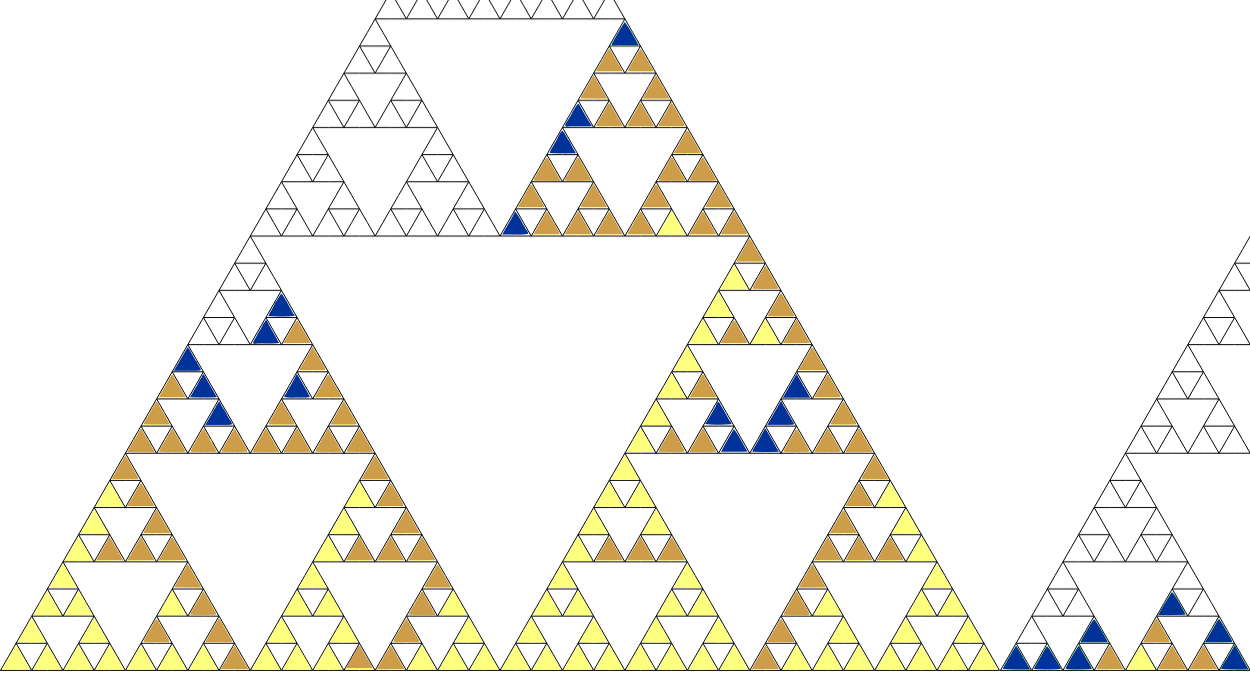}
		\subcaption{An illustration of possible mountains (in yellow) with bad cells highlighted with a darker tone. In dark blue the cells belonging to the \Lip{} $F$.}
	\end{subfigure}
	\hfill
	\begin{subfigure}[t]{0.48\linewidth}
		\includegraphics[width=\linewidth]{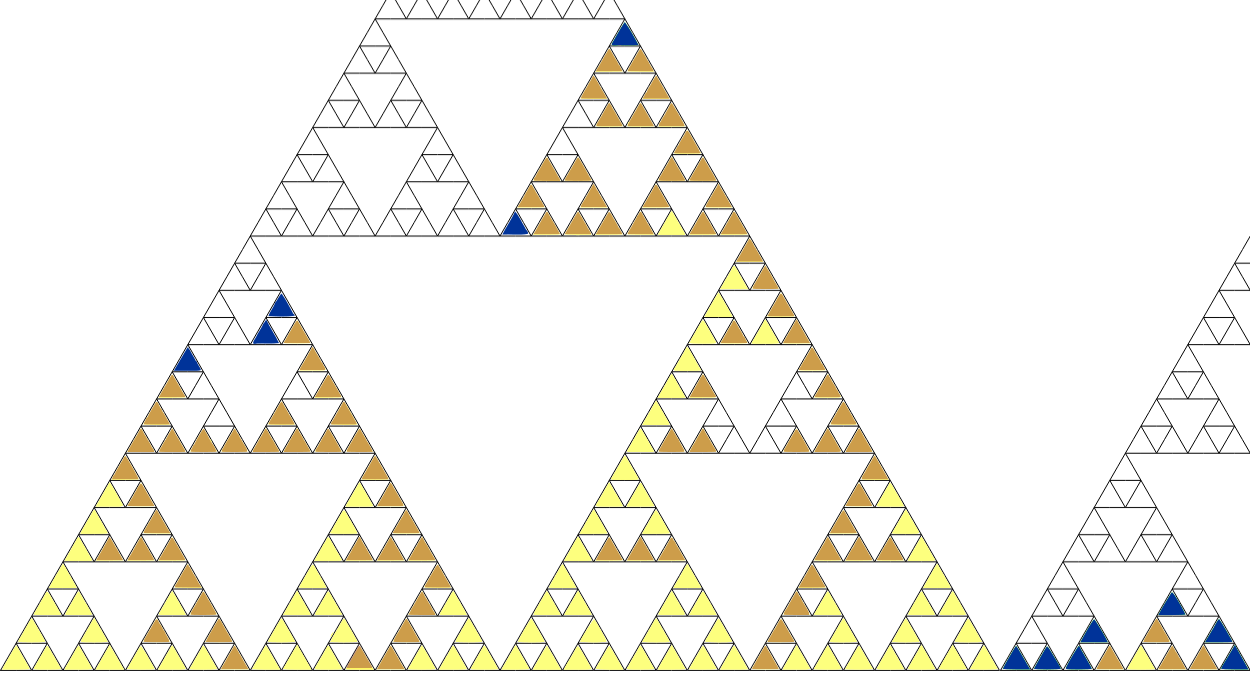}
		\subcaption{The resulting minimal \Lip{} $F^\mathrm{o}$ as obtained in Corollary \ref{corol-LipschitzInTime}. The removed cells are left blank as, even though they are good, we are ignoring this information.}%
	\end{subfigure}
	\caption{Constructing the minimal \Lip{}: a slab in $\mathbb{G}^2\times \{0\}$.}\label{fig-LipMinimal}
\end{figure}
\begin{figure}[!h]
	\centering
	\includegraphics[width=\linewidth]{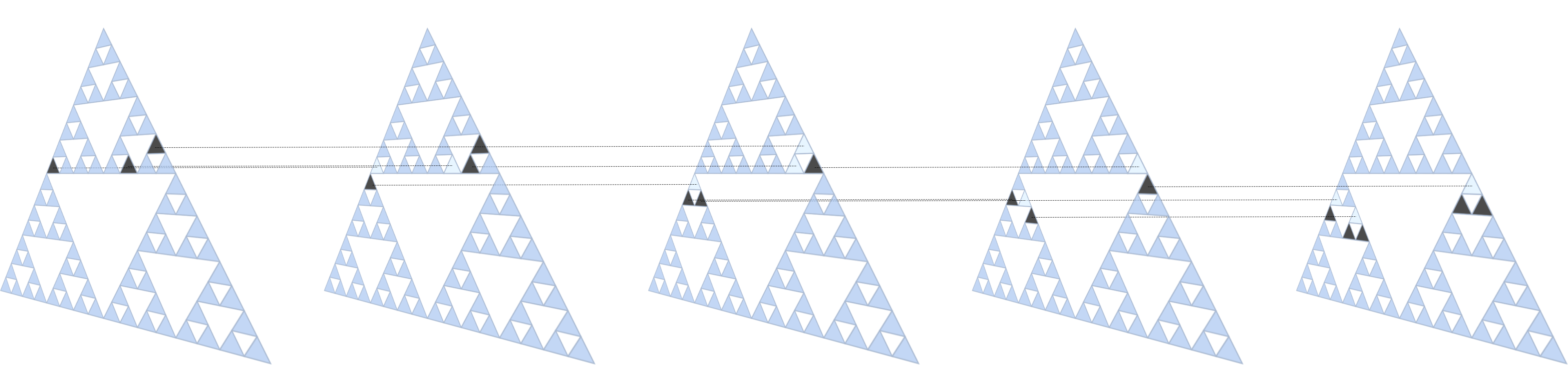}
	\caption{A possible evolution of the minimal \Lip{} $F^\mathrm{o}$ over 5 sequential time steps. Black tiles represent the cells of the minimal \Lip{} at the current time index $\tau$, the light blue tiles represent the cells of the minimal \Lip{} at the previous time index $\tau -1$. The two are connected with dashed lines to help visualise the relationship.}\label{fig-sierpTime}
\end{figure}

\begin{proof}
	$F^\mathrm{o}$ is a \Lip{} as a consequence of the definition of $F$ as we now argue. Let $π:=\{u_i\}_{i \in \N}$ be any path of cells with $u_1 \in L_1$ such that $d(u_i, L_0) \to \infty$ as $i \to \infty.$ We construct a path $π'$ with the help of $π$ as follows.
    Let $\pi_F$ be the sequence of (not necessarily unique) cells of the path $\pi$ that lie in $F$, ordered according to their appearance in $\pi$. We claim that this sequence is a.s.\  finite. Assume the converse. We can associate each cell $u$ of $\pi_F$ to a mountain $M_v$ for which $u\in\partial_{\text{ext}}(M_v)$ (when this choice is not unique, we can take for example the mountain for which $d(u,v)$ is smallest). Due to our assumption on $\pi$, and by extension $\pi_F$, this implies the a.s.\ existence of a sequence of mountains $(M_{v_i})_{i\in\N}$ with $\operatorname{rad}_{v_i}(M_{v_i})\to\infty$ as $i\to\infty$. Together with translation invariance this however contradicts \eqref{eq:cutsetAssumption} and so $\pi_F$ is by necessity a.s.\ finite.
    Let $u$ now be the last cell of $\pi_F$.
    Define now $π'$ to be the part of $π$ from the last visit of $u$ (including $u$) onward. By the definition of $F$ as external boundary of a union of mountains, we can extend $π'$ before $u$ by some arbitrary (finite) path of cells from $L_1$ to $u$ which does not intersect $F$:
	for example we can use a d-path that ends in a cell neighbouring $u$.
	Since any \Lip{} $ F'\subset F$ needs to intersect any such path and in particular $π'$ and $F'\subseteq F$ we have $u\in F'$, and thus $u\in F^{\mathrm{o}}$. Since $π$ was an arbitrary path starting in $L_1$ with $d(\pi_i, L_0) \to \infty$ as $i \to \infty$, we obtain that $F^\mathrm{o}$ is a \Lip{}.

	The minimality is straightforward due to the definition of $F^\mathrm{o}$ and it remains to show the temporal Lipschitz connectivity claim.

	For this purpose, let $ R_1(ι,τ)\in F^{\mathrm{o}}$ be arbitrary,
	and we show the claim only for $ ι_{+1} $ and $ τ+1, $ the other case being identical. Suppose that such $ ι_{+1} $ does not exists. We show now that it would be possible to construct a sequence $ \{u_j\}_j $ of adjacent cells which includes some of the cells in
	\[\overline{R_1(ι,τ)}:= \big\{R_1(\bar{ι}, \bar{τ})	\colon \bar{τ}\in \{τ,τ+1\}, \bar{ι}=ι \text{ or such that } S_1(ι)\text{ adjacent to } S_1(\bar{ι})\big\}\setminus\{R_1(\iota,\tau)\}, \]
	starts from $ L_1 $, with $d(u_j,L_0)\to ∞ $ and does not intersect $ F^{\mathrm{o}} $. Note that by our supposition, none of the cells in $\overline{R_1(ι,τ)}$ are in $F^{\mathrm{o}}$.

	We construct the sequence of adjacent cells $ \{u_j\}_j $ so that it starts from $L_1$ and reaches $\overline{R_1(ι,τ)} $ without intersecting $ F^{\mathrm{o}};$ note that this is possible  due to  the assumption of minimality of $ F^{\mathrm{o}} $. Similarly, the sequence $ \{u_j\}_j $ can be extended from $\overline{R_1(ι,τ)}$ without intersecting $ F^{\mathrm{o}} $ and with $d(u_j,L_0)\to ∞ $. Since all of the cells in $\overline{R_1(ι,τ)}$ are adjacent, the resulting sequence $ \{u_j\}_j $ contradicts the  definition of \Lip{}, and proves the claim.
\end{proof}

The next three sections are devoted to the multi-scale argument which will establish the assumption \eqref{eq:cutsetAssumption}.

\section{Decoupling Theorem}\label{SEC-Mixing}

We begin by proving that when \eqref{eq-PHI} holds, random walks started from vertices close to each other have similar probability distributions at sufficiently large times. More precisely, we have the following \emph{fluctuation bound} for caloric functions. Recall the definition of the weighted graph $ (\mathbb{G}^d,(λ_{x,y})_{x\sim y}) $ from Subsection \ref{sec:Sierp}.

\begin{prop}\label{prop-fluctuations}
	Let $x_0\in \mathbb{G}^d$ be arbitrary and suppose that \eqref{eq-PHI} holds with constant $\Cr{PHI}>1$. Let $\Theta:=\log_2(\Cr{PHI}/(\Cr{PHI}-1))>1$ and define for $x,y\in \mathbb{G}^d$
	\[
		ρ(x_0,x,y):=d(x_0,x)\vee d(x_0,y).
	\]
    Write $Q(x_0,R):=B_{2R}(x_0)\times(0,4R^{\dw}).$
	
 Then there exists a constant $\Cl{BH}>0$ such that the following holds: Let $r_0\geq 2$ and suppose that $u$ is caloric in $Q(x_0,r_0)$. Then, for any $x_1,x_2\in B_{r_0/2}(x_0)$ and any $t_1,t_2$ for which $r_0^{\dw}-\rho(x_0,x_1,x_2)^{\dw}\leq t_1,t_2\leq r_0^{\dw},$ we have that
	\begin{equation*}
		|u(x_1,t_1)-u(x_2,t_2)|\leq \Cr{BH}\left(\rho(x_0,x_1,x_2)/r_0\right)^\Theta \sup_{(t,x)\in Q_+(x_0,r_0)}|u(x,t)|,
	\end{equation*}
	where $Q_+(x_0,r_0):=B_{r_0}(x_0)\times[3r_0^{\dw},4r_0^{\dw}]$.
\end{prop}

The above inequality when applied to the heat kernel $u(x,t):=p_t(y,x)$ tells us that the fluctuations in $x$ and $t$ of the probability that a particle starting in $y$ is at $x$ after time $t$ can be precisely controlled and have an upper bound that is polynomial in $x$. We will use this bound below when comparing various heat kernel values with a representative one.

\begin{proof}
	In addition to $Q$ and $Q_+$, we define $Q_-(x_0,r_0):=B_{r_0}(x_0)\times[r_0^{\dw},2r_0^{\dw}]$. Next, define $r_k:=2^{-k}r_0$ and set
	\begin{align*}
		Q(k)   & :=4(r_0^{\dw}-r_k^{\dw})+Q(x_0,r_k),              \\
		Q_+(k) & :=4(r_0^{\dw}-r_k^{\dw})+Q_+(x_0,r_k),\text{ and} \\
		Q_-(k) & :=4(r_0^{\dw}-r_k^{\dw})+Q_-(x_0,r_k),
	\end{align*}
	where the summation is to be seen as a shift of the time interval of $Q$ (resp.\ $Q_+$ and $Q_-$).
	A quick calculation using that $\dw \geq 2$ then yields that
	$Q(k)\subset Q_+(k-1)$.
	Take now $k\geq 1$ small enough so that $r_k\geq 2$. We can without loss of generality consider the shifted set $Q(k)$ with the functions \(-u+\sup_{Q(k)}u\) and \(u-\inf_{Q(k)}u\). Indeed, note that under the change of time variable $\hat t:=t+4(r_0^{\dw}-r_k^{\dw})$, the function $\hat u(x,t):=u(x,\hat t)$ remains caloric. Since \eqref{eq-PHI} holds for any non-negative caloric function on $Q(x_0,r_k)$, it therefore holds for \(-\hat u+\sup_{Q(k)}u\) and \(\hat u-\inf_{Q(k)}u,\) and in particular also for \(-u+\sup_{Q(k)}u\) and \(u-\inf_{Q(k)}u\) on $Q(k)$.
	Applying \eqref{eq-PHI} to these two functions then yields the inequalities
	\[
		-\inf_{Q_-(k)}u + \sup_{Q(k)}u\leq \Cr{PHI}(-\sup_{Q_+(k)}u+\sup_{Q(k)}u)
	\]
	and
	\[
		\sup_{Q_-(k)}u-\inf_{Q(k)}u\leq \Cr{PHI}(\inf_{Q_+(k)}u-\inf_{Q(k)}u),
	\]
	respectively.
	Adding the two together and using that $\sup_{Q_-(k)}u-\inf_{Q_-(k)}u\geq 0$ then leads to
	\[
		\sup_{Q(k)}u-\inf_{Q(k)}u\leq \Cr{PHI}(\sup_{Q(k)}u-\inf_{Q(k)}u)-\Cr{PHI}(\sup_{Q_+(k)}u-\inf_{Q_+(k)}u).
	\]
	If we define now the \emph{oscillation} of $u$ inside $A$ as $\operatorname{Osc}(u,A):=\sup_A u-\inf_A u$ and set $\delta:=\Cr{PHI}^{-1} \in (0,\infty)$, we get
	\[
		\operatorname{Osc}(u,Q_+(k))\leq (1-\delta)\operatorname{Osc}(u,Q(k)).
	\]
	Take now the largest $m$ such that $r_m\geq \rho(x_0,x_1,x_2)$. Applying the above oscillation inequality on $Q(1)\supset Q(2)\supset\cdots\supset Q(m)$, we get since $(x_1,t_1), (x_2,t_2)\in Q(m)$ that
	\[
		|u(x_1,t_1)-u(x_2,t_2)|\leq \operatorname{Osc}(u,Q(m))\leq(1-\delta)^{m-1}\operatorname{Osc}(u,Q(1)).
	\]
	Using that $(1-\delta)^m=2^{-m\Theta}\leq \left(2\rho(x_0,x_1,x_2)/r_0\right)^\Theta$ we get the claim.
\end{proof}
Next, we state a result of Popov and Teixeira \cite{Popov2015}, which will let us couple the locations of our particle system after they have moved with an independent Poisson point process on $\mathbb{G}$.

\begin{prop}[Soft local times]\label{prop-softLocalTimes}
	Let $J$ be an at most countable index set and let \((Z_j)_{j\in J}\) be a collection of points distributed independently on $ \mathbb{G}^d $ according to a family of probabilities \(g_j:\mathbb{G}^d\rightarrow\mathbb{R}\), \(j\leq J\). Define for all \(y\in \mathbb{G}^d\) the soft local time function \(H_J(y)=\sum_{j=1}^{J}\xi_jg_j(y)\), where the \(\xi_j\) are i.i.d.\ exponential random variables of mean \(1\). Let \(\psi\) be a Poisson point process on $ \mathbb{G}^d $ with intensity measure \(\rho:\mathbb{G}^d\rightarrow\mathbb{R}\) and define the event
	\(
	E:=\left\{\psi\subseteq(Z_j)_{j\leq J}\right\},
	\)
    i.e.\ the particles belonging to $\psi$ are a subset of $(Z_j)_{j\leq J}$.
	Then there exists a coupling $\mathbb{Q}$ of \((Z_j)_{j\leq J}\) and \(\psi\), such that
	\[
		\mathbb{Q}\left(E\right)\geq\mathbb{Q}\left(H_J(y)\geq\rho(y),\;\forall y\in \mathbb{G}^d\right).
	\]
\end{prop}
\begin{proof}
	The coupling is introduced in \cite[Section 4]{Popov2015} and proven in \cite[Corollary 4.4]{Popov2015}. A reformulation of the construction for particles on a graph can be found in \cite[Appendix A]{Hilario2015}, and our claim corresponds to \cite[Corollary A.3]{Hilario2015}.
\end{proof}

Before stating the next result, it is useful to recall from Remark \ref{rem:regions} what we refer to as regions and sub-regions. Intuitively, one should think of a (sub-)region as either a large tile, or a ``ball-like'' union of tiles in $\mathbb{G}^d$.

\begin{prop}\label{prop-mixing-unbounded}
	Consider uniformly elliptic conductances $ λ_{x,y} $ satisfying \eqref{def-conductances} for some $ C_λ>0. $ For each $ \Cl[m]{Mixing1}>0 $ there exist constants $  \Cl[m]{MixingDelta}, \Cl[m]{MixingRoot}, \Cl[m]{MixingProb}  \in (0,\infty)$, and $Θ$ as in Proposition \ref{prop-fluctuations} such that the following holds.

	Fix large enough $\ell>0 $ and $ \bar{ε} >0$. Consider a region $ S_K $ of diameter $K\gg\ell$ tessellated into sub-regions $S_i^{\ell} $ of diameter at most $ \ell $ such that at time 0 there is a collection of particles (i.e., random walkers) where each sub-region $ S_i^\ell $ contains at least
	$ δ\sum_{y\in S_i^l} λ_y > \Cr{Mixing1} $ particles for some $ δ>0. $ Let $ \mathbf{Δ} , K' >0$ with 
	\begin{equation} \label{eq:Delta}
		\begin{split}\mathbf{Δ} & \ge \mathbf{Δ}_0:=\Cr{MixingDelta} \ell^{\dw} \bar{ε}^{\,-\tfrac{4}{Θ}} \\
		K-K'       & \ge \Cr{MixingRoot} (\mathbf{Δ})^{\frac{1}{\dw}},
        \end{split}
	\end{equation}
	and for $j\in J$, denote by $ Y_j $ the location of the $ j-$th  particle at time $ \mathbf{Δ} $, where $J$ is the index set of all particles that are inside the sub-region $S_{K'}\subset S_K$ of diameter $K'$ at time $\mathbf{\Delta}$, where $S_{K'}$ has Hausdorff distance at least $\frac{K-K'}{2}$ from the complement of $S_K$.

	Then, if $K$ is large enough for \eqref{eq:Delta} to be satisfied, there exists a coupling $ \mathbb{Q} $ of a Poisson point process $ Ξ $ with intensity measure $δ(1-\bar{ε})λ_y, $  $ y\in S_{K'} ,$ and $(Y_j)_{j\in J}$, such that
	\begin{equation*}
		\mathbb{Q}\big(Ξ\subseteq (Y_j)_{j\in J}  \big)
		\ge
		1-\sum_{y\in S_{K'}}e^{-\Cr{MixingProb} δλ_y \bar{ε}^2	\mathbf{Δ}^{\frac{\dv}{\dw}}	}.
	\end{equation*}
\end{prop}

\begin{remark}\label{rem-Ksize}
    In the above proposition, we require both $\ell$ and $K$ to be large. More precisely, $K$ should be sufficiently larger than $\ell$ so that $K-K'$ (which is at least of order $\ell$) remains small in comparison to $K$.
\end{remark}

\begin{proof}
	Using Proposition \ref{prop-softLocalTimes}, we can deduce that there exists a coupling $\mathbb{Q}$ of an independent Poisson point process $\Psi$ on $\mathbb{G}$ with intensity measure $\zeta(y)=δ(1-\bar{ε})\lambda_y$ and the locations of the particles $Y_j$ after they have moved for time $\mathbf{\Delta}$, which are distributed according to the density functions $f_{\mathbf{\Delta}}(x_j,y):=p_{\mathbf{\Delta}}(x_j,y)\lambda_y$, where $x_j$ is position of $Y_j$ at time $0$, such that the particles belonging to $\Psi$ are a subset of $(Y_j)_{j\in J}$ with probability at least
	\[
		\mathbb{Q}\big(H_J(y)\geq \delta\lambda_y(1-\bar{ε}),\,\forall y\in S_{K'}\big),
	\]
	where $H_J(y)=\sum_{j\in J}\xi_j f_{\mathbf{\Delta}}(x_j,y)$, the $(\xi_j)_{j\in J}$ are i.i.d.\ exponential random variables with parameter 1, and $J$ is the index set of particles inside $S_{K'}$ at time $ Δ $.

	We first observe that for any $\Gamma>0,$
 the probability of the complement is upper bounded via 
	\begin{align*}
		\mathbb{Q}\big(\exists y\in S_{K'}:\, H_J(y)<\delta\lambda_y(1-\bar{ε})\big)
		 & \leq \sum_{y\in S_{K'}}\mathbb{Q}	\big(H_J(y)<\delta\lambda_y(1-\bar{ε})\big)                     \\
		 & \leq \sum_{y\in S_{K'}}e^{\Gamma\lambda_y\delta(1-\bar{ε})}\mathbb{E}_{\mathbb{Q}}[\exp\{-\Gamma H_J(y)\}],
	\end{align*}
	 due to the exponential Chebychev inequality.

	Let now
	\begin{equation}\label{def-MixingThmR}
		R:=\Cr{MixingRoot}\mathbf{\Delta}^{1/\dw}\bar{ε}^{\,-\tfrac{\dw-1}{\dw}},
	\end{equation}
 where $\Cr{MixingRoot}$ will be chosen large enough later on.
	Next, let $J'$ be any subset of $J$  such that
	exactly $\lceil\sum_{y\in S_i^{\ell}}\delta\lambda_y\rceil$ many particles $Y_j$, $j \in J',$ were inside $S_i^{\ell}$ at time $0$ for every sub-region $S_i^{\ell}$ of $S_K$. 
	For $y\in \mathbb{G}$, define also $J'(y)\subseteq J'$ to be the set of all indices $j\in J'$ for which $d(x_j,y)\leq R$ and define $H^{J'}(y)$ as $H_J(y)$, but with the sum in the definition restricted to the indices $j\in J'(y)$, i.e.\ $H^{J'}(y)=\sum_{j\in J'(y)}\xi_j f_{\mathbf{\Delta}}(x_j,y)$. By definition, $H_J(y)\geq H^{J'}(y)$ and therefore
	\begin{align*}
		\mathbb{E}_{\mathbb{Q}}\big[\exp\{-\Gamma H_J(y)\} \big]\leq\mathbb{E}_{\mathbb{Q}} \big[\exp\{-\Gamma H^{J'}(y)\}\big].
	\end{align*}
	Since the $\xi_j$ in the definition of $H$ are independent exponential random variables of parameter 1, we can calculate further
	\begin{align*}
		\mathbb{E}_{\mathbb{Q}}\big[\exp\{-\Gamma H^{J'}(y)\}\big] & =\prod_{j\in J'(y)}\mathbb{E}_{\mathbb{Q}}\big[\exp\{-\Gamma\xi_j f_{\mathbf{\Delta}}(x_j,y)\} \big] \\
		                                          & =\prod_{j\in J'(y)}(1+\Gamma f_{\mathbf{\Delta}}(x_j,y))^{-1}.
	\end{align*}
	Furthermore, choosing $\Cr{MixingDelta}$ large enough, due to \eqref{eq-HKBound} we have for all $x$ with $d(x,y)\leq R$ that $p_{\mathbf{\Delta}}(x,y)\leq \Cl[s]{hkdiag}\mathbf{\Delta}^{-\dv/\dw}$ for some constant $\Cr{hkdiag}$. In particular, this holds for all $y\in S_{K'}$ and all $x\in \bigcup S_i^\ell$, where the union runs across all $S_i^\ell$ for which there exists $j\in J'(y)$ such that $x_j\in S_i^\ell$. Setting now $\Gamma:= \tfrac{1}{4\Cr{hkdiag}C_λ}\bar{ε} \mathbf{Δ}^{\dv/\dw}$ gives

	\begin{equation}\label{eq-MixingThm1}
		\sup_{x\in B_R(y)}\Gamma f_{\mathbf{\Delta}}(x,y)=\sup_{x\in B_R(y)}\Gamma\lambda_y p_{\mathbf{\Delta}}(x,y)\leq \Cr{hkdiag} C_λ \Gamma\mathbf{Δ}^{-\dv/\dw}<\bar{ε}/4.
	\end{equation}
	For this value of $\Gamma$ and using that for $|z|\leq \frac{1}{2},$ Taylor expansion yields that $\log(1+z)\geq z-z^2$, it further holds that
	\begin{align*}
		\prod_{j\in J'(y)}(1+\Gamma f_{\mathbf{Δ}}(x_j,y))^{-1} & \leq \prod_{j\in J'(y)}\exp\big\{-\Gamma f_{\mathbf{\Delta}}(x_j,y)(1-\Gamma f_{\mathbf{Δ}}(x_j,y))\big\}                          \\
		                                                   & \leq\exp\Big\{-
		\Big(1-\sup_{x\in B_R(y)}\Gamma f_{\mathbf{Δ}}(x,y)\Big)
		\sum_{j\in J'(y)}\Gamma f_{\mathbf{Δ}}(x_j,y)\Big\}                                                                                                                        \\
		                                                   & \leftstackrel{\eqref{eq-MixingThm1}}{\le} \exp\Big\{-\Gamma(1-\bar{ε}/4)\sum_{j\in J'(y)}f_{\mathbf{Δ}}(x_j,y)\Big\}.
	\end{align*}
	We claim now (and prove below) that
	\begin{equation}\label{eq-mixing-movetoy}
		\sum_{j\in J'(y)}f_{\mathbf{Δ}}(x_j,y)\geq \delta\lambda_y(1-\bar{ε}/2),
	\end{equation}
	which then gives us that
	\begin{align*}
		\mathbb{Q} \big(\exists y\in S_{K'}:\, H_J(y)<\delta\lambda_y(1-\bar{ε})\big) & \leq\exp \big \{\gamma\lambda_y\delta(1-\bar{ε})-\gamma(1-\bar{ε}/4)\delta\lambda_y(1-\bar{ε}/2)\big \} \\
		                                                                              & \leq\exp\{-\gamma\delta\lambda_y\bar{ε}/4\}.
	\end{align*}
	Plugging in the definition of $\gamma$ then yields the claim. We therefore proceed to prove \eqref{eq-mixing-movetoy}.

	Recall that a particle $Y_j$ has its initial location at time $0$ equal to $x_j$. For each $S_i^{\ell}$ and each particle $Y_j$ whose location at time 0 was $x_j\in S_i^{\ell}$, let $x_j'\in S_i^{\ell}$ be such that $f_{\mathbf{Δ}}(x_j',y)=\max_{w\in S_i^{\ell}}f_{\mathbf{Δ}}(w,y)$. Note that for two particles $Y_i,Y_j$ that were located inside $S_i^{\ell}$, $x_i'=x_j'$. This however does not hold for two particles if they did not lie in the same subregion at time 0. Next, we can bound
	\[
		\sum_{j\in J'(y)}f_{\mathbf{\Delta}}(x_j,y)\geq \sum_{j\in J'(y)}\left(f_{\mathbf{\Delta}}(x_j',y)-|f_{\mathbf{\Delta}}(x_j',y)-f_{\mathbf{\Delta}}(x_j,y)|\right).
	\]
	We will look at the first summand: for each $S_i^{\ell}$, it holds by our choice of $x_j'$ that 
	\begin{align*}
		\sum_{\substack{j\in J'(y)                         \\ x_j\in S_i^{\ell}}}f_{\mathbf{\Delta}}(x_j',y)&=\max_{w\in S_i^{\ell}}f_{\mathbf{Δ}}(w,y)\sum_{\substack{j\in J'(y)\\ x_j\in S_i^{\ell}}}1\\
		\intertext{which by definition of $J'$ can be lower bounded by }
		\max_{w\in S_i^{\ell}}f_{\mathbf{Δ}}(w,y)
		\Big\lceil \sum_{z\in S_i^{\ell}}\delta\lambda_z \Big\rceil
		 & \geq \sum_{z\in S_i^{\ell}}δλ_z f_{\mathbf{Δ}}(z,y).
	\end{align*}
	Set $R(y)$ to be the set of all sites $z$ of $S_K$ for which $d(z,y)\leq R$. Note that $R$ is always strictly positive since it (cf.\ \eqref{def-MixingThmR}) is proportional to $l$ and $\Cr{MixingRoot}$ is assumed to be large. Furthermore, note that if $z\in R(y)$ then for all particles $Y_j$ with their respective $x_j'=z$ and $j\in J'$ we have that $j\in J'(y)$. It also holds that $λ_z f_{\mathbf{Δ}}(z,y)=λ_y f_{\mathbf{Δ}}(y,z)$, which combined with the preceding calculation yields for each $S_i^{\ell}$ that
	\begin{align*}
		\sum_{j\in J'(y)}f_{\mathbf{\Delta}}(x_j',y) & \geq \sum_{z\in R(y)}\delta\lambda_z f_{\mathbf{\Delta}}(z,y)
		                                              =\delta\lambda_y\sum_{z\in R(y)} f_{\mathbf{\Delta}}(y,z)     
		                                              \geq \delta \lambda_y \P(\Conf(R,\mathbf{\Delta})).         
    \end{align*}                                            
		By Lemma \ref{lemma-confining} we have that there exists a constant $\Cr{conf1}$ so as to lower bound the previous expression by 
  \begin{align*}
		\delta\lambda_y \big(1-\Cr{conf1} e^{-\Cr{conf1}^{-1}\big( \frac{R^{\dw} }{\mathbf{\Delta}} \big)^{\frac{1}{\dw-1}}} \big)
		                                             & \geq \delta\lambda_y(1-\bar{ε}/4),
	\end{align*}
	where the last inequality holds by setting $R$ (cf.\ \eqref{def-MixingThmR}) through $\Cr{MixingRoot}$ large enough with respect to $\Cr{conf1}$.

	It remains to find an upper bound for the second summand $\sum_{j\in J'(y)}|f_{\mathbf{\Delta}}(x_j',y)-f_{\mathbf{\Delta}}(x_j,y)|$. Let $I$ be the set of all $i$ for which $S_i^{\ell}$ contains a location $x_j$ from the set $(x_j)_{j\in J'(y)}$. Then
	\begin{align}\begin{split}
		\sum_{j\in J'(y)}|f_{\mathbf{Δ}}(x_j',y)-f_{\mathbf{\Delta}}(x_j,y)|
		 & 
		=\sum_{i\in I}\sum_{\substack{j\in J'(y)          \\ x_j\in S_i^{\ell}}} |f_{\mathbf{\Delta}}(x_j',y)-f_{\mathbf{\Delta}}(x_j,y)|\\\label{eq-PHIform_pre}
		 & 
		=\lambda_y\sum_{i\in I}\sum_{\substack{j\in J'(y) \\ x_j\in S_i^{\ell}}}|p_{\mathbf{Δ}}(x_j',y)-p_{\mathbf{\Delta}}(x_j,y)|.
	\end{split}\end{align}
	Since for any $y \in \mathbb G^d,$ the heat kernel $\mathbb G^d \times \R \ni (x,t) \mapsto p_t(x,y)$  is caloric and the parabolic Harnack inequality is fulfilled on $\mathbb G^d,$ Proposition \ref{prop-fluctuations} with $r_0^{\dw}=\mathbf{\Delta}$ applies. This allows us to use the upper bound 
    \[
    |p_{\mathbf{Δ}}(x_j',y)-p_{\mathbf{\Delta}}(x_j,y)| \leq\frac{\Cr{BH} \ell^{Θ}}{\mathbf{\Delta}^{\Theta/\dw}}\sup_{(t,x)\in Q_+(x_j',\mathbf{\Delta}^{1/\dw})}|p_{t}(x,y)|.
    \]
    We can also take advantage of the upper heat kernel bound \eqref{eq-HKBound}, uniformly in $x_j'$, on the supremum term and upper bound it by $\Cl{HK}\mathbf{Δ}^{-\dv/\dw}$. Combined, \eqref{eq-PHIform_pre} is smaller than
	\begin{align}
		\begin{split}
			\lambda_y\sum_{i\in I}\sum_{\substack{j\in J'(y)\\ x_j\in S_i^{\ell}}}\frac{\Cr{BH} \ell^{Θ}}{\mathbf{\Delta}^{\Theta/\dw}}\Cr{HK}\mathbf{Δ}^{-\dv/\dw}
			&
			\leq \lambda_y\sum_{i\in I}\sum_{\substack{x\in S_i^{\ell}}}\frac{\Cr{BH} \delta\lambda_x l^{\Theta}}{\mathbf{\Delta}^{\Theta/\dw}}\Cr{HK}\mathbf{\Delta}^{-\dv/\dw}\\
			&\label{eq-PHIform}
			=\delta\lambda_y \Cr{BH}\Cr{HK}\sum_{i\in I}\sum_{\substack{x\in S_i^{\ell}}}\lambda_x l^{\Theta}\mathbf{\Delta}^{-(\dv+\Theta)/\dw}\\
			&
			\leftstackrel{\eqref{def-dv}}{\le} \delta\lambda_y \Cr{BH}\Cr{HK}\CVol C_{\lambda}R^{\dv}l^{\Theta}\mathbf{\Delta}^{-(\dv+\Theta)/\dw}\\
			&
			\leq \delta\lambda_y \bar{ε}/4,
		\end{split}
	\end{align}
	where the last inequality follows from \eqref{def-MixingThmR} as well as \eqref{eq:Delta}, and by setting $\Cr{MixingRoot}$ sufficiently large with respect to the constants $\Cr{BH},\Cr{HK},\CVol$, and $C_{\lambda}$. This proves \eqref{eq-mixing-movetoy} and concludes the proof.
\end{proof}

The statement of Proposition \ref{prop-mixing-unbounded} does not depend on particles located outside of the region $S_K$ at time 0. However, since the particles can move in an unrestricted way, repeated applications of the theorem across multiple regions of time and space (cf.\ Sections \ref{subsec-firstlevel} and \ref{SEC-LipSurf}) still exhibit long range correlations that we would like to avoid. To that end, we will prove a version of Proposition \ref{prop-mixing-unbounded} also for particle systems conditioned on having the particle movement confined (cf.\ Lemma \ref{lemma-confining}). The main difficulty is that by conditioning the particles in this way, their transition probabilities do not necessarily satisfy \eqref{eq-HKBound} and by extension \eqref{eq-PHI} any longer. It turns out, however, that these probabilities are still quantitatively the same under some mild modifications of the assumptions, which we prove in the following lemma.

\begin{lemma}
	\label{lemma-conditionedBound}
	Let $\lambda$ satisfy \eqref{def-conductances}. Then there exist constants $\Cl[s]{Cond1}$ and $\Cl[s]{Cond2}$  so that the following holds. Consider a region $S_\ell$ with $\ell>0$. Let $\mathbf{\Delta}>\Cr{Cond1} \ell^{\dw}$ and $\rho\geq \Cr{Cond2}(\mathbf{\Delta}\log_2^{\dw-1}(\mathbf{\Delta}))^{1/\dw}$. Denote by $Y$ a random walk on $\mathbb{G}^d$ conditioned on being confined to $\displ_{\rho/2}$ during the time interval $[0,\mathbf{\Delta}]$. Let $x,y\in S_\ell$ with $x$ being the starting point of the random walk, and define

	\[
		q(x,y):=P_x\big(Y_{\mathbf{\Delta}}=y\,|\, Y_t \text{ is confined to }\displ_{\rho/2}\text{ during }[0,\mathbf{\Delta}]\big).
	\]
	Then there exists a constant $C>2$ such that for $x,y,z\in S_l$ we have
	\[
		\left|\frac{q(x,y)}{\lambda_y}-\frac{q(z,y)}{\lambda_y}\right|\leq C \ell^{\Theta}\mathbf{\Delta}^{-(\dv+\Theta)/\dw}.
	\]
\end{lemma}

\begin{remark}
	It is important to note that the above bound is of the same form as the bound we used in \eqref{eq-PHIform} for the unconditioned random walk. Consequently, we will use this lemma to prove a conditioned version of Proposition \ref{prop-mixing-unbounded} without having to directly use \eqref{eq-PHI}, which as mentioned above might not necessarily hold in this case.
\end{remark}

\begin{proof}[Proof of Lemma \ref{lemma-conditionedBound}]
    Define the probability 
    \[
        p_x(\rho):=P_x\big(Y_t\in B_{\rho/2}(x) \text{ for all } t\in [0,\mathbf{Δ}]\big),
    \]
	that a random walk started at $x$ is confined to $B_{\rho/2}$ during $[0,\mathbf{\Delta}]$. Using Lemma \ref{lemma-confining}, we have that
	\begin{align} \label{eq:LBp}
		1-p_E(\rho)
		 & \leq \Cr{conf1} e^{-\Cr{conf1}^{-1}(\rho^{\dw}/\mathbf{\Delta})^{\frac{1}{\dw-1}}}.
	\end{align}
	Next, writing $h(x,y):=P_x(Y_{\mathbf{Δ}}=y\,|\,Y\text{ exits }B_{ρ/2}(x)\text{ during }[0,\mathbf{\Delta}])$ and $f_{\mathbf{\Delta}}(x,y)=P_x(Y_{\Delta}=y)$, we can write
	\[
		f_{\mathbf{\Delta}}(x,y)=q(x,y)p_E(\rho)+h(x,y)(1-p_E(\rho)).
	\]
	From this, we  immediately obtain the bound
	\[
		q(x,y)\leq f_{\mathbf{\Delta}}(x,y)\frac{1}{p_E(\rho)}.
	\]
	We can then further upper bound
	\begin{equation}\label{eq-conditionedBoundCalculation}\begin{split}
			\left|\frac{q(x,y)}{\lambda_y}-\frac{q(z,y)}{\lambda_y}\right|&=\mathds{1}_{\{q(x,y)>q(z,y)\}} \left(\frac{q(x,y)}{\lambda_y}-\frac{q(z,y)}{\lambda_y}\right)\\
			&\quad+\mathds{1}_{\{q(x,y)\leq q(z,y)\}} \left(\frac{q(z,y)}{\lambda_y}-\frac{q(x,y)}{\lambda_y}\right)\\
			&\leq\mathds{1}_{\{q(x,y)>q(z,y)\}} \left(\frac{f_{\mathbf{\Delta}}(x,y)}{\lambda_y p_E(\rho)}-\frac{f_{\mathbf{\Delta}}(z,y)}{\lambda_y p_{E}(\rho)}+\frac{h(z,y)(1-p_E(\rho))}{p_E(\rho)\lambda_y}\right)\\
			&\quad+\mathds{1}_{\{q(x,y)\leq q(z,y)\}} \left(\frac{f_{\mathbf{\Delta}}(z,y)}{\lambda_y p_E(\rho)}-\frac{f_{\mathbf{\Delta}}(x,y)}{\lambda_y p_{E}(\rho)}+\frac{h(x,y)(1-p_E(\rho))}{p_E(\rho)\lambda_y}\right)\\
			&\leq \frac{|p_{\mathbf{\Delta}}(y,x)-p_{\mathbf{\Delta}}(y,z)|}{p_E(\rho)}+\frac{\max\{h(x,y),h(z,y)\}(1-p_E(\rho))}{p_E(\rho)\lambda_y}.
		\end{split}\end{equation}

	Next, observe that we can write $h(x,y)$ as $\mathbb{E}_x[f_{\mathbf{\Delta}-\tau}(w,y)\,|\,\tau<\mathbf{\Delta}]$ with $\tau$ being the first time $Y$ exits $B_{ρ/2}(x)$ and $w:=Y_{\tau}$ the random vertex at the boundary of $B(x,\rho/2)$ where $Y$ is at time $\tau$. Since the weights $\lambda_{x,y}$ satisfy \eqref{def-conductances} we can bound $\frac{f_{\mathbf{\Delta}-\tau}(w,y)}{\lambda_y}$ from above by some positive constant $\C$. This is because either $\Delta-\tau$ is larger than $d(w,y)$, which allows us to use \eqref{eq-HKBound}, or $\Delta-\tau$ is smaller than $d(w,y)$, so that $f_{\mathbf{\Delta}}(w,y)$ is bounded above by the probability that a random walk jumps at least $d(w,y)$ steps in time $\mathbf{\Delta}-\tau$, which is small enough since $d(w,y)$ is large. Therefore we have that $\frac{\max\{h(x,y),h(z,y)\}(1-p_E(\rho))}{p_E(\rho)\lambda_y}$ is at most $\Cl{max}$. This together with the bound on $1-p_E(\rho)$ yields
	\begin{align*}
		\frac{\max\{h(x,y),h(z,y)\}(1-p_E(\rho))}{p_E(\rho)\lambda_y} & \leq \frac{\Cr{max}\cdot \Cr{conf1}}{p_E(\rho)}\exp\{-\Cr{conf1}^{-1}(\rho^{\dw}/\mathbf{\Delta})^{\frac{1}{\dw-1}}\}                   \\
		                                                              & \leq \frac{\Cr{max}\cdot \Cr{conf1}}{p_E(\rho)}\exp\big \{-\Cr{conf1}^{-1}(\Cr{Cond2}^{\frac{1}{\dw-1}}\log_2(\mathbf{\Delta})) \big\}.
	\end{align*}

	We now return to \eqref{eq-conditionedBoundCalculation}. By setting $\Cr{Cond2}$ (and by extension $\rho$) large enough and using the upper bound for $1-p_E(\rho)$ from \eqref{eq:LBp}, we can bound $p_E(\rho)$ from below by $1/2$. Applying Proposition \ref{prop-fluctuations} to the term $|p_{\mathbf{\Delta}}(y,x)-p_{\mathbf{\Delta}}(y,z)|$, using \eqref{eq-HKBound} to bound the resulting supremum term, and finally setting $\Cr{Cond2}$ even larger if necessary for $\exp\{-\Cr{conf1}^{-1}(\Cr{Cond2}^{\frac{1}{\dw-1}}\log_2(\mathbf{\Delta}))\}$ to be smaller than $\mathbf{\Delta}^{-\dv/\dw}$ concludes the proof.
\end{proof}

We now state the  version of Proposition \ref{prop-mixing-unbounded} for particles that are confined. Note that the statement remains essentially unchanged, other than having a stronger condition on $K-K'$ than before. This is also the statement of the result that we will rely on to conduct our multi-scale analysis (cf. Lemma \ref{lemma-Mixing}).

\begin{thm}\label{thm-mixing}
	Consider elliptic conductances $ λ_{x,y} $ satisfying \eqref{def-conductances} for some $ C_λ>0. $ 
	For $Θ$ as in Proposition \ref{prop-fluctuations} and each $ \Cr{Mixing1}>0 $ there exist $  \Cr{MixingDelta}, \Cr{MixingRoot}, \Cr{MixingProb}$ such that the following holds.

	Fix large enough $\ell>0 $ and $ \bar{ε} >0$. Consider a region $ S_K $ of diameter $K\gg\ell$ tessellated into sub-regions $S_i^{\ell} $ of diameter at most $ \ell $ such that at time 0 there is a collection of particles, where each sub-region $ S_i^\ell $ contains at least
	$ δ\sum_{y\in S_{i}^\ell} λ_y > \Cr{Mixing1} $ particles for some $ δ>0. $ Let $ \mathbf{Δ} $ and $ K' >0$ be such that
	\begin{align*}
		\mathbf{Δ} & \ge \mathbf{Δ}_0:=\Cr{MixingDelta} \ell^{\dw} \bar{ε}^{\,-\tfrac{4}{Θ}}        \\
		K-K'       & \ge \Cr{MixingRoot}(\mathbf{Δ}(\log_2 \mathbf{Δ})^{\dw-1})^{\frac{1}{\dw}},
	\end{align*}
    and for $j\in J$, denote by $ Y_j $ the location of the $ j-$th  particle at time $ \mathbf{Δ} $ conditioned on being confined to a ball $ B_{ (K-K')}$  during $ [0,\mathbf{Δ}] $, where $J$ is the index set of all particles that are inside the sub-region $S_{K'}\subset S_K$ of diameter $K'$ at time $\mathbf{\Delta}$, where $S_{K'}$ has Hausdorff distance at least $\frac{K-K'}{2}$ from the complement of $S_K$.

	Then, if $K$ is large enough to satisfy \eqref{eq:Delta}, there exists a coupling $ \mathbb{Q} $ of a Poisson point process $ Ξ $ with intensity measure $δ(1-\bar{ε})λ_y $ for $ y\in S_{K'} $ and the family $ (Y_j)_{j\in J} $ such that
	\begin{equation*}
		\mathbb Q \big (
		Ξ\subseteq (Y_j)_j \big ) \ge 1-\sum_{y\in S_{K'}}e^{-\Cr{MixingProb} δλ_y \bar{ε}^2	\mathbf{Δ}^{\frac{\dv}{\dw}}	}.
	\end{equation*}
\end{thm}

\begin{remark}\label{rem-Ksize2}
    Similar to Proposition \ref{prop-mixing-unbounded}, we require both $\ell$ and $K$ to be large. Here, $K$ has to be sufficiently larger than $\ell$ so that $K-K'$ (which is grows polylogarithmically faster than $\ell$) remains small in comparison to $K$.
\end{remark}

\begin{proof}
	Using Lemma \ref{lemma-conditionedBound} and the upper bound on $g(x,y)$ from its proof when setting $\Gamma$, the proof proceeds the same as in Proposition \ref{prop-mixing-unbounded}. The independence from the graph outside of $S_{(2K-K')}$ follows from the fact that we consider only particles which are confined to $\displ_{(K-K')}$ (recall that confinement is defined with respect to the starting position of a particle) and ended in $S_K'$, so they never left $S_{(2K-K')}$ during $[0,\mathbf{\Delta}]$.
\end{proof}

\section{Multi-scale setup}\label{SEC-MultiScaleSetup}

In this section we define the multi-scale set-up for the construction. For some (large) $ κ\in \N $, we will define for each $1\le k \le κ$ cells at scale $ k $: in the fractal graph, spatial tiles will be denoted by $S_k(ι)$ and indexed by some $ι\in \mathbb{B}^d$; the time line $\R$ will be subdivided into intervals $T_k(τ)$ and indexed by $τ\in\Z.$ The space-time cells $R_k(ι,τ)$ will simply be the Cartesian product $S_k(ι) \times T_k(τ)$.
We will also need to introduce, for each scale $k$, extensions of the cells which do not need to be of the same scale. Those cells will be necessary to work with the dependencies between adjacent cells.
scale one will correspond to and agree with the first tessellation introduced in Definition \ref{def-Tessel1}. The value $κ$ instead is the largest scale that we will consider. The reader might want to think of $κ$ to be fixed for the moment. It will be determined later in the proof of Proposition \ref{prop-KnotinR}: if the largest area that we take into consideration is roughly $B_t(0)\times [-t,t]$, then we will set $κ=\mathcal{O}(\sqrt{\log(t)}).$

\subsection{Multi-scale tessellation}\label{SUBSEC-Tessel}
\paragraph{Space tessellation.}
We start by defining the space tessellation on the graph $\mathbb{G}^d$. After the full definition of all relevant tiles and intervals and a statement of useful properties, we refer to the end of this paragraph for a short motivation and intuition regarding the roles of the different tiles introduced here.

For  $ ε\in(0,1)$ and $ \ell, m, a  $ be positive (large) integers which we will fix later. Set $ \ell_0:= \ell - m $ and let
\begin{align}\label{def-ell}
	\begin{split}
		\ell_k&:= a(k-1)^2+m(k-1) + \ell.
	\end{split}
\end{align}
Define the space tiles at scale $ k\in \N$ indexed by $ι\in \mathbb{B}^d $ (cf.\ \eqref{def-binaryLattice}) as the subgraphs of $\mathbb{G}^d$ with vertex sets
\begin{equation*}
	S_k(ι) := ι2^{\ell_k} + \bigtriangleup^d_{\ell_k},
\end{equation*}
and induced edges,
which are well-defined in view of \eqref{eq-PropOfB}. We say that two cells $ S_k(ι_1)\neq S_k(ι_2) $ are \emph{adjacent} if $ d(S_k(ι_1),S_k(ι_2))=0 $.
It is easy to verify that
\begin{align}
	 & \text{$ S_k(ι) $ has side length of $ 2^{\ell_k},$ and that } \label{eq-SideLengthSk}                                                                       \\
	 & \text{$ S_{k+1}(ι) $  is the union of exactly $ 2^{\dv(\ell_{k+1}-\ell_{k})}=(d+1)^{2ak-a+m }$ tiles of scale $ k. $} \label{eq-EnumSmallerCells}
\end{align}
Next, we introduce a hierarchy of the space tiles. We define for $ k,j\ge 0 $ the function $\pi_k^{(j)}$ by
\begin{equation} \label{eq:well-def}
	π_{k}^{(j)}	(ι)= ι'		\qquad		\Leftrightarrow		\qquad		S_k(ι)	\subseteq S_{k+j}(ι'),
\end{equation}
and we say that $ S_{k'}(ι') $ is an ancestor of $ S_k(ι) $ (or equivalently that $ S_k(ι) $ is a descendant of $ S_{k'}(ι') $) if $ π_{k}^{(k'-k)} (ι)	= ι'$. Note the map is well-defined by the uniqueness of the choice of
	$ι'$ on the right-hand side of \eqref{eq:well-def}, and  that any cell is also a descendant and an ancestor of itself.

Recall now the constant $\Theta$ from Proposition \ref{prop-fluctuations}.
We define for $ k\ge 0 $ and $b(k):=ak^{2+ \frac{8}{Θ\dw}}m2^m$ the \emph{base}, the \emph{area of influence}, and for $ k\ge 1 $ the \emph{extension}, the \emph{support} and the \emph{extended support} as
\begin{align}
	S\base_k(ι)   & :=\hspace{2pt}\bigcup_{ι'\colon d(S_k(ι'),S_k(ι))\le b(k)}   S_k(ι'),                                           \label{def-Sbase} \\
	S\inbase_k(ι) & :=\bigcup_{ι'\colon d(S_k(ι'),S_k(ι))\le 2b(k)} S_k(ι'),                                         \label{def-Sinbase}              \\
	S\ext_k(ι)    & :=\hspace{18 pt}\bigcup_{ι'\colon π_{k-1}^{(1)}(ι')=ι} S\base_{k-1}(ι'),               \label{def-Sext}                           \\
	S\Sup_k(ι)    & :=\bigcup_{ι'\colon d(S_{k+1}(ι'),S_{k+1}(π^{(1)}_k(ι) )) \le m}   S_{k+1}(ι'),         \nonumber 
    \\
	S\Esup_k(ι)   & :=\bigcup_{ι'\colon d(S_{k+1}(ι'),S_{k+1}(π^{(1)}_k(ι) ))\le 3m+1}   S_{k+1}(ι') .\label{def-SEsup}
\end{align}
The choice of $b(k)$ will be made clear later  in \eqref{def-b(k)}. Recalling the value $η$ from Definition \ref{def-Tessel1η}, we also assume that $ b(1)\ge η $, which holds if we choose $ a $ large enough. See Figure \ref{fig-BaseExt} for an illustration of how the different tile extensions relate to each other.

\begin{figure}[!h]
	\centering
	\includegraphics[width=0.6\linewidth]{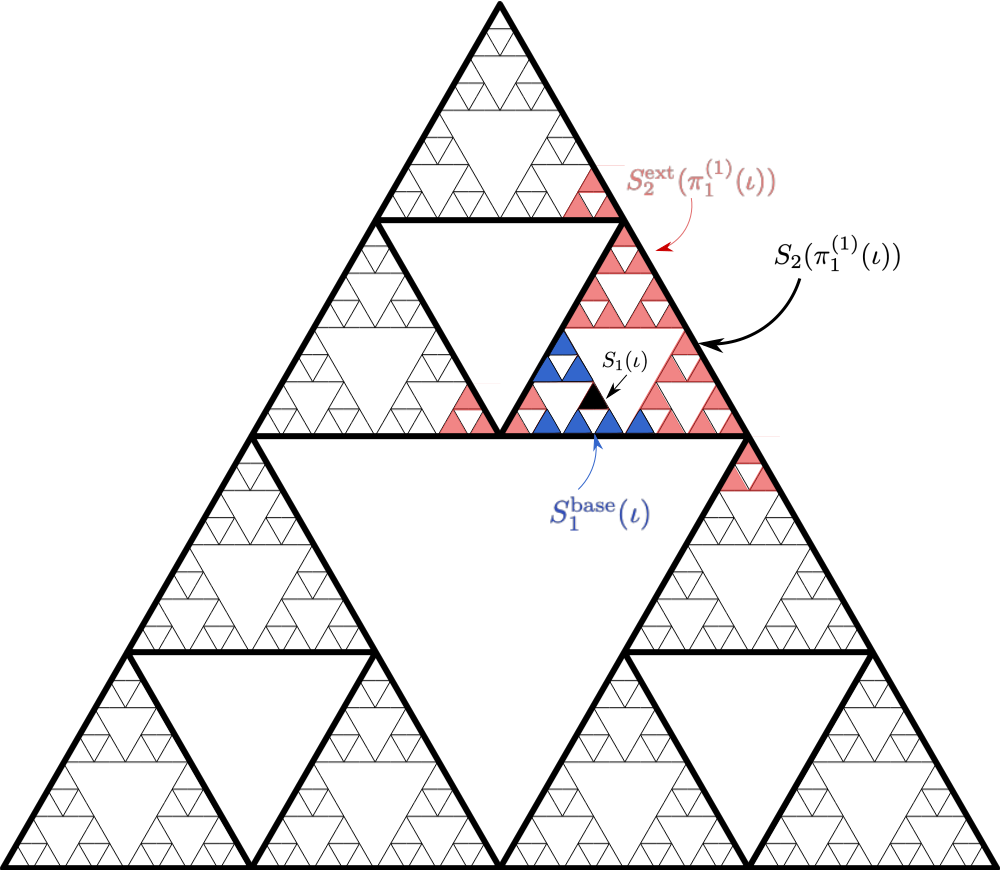}
	\caption{Illustration of $S_1\base(ι)$ and $S_2\ext(π^{(1)}_1(ι))$. The thin line triangles represent the many tiles $ S_1 $ of scale one, the thick black line triangles are tiles $ S_2 $ of scale 2. The black triangle represents the specific tile $S_1(\iota)$, while the dark blue region is $ S\base_1(ι) $ and the light red is $ S\ext_2(π^{(1)}_1(ι)) $.  $ S_1\inbase(ι) $ is not represented in order to keep the image legible.}\label{fig-BaseExt}
\end{figure}

We now state some properties of the above defined sets and the relations of the different tiles.
It is easy to check that for all $ (k,ι) \in \N_0\times \mathbb{B}^d$ it holds $ S_k(ι) 	\subseteq S\base_k(ι) \subseteq S\inbase_k(ι) $ and
\begin{equation}\label{eq-BaseInExt}
	S\base_k(ι) \subseteq S\ext_{k+1}(π^{(1)}_k(ι)).
\end{equation}
Since $ b(k) $ is increasing in $ k $, it also holds that
\begin{equation*}
	S\ext_{k}(ι)	\subseteq	S\base_k(ι).
\end{equation*}
Further simple properties of space tiles can easily be inferred: we will use later that
\begin{equation*}
	S_k\base(ι)		\text{ contains at most } \CVol b(k)^{\dv} \text{ tiles of scale } k;\text{ and}
\end{equation*}
\begin{equation*}
	S_k\ext(ι)		\text{ contains at most } \CVol(b(k-1)+2^{\ell_k})^{\dv} \text{ tiles of scale } k-1,
\end{equation*}
which both follow from \eqref{def-dv}.

We now look at the properties of the larger scales. Comparing the exponential growth of $ S_{k} $ in \eqref{eq-EnumSmallerCells} with the polynomial growth of $ b(k) $ in \eqref{def-Sinbase}, one sees that for $ a , m$ large enough, for all $ k $ and $ ι $, it holds that
\begin{equation}\label{eq-SInfInSup}
	S_k\inbase(ι)		\subseteq S_k\Sup(ι).
\end{equation}

\begin{remark}\label{rmk-SηSbase}
	The assumption $ b(k)\ge η $ implies that $ S_1\base(ι) $, and a fortiori $S_1\ext(ι) ,$ contains the \emph{super-tile} $ S_1^η(ι)$ defined in Definition \ref{def-Tessel1η}.
\end{remark}

We now quickly motivate the introduction of the different tiles. The tiles $S_k(ι)$ constitute the basic tiles at each scale. The introduction of the multi-scale argument suggests that we will introduce a notion of goodness for every scale $k$: this is related to $S_k\base(ι)$ and $S_{k+1}\ext(π_1^{(1)}(ι)),$ as well as to the events $D\base$ and $D\ext$ which we are going to define in \eqref{def-Dbase} and \eqref{def-Dext}.

Furthermore, $ S\inbase$, which is defined as $S\base$ but with a slightly larger border, will help us to keep tiles apart: if for two tiles the \emph{areas of influence} do not intersect, we will call these tiles \emph{well-separated} and we will be able to treat the tiles as essentially independent. Finally, we introduced $S\Sup$ and $S\Esup$ so that tiles whose (extended) supports intersect each other, even if otherwise well-separated, are still close enough to be part of a very general kind of path, the ScD-path (see Definition \ref{def-ScDpath}).

\paragraph{Temporal tessellation.}

We now turn to the temporal tessellation of $ \R. $ The tessellation itself is easier than the previous one introduced for space, and it corresponds to the one in \cite{Pete19}. Define for $k\ge 2$
\begin{align}
	β_k:=\Cmix ( \tfrac{k^2}{ε})^{\frac{4}{Θ}} 	\big( 2^{\ell_{k-1}} \big)^{\dw} 
	\label{def-β}
\end{align}
where $  \Cmix  $ is a constant larger than $  8^{4/Θ}\Cr{MixingDelta} $, $ Θ $ and $ \Cr{MixingDelta} $ are constants from Theorem \ref{thm-mixing} and $\varepsilon$ is from the beginning of Subsection \ref{SUBSEC-Tessel}. Set as well $β:=β_1:=\Cmix \tfrac{2^{\dw(\ell-m)}}{ε^{4/Θ}}$, assuming $m$ large enough so that  $  \Cmix  \ge  8^{4/Θ}\Cr{MixingDelta} $ still holds.
On first reading, one should not be distracted by the constant $ \Cmix $ or the fine-tuning power $ k^{8/Θ} $ in $ β_k $ and instead focus on the leading term $ 2^{\ell_{k-1}} $ which is raised to the power $ \dw $. As discussed before, the term $ \dw  $ represent the power scaling between time and space from the perspective of the random walkers. That is a major difference from the lattice $ \Z^d $ where the ``walk dimension'' $ \dw $ equals 2 for every dimension $ d $ of the lattice. Note in particular that ratios between two consecutive time-scales satisfy
\begin{equation}	\label{eq-rateBeta}
	\frac{β_{k+1}}{β_k}=(\tfrac{k+1}{k})^{8/Θ}  \big( 2^{2ak-3a+m} \big)^{\dw}.
\end{equation}
Define the time intervals at scale $ k\in \N$ as the intervals
\begin{equation*}
	T_k(τ) =[τβ_k, (τ+1)β_k), \quad  τ\in \Z,
\end{equation*}
and we say that two intervals $ T_k(τ_1)\neq T_k(τ_2) $ with $τ_1, τ_2 \in \Z$ are \emph{adjacent} if $ |τ_1-τ_2|\le 1.$
We now introduce a hierarchy over time, which is more complex than the spatial one. While for space, a parent contains its descendants, since ``time flows forward'', parents with respect to time will still have larger intervals than their children, but will lie to the left (i.e.\ ``before''): see Figure \ref{fig-TimeAncestors}. Formally, let $γ_k^{(0)}(τ)=τ, $ and for $ j\ge 1 $ define
\begin{equation*}
	γ_{k}^{(j)}(τ):=τ' \quad \text{ if } \quad 		  γ_{k}^{(j-1)}(τ) β_{k+j-1}\in T_{k+j}(τ'+1),
\end{equation*}
see Figure \ref{fig-TimeAncestors} for visualisation. In analogy with the terminology introduced in the  spatial setting, we say that $ T_{k'}(τ') $ is an ancestor of $ T_k(τ) $ or equivalently that $ T_k(τ) $ is a descendant of $ T_{k'}(τ') $ if  $ γ_{k}^{(k'-k)} (τ)	=τ'$ and it still holds that any time interval is also a descendant and an ancestor of itself. Note that due to the ``time drift'' it does not contain its own descendants of any scale as subintervals.

\begin{figure}[ht]
	\centering
	\includegraphics[width=0.8\linewidth]{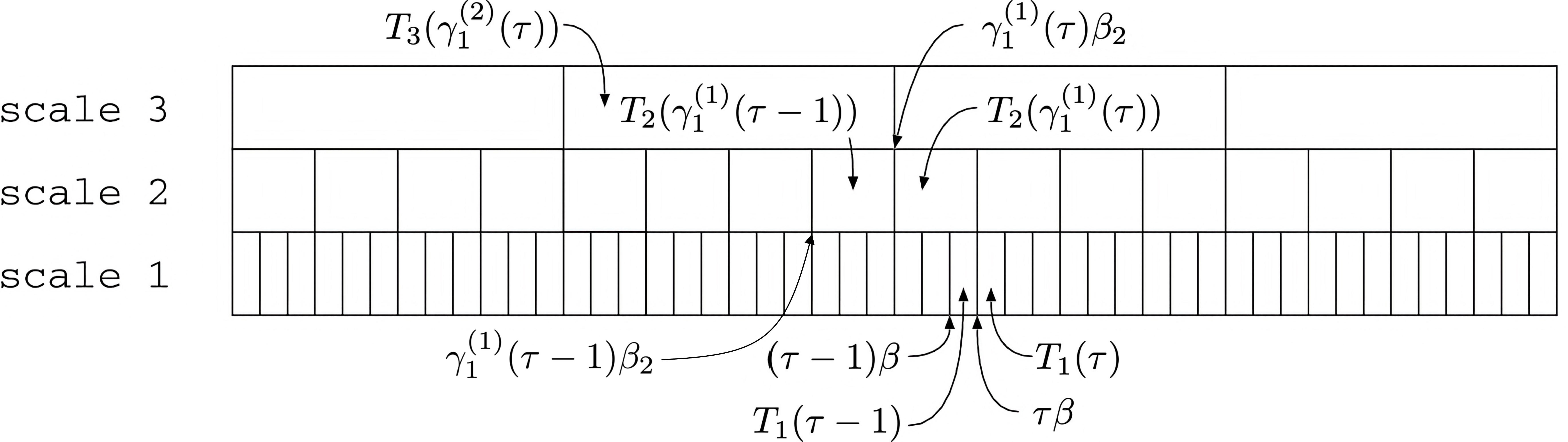}
	\caption{Temporal tessellation and its hierarchy structure. Image adapted from \cite{Pete19}.}\label{fig-TimeAncestors}
\end{figure}

As we did for space, we define for each scale $ k $ larger intervals that we will need:
\begin{align}
	T\inbase_1(τ) & :=[ γ_1^{(1)}(τ)  β_{2}		,		(τ+η\wedge2)β_{1} ], \nonumber                       \\
	T\inbase_k(τ) & :=[ γ_1^{(1)}(τ)  β_{2}		,		(τ+2)β_{k} ], \label{def-Tinbase}           \\
	T\Sup_k(τ)    & := \bigcup_{i=0}^{8} T_{k+1} (γ^{(1)}_{k}(τ)-3+i),		\label{def-Tsup}    \\
	T\Esup_k(τ)   & := \bigcup_{i=0}^{26}	T_{k+1} (γ^{(1)}_{k}(τ)-12+i).		\label{def-TEsup}
\end{align}

We now claim and prove that the time analogue of \eqref{eq-SInfInSup} still holds.
\begin{lemma}\label{lemma-TInfinSup}
	Let $ T_{k'}(τ') $ be a descendant of $ T_k(τ) ,$ and let $ T_{k'}(τ'') $ be adjacent to $ T_{k'}(τ') .$ Then for $a,m$ large enough
	\begin{equation*}
		T\inbase_{k'}(τ'')	\subseteq T_k\Sup(τ).
	\end{equation*}
\end{lemma}
\begin{proof}
	Recall that $ T\inbase_{k'}(τ'')		\subseteq [ γ_{k'}^{(1)}(τ'')  β_{k'+1}		,		(τ''+2\wedge η)β_{k'} ]  $, the definition of $T_{k}\Sup (τ) 		$%
	$	=((γ^{(1)}_{k}(τ)-3)β_{k+1},  (γ^{(1)}_{k}(τ)+5)β_{k+1})),$
	in \eqref{def-Tsup}, and $ |τ''-τ'|\le 1 $ by adjacency.

	It is easy to verify the inequality $ (γ^{(1)}_{k}(τ)-3)β_{k+1} \le γ^{(1)}_{k'}(τ'-1)β_{k'+1}$, so we concentrate on the right delimiters of the intervals. To prove the other inequality, note that for any interval $ T_{k'}(τ') $, we have $ τ'β_{k'} \le γ^{(1)}_{k'}(τ')β_{k'+1} + 2β_{k'+1}$ so iterating this $k-k'$ times we obtain
	\begin{equation*}
		τ'β_{k'} \le γ^{(k-k')}_{k'} (τ')β_{k} + 2\sum_{j=1}^{k-k'}β_{k'+j} .
	\end{equation*}
	We can bound using that $k'\ge 1$
	\begin{align*}
		\sum_{j=1}^{k-k'}β_{k'+j}
		\le \sum_{j=2}^{k}β_{j}
		= & \Cmix  \sum_{j=2}^{k}\Big(  \frac{j^2}{ε} \Big)^{4/Θ} 2^{\dw \ell_{j-1}}
		= \Cmix  ε^{-4/Θ} \sum_{j=2}^{k}  j^{8/Θ} 2^{\dw (a(j-2)^2+m(j-2)+\ell)}
		\intertext{which by induction is smaller than}
		  & \Cmix  ε^{-4/Θ} 2 k^{8/Θ} 2^{\dw (a(k-1)^2+m(k-1)+\ell)}
		= 2β_{k}.
	\end{align*}
	Hence, we have
	\begin{align*}
		(τ''+2\vee η)β_{k'} & \le  		(τ'+1+2\vee η)β_{k'}                                                 \\
		                    & \le  γ^{(k-k')}_{k'} (τ')β_{k} + 2\sum_{j=1}^{k}β_{k'+j} +(1+2\vee η)β_{k'} \\
		                    & \le τβ_k+4β_k+(1+2\vee η)β_{k'},
		\intertext{and since $ 4β_k+(1+2\vee η)β_{k'} \le (5+2\vee η)β_{k} \le β_{k+1}$ for $ a,m $ large enough, this is further smaller than}
		                    & τβ_{k} +β_{k+1} \le (γ^{(1)}_{k}(τ)+5)β_{k+1},
	\end{align*}
	proving the lemma.
\end{proof}
\paragraph{Space-time  tessellation.}

We can now define the space-time tessellation at different scales via the Cartesian products
\begin{align*}
	R_k(ι,τ)        & := S_k(ι)\times T_k(τ),               \\
	R\inbase_k(ι,τ) & := S\inbase_k(ι)\times T\inbase_k(τ), \\
	R\Sup_k(ι,τ)    & :=S_k\Sup(ι)\times T\Sup_k(τ),        \\
	R\Esup_k(ι,τ)   & :=S_k\Esup(ι)\times T\Esup_k(τ).
\end{align*}
We say two cells $ R_k(ι_1,τ_1)$ and $ R_k(ι_2,τ_2) $ of same scale are \emph{adjacent} if either $ d(S_k(ι_1),S_k(ι_2))=0$ and $τ_1=τ_2$, or else if $ι_1=ι_2$ and $|τ_1-τ_2|\le 1.$ We extend the mappings $π$ and $γ$ to a hierarchy of space-time cells.
We say that $R_k(ι,τ)$ is an ancestor of $R_{k'}(ι',τ')$ if $S_k(ι)$ is an ancestor of $S_{k'}(ι')$ and $T_k(τ)$ is an ancestor of $T_{k'}(τ')$.

We observe, combining \eqref{eq-SInfInSup} and Lemma \ref{lemma-TInfinSup}, for any cell $ R_k(ι,τ) $ and any cell $ R_{k'}(ι'',τ'') $ which is adjacent to a descendant of $R_k(ι,τ)$ of scale $k'$, it holds that
\begin{equation}\label{eq-RInfinSup}
	R\inbase_{k'}(ι'',τ'')	\subseteq R_k\Sup(ι,τ).
\end{equation}
In particular, for any two cells $ R_k(ι,τ) $ and $ R_{k'}(ι',τ') $,
\begin{equation}\label{eq-InfIntersSupInters}
	R\inbase_{k}(ι,τ)\cap R\inbase_{k'}(ι',τ')	\neq \emptyset
	\quad \Rightarrow	\quad
	R_k\Sup(ι,τ)		\cap R_{k'}\Sup		(ι',τ')		\neq \emptyset,
\end{equation}
which means that if the areas of influence of two cells intersect then also the supports intersect.

Note that we defined the extended supports \eqref{def-SEsup} and \eqref{def-TEsup} in such a way that it holds for two cells $ R_{k_1}(ι_1,τ_1) $ and $ R_{k_2}(ι_2,τ_2) $ with $ k_1\le k_2 $ that
\begin{equation}\label{eq-SuppIntersEsup}
	R\Sup_{k_1}(ι_1,τ_1)\cap R\Sup_{k_2}(ι_2,τ_2)\neq \emptyset
	\quad \Rightarrow \quad
	R\Esup_{k_2}(ι_2,τ_2)\supseteq R\Sup_{k_1}(ι_1,τ_1),
\end{equation}
which means that if the supports of two cells intersect, then the bigger extended support contains the smaller support.

\subsection{Fractal percolation}\label{subsec-fractal}

We now introduce several events to define new notions of goodness for each scale $ k. $ Having multi-scale levels of goodness is the link to the theory of fractal percolation. We will provide details about the analogy and an intuitive explanation of the following definitions at the end of the subsection.

For $ε>0$ as in the assumptions of Theorem \ref{thm-main}, we define the sequence
\begin{equation}\label{def-delta}
	\mathfrak{d}_1:=ε,				\qquad\qquad \mathfrak{d}_{k+1}:=\mathfrak{d}_{k}-\frac{ε}{2k^2}, \quad k \ge 1.
\end{equation}
Recalling the definition of $ S\base $ and $ S\ext $ in \eqref{def-Sbase} and \eqref{def-Sext}, as well as the particle system under consideration (see Section \ref{subsec-PRW}),  define the following indicator random variables:
\begin{align}
	\label{def-Dk}    & D_k(ι,τ)=1      &  & \begin{array}{l}
		                                         \text{if all tiles } S_{k-1}(ι')\subseteq S_k(ι) \text{ contain at least }\big( 1-\mathfrak{d}_{k}  \big)μ_0\sum_{y\in S_{k-1}(ι')}λ_y \\
		                                         \text{particles at time } τβ_k,\end{array}        \\[9pt]
	\label{def-Dext}  & D\ext_k(ι,τ)=1  &  & \begin{array}{l}
		                                         \text{if all tiles } S_{k-1}(ι')\subseteq S\ext_k(ι) \text{ contain at least }\big(1-  \mathfrak{d}_{k} \big) μ_0\sum_{y\in S_{k-1}(ι')}λ_y \\
		                                         \text{particles at time } τβ_k                                                                                                              \\
		                                         \text{that are confined during $ [τβ_k,(τ+2)β_k] $ inside } \displ_{b(k-1)2^{\ell_{k-1}}},\end{array}   \\[9pt]
	\label{def-Dbase} & D\base_k(ι,τ)=1 &  & \begin{array}{l}
		                                         \text{if all tiles } S_{k}(ι')\subseteq S\base_{k}(ι) \text{ contain at least } \big(1-\mathfrak{d}_{k+1}  \big)μ_0\sum_{y\in S_{k}(ι')}λ_y \\
		                                         \text{particles at time } γ_k^{(1)}(τ)β_{k+1}                                                                                               \\
		                                         \text{that are confined during $ [γ_{k}^{(1)}(τ)β_{k+1},τβ_k] $ inside } \displ_{b(k)2^{\ell_{k}}}.
	                                         \end{array}
\end{align}
Since $S_{k}\subseteq S\ext_k $, trivially $ D\ext_k(ι,τ)=1  $ implies $ D_k(ι,τ)=1  $.
Noting that $S\base_{k}(ι)\subseteq S\ext_{k+1}(π_{k}^{(1)}(ι)) $ as mentioned in \eqref{eq-BaseInExt}, and that $ [γ_{k}^{(1)}(τ)β_{k+1},τβ_k] \subset [γ_{k}^{(1)}(τ)β_{k+1},(γ_{k}^{(1)}(τ)+2)β_{k+1}] $ we have by definition
\begin{equation}\label{eq-Dext=>Dbase}
	D\ext_{k+1}\big(π_{k}^{(1)}(ι),γ_{k}^{(1)}(τ)\big) =1
	\qquad	\Rightarrow \qquad
	D\base_k(ι,τ)=1		\qquad \forall (k,ι,τ)\in \N\times \mathbb{B}^d\times \Z,
\end{equation}
and the goal of Lemma \ref{lemma-Mixing} below will be to show that with exponentially large probability, $\{D\base_{k}(ι,τ)=1\}$ implies $ \{D\ext_k(ι,τ) =1\}. $ To this end, we define
\begin{align}
	A_1(ι,τ) & :=\max\{ \ind_{E(ι,τ)}				,	1-D_1\base(ι,τ)\}	\label{def-A1}, \\
	A_k(ι,τ) & :=\max\{D_k\ext(ι,τ)	,	1-D_k\base(ι,τ)\}	\label{def-Ak},      \\
	A_κ(ι,τ) & :=D_κ\ext(ι,τ),
	\label{def-Aκ}
\end{align}
and
\begin{equation}\label{def-Aprod}
	A(ι,τ) := \prod_{k=1}^{κ} A_k (π_{1}^{(k-1)} (ι),    γ_{1}^{(k-1)}(τ)				).
\end{equation}
The first-time reader should think that $ A_k(ι,τ) =0$ intuitively indicates that ``in the chain of space-time cells that are ancestors of $R_1(ι,τ)$, the particles misbehaved at scale $k$'': more precisely, $ A_k(ι,τ)=0 $ if, even despite the favorable event $ D_k\base(ι,τ)=1 $, according to which the particle were in a good state inherited from higher scales, it resulted in $ D_k\ext(ι,τ )=0.$ As already mentioned above \eqref{def-A1}, we will prove that the previous situation happens with small probability in Lemma \ref{lemma-Mixing}.

We can now define the notions of goodness that we will consider. Recall that we defined at the very start of subsection \ref{subsec-mainres} that
\begin{equation}\label{def-bad}
	\text{a cell } R_1(ι,τ) \text{ is \emph{bad} if } \ind_{E(ι,τ)}=0.
\end{equation}
We consider now a stronger notion of bad cells for any scale $ 1\le k\le κ: $
\begin{equation}	\label{def-multiscalebad}
	\text{a cell } R_k(ι,τ) \text{ is \emph{multi-scale bad} if } A_k(ι,τ)=0.
\end{equation}Note that for scale one this definition is stricter then the definition of being \emph{bad}: as a simple consequence of \eqref{def-A1}, a \emph{multi-scale bad} cell is also \emph{bad}. Finally, we say for scale-one cells that
\begin{equation}	\label{def-badancestry}
	\text{a cell } R_1(ι,τ) \text{ has \emph{bad ancestry} if } A(ι,τ)=0,
\end{equation}  or equivalently that the cell has a \emph{multi-scale bad} ancestor.

In particular, a \emph{bad} cell of scale one has \emph{bad ancestry}, as we prove in the following lemma.
\begin{lemma}\label{lemma-Bad=>BadAnc}
	For a cell $ R_1(ι,τ) $ it holds  $ \ind_{E(ι,τ)}\ge A(ι,τ). $
	Equivalently, a scale-one cell which is \emph{bad}, in particular has \emph{bad ancestry}.
\end{lemma}
\begin{proof}
	Suppose  that $ A(ι,τ)=1 $. By \eqref{def-Aprod}, it therefore holds for all $ 1\le k\le κ $,  that $$ A_k \big(π_{1}^{(k-1)} (ι),    γ_{1}^{(k-1)}(τ)				\big)=1.$$

	In particular $ D\ext_{κ}\big(π_{1}^{(κ-1)} (ι),    γ_{1}^{(κ-1)}(τ)	\big) =1$, so applying the property in \eqref{eq-Dext=>Dbase} we obtain $ D\base_{κ-1}\big(π_{1}^{(κ-2)} (ι),    γ_{1}^{(κ-2)}(τ)	\big)   =1$. Since $ A_{κ-1}\big(π_{1}^{(κ-2)} (ι),    γ_{1}^{(κ-2)}(τ)	\big) =1  $ and it is defined as a maximum, the first argument needs to be a 1, and we obtain $ D\ext_{κ-1}\big(π_{1}^{(κ-2)} (ι),    γ_{1}^{(κ-2)}(τ)	\big)   =1.$

	Repeating this argument for all scales down to scale one, we need the first argument in the maximum of $ A_1(ι,τ) $ to be 1, i.e.\  it must hold that $ \ind_{E(ι,τ)}=1. $
\end{proof}

\paragraph{Intuition.}
We conclude this subsection by explaining the analogy of our setup to fractal percolation, whose framework has inspired this proof. For simplicity, we will explain the arguments on $ \R^d$ instead of the \Sier{} gasket.

Fix some value $ r\in\N. $ Consider the unit hyper-cube and subdivide it into $ r^d $ cubes of side length $ \frac{1}{r}. $ Then, for some value $ p\in [0,1], $ declare them open independently with probability $ p $ and closed otherwise. Then, subdivide again each of the \emph{open} cubes  into $ r^d $ cubes of side length $ \frac{1}{r^2}, $ and each of the second-level cubes is open with probability $ p $ and closed otherwise. Note that each level-1 cube that was closed is not further subdivided and so it is entirely closed. One can then repeat the above procedure with further subdivisions, see Figure \ref{fig-FracPerc}. This recursive construction introduces correlations into the system that would not be present in standard Bernoulli percolation: if we take two cubes of some arbitrary size, the probability that both of them are open simultaneously is strongly affected by how far back in the subdividing procedure their most recent open common ``ancestor'' was.

\begin{figure}[ht]
	\centering
	\includegraphics[width=0.8\linewidth]{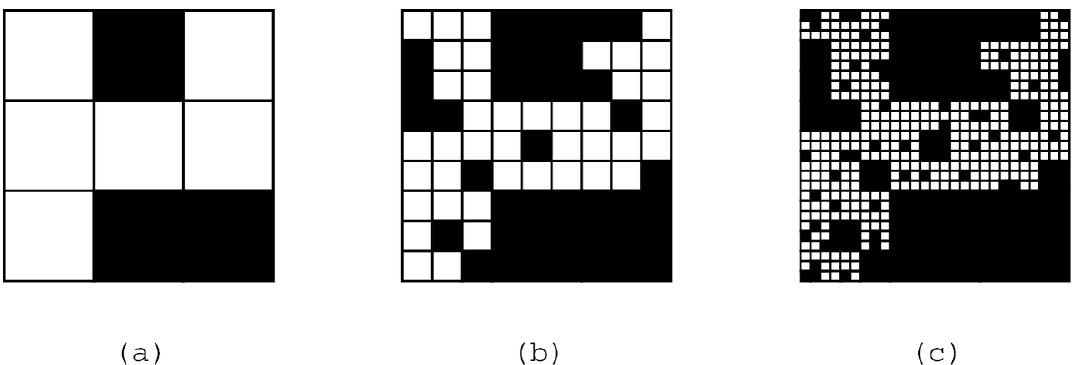}
	\caption{An example of fractal percolation in $ \R^2. $ Image from \cite{Pete19}.}\label{fig-FracPerc}
\end{figure}

The similarity with our case is straightforward. To obtain $ A(ι,τ)=1 $ we need a cell and all its ancestors to be multi-scale good, similarly to the fractal percolation where the cubes must be open at every level-$ k $ in order to be open at the last and smallest level. In view of Lemma \ref{lemma-Bad=>BadAnc}, a cell with $ A(ι,τ)=1 $ is then \emph{good}, in the sense below Definition \ref{def-E}.
It may seem now that directly performing a single-level percolation at scale $ k=1 $ might be easier, but unlike the fractal percolation described above, cells in our setting have further dependencies beyond the ones introduced by the subdivisions. In particular, note that knowing a cell of some scale $k$ is bad reveals information not only about its descendant cells, but also any other cells that are spatially and temporally close enough to be affected by the behaviour of the particles from the cell in question.
The other difference is that the percolation parameter $ p $ will not be kept constant: in our case the probability to be a multi-scale good cell $ \P(A_k(ι,τ)=1) $ is higher at larger scales, as we will prove in Lemma \ref{lemma-Mixing}. The proof there involves the events $ D\ext_k $ and $ D\base_k $ defined above in \eqref{def-Dext} and $ \eqref{def-Dbase} $, and in particular the strategy is as follows: assuming the favorable event $ D\ext_k(ι,τ) =1$, using the decoupling Theorem \ref{thm-mixing}, if we restrict to a slightly smaller cell (so from $S\ext_k(ι) $ to $ S\base_k(ι) $) and ``wait a bit'', we are able to resample the particles according an independent Poisson point process with only a slightly smaller intensity. This resampling allows us to essentially treat the configuration of the particles in the space-time cell in question as independent of the configuration elsewhere, thus roughly recovering the fractal percolation setup outlined above and taking care of both types of correlations mentioned at once.

\subsection{Paths of cells}\label{SUBSEC-Paths}

We next define the two notions of ``paths of cells'' that we will consider. As we will see momentarily, both notions are strongly related to d-paths from Definition \ref{def-dpath}.

Recall that, in line with Definition \ref{def-adjacent}, two cells $ R_k(ι_1,τ_1)\neq R_k(ι_2,τ_2) $ of same scale are called \emph{adjacent} if either $ d(S_k(ι_1),S_k(ι_2))=0$ and $τ_1=τ_2$, or $ι_1=ι_2$ and $|τ_1-τ_2|\le 1.$ We now extend this to cells of different scales.
Two cells $ R_{k_1}(ι_1,τ_1), $ and $R_{k_2} (ι_2,τ_2) $ with scales $ k_1>k_2 $ are called \emph{adjacent} if $ R_{k_1}(ι_1,τ_1), $ is adjacent to $ R_{k_1}\big(π_{k_2}^{k_1-k_2}(ι_2),γ_{k_2}^{k_1-k_2}(τ_2)  \big)$. Note that in particular, a cell is not adjacent to any of its ancestors.

Before proceeding, recall the definition of the base of the space-time tessellation $L_0$ given in \eqref{def-L0}. We say for two scale-one cells $ R_1(ι,τ) $ and $ R_1(ι',τ') $ that $ R_1(ι,τ) $ is  \emph{diagonally connected}  to $ R_1(ι',τ') $ if there exists a sequence of adjacent cells $ \{R_1(ι_1,τ_1),\dots, R_1(ι_n,τ_n)\} $  of scale one such that $R_1(ι,τ)=R_1(ι_1,τ_1)$, for all $ j\in \{1,\dots,n-1\} $ we have
$d(R_1(ι_{j+1},τ_{j+1}),L_0)< d(R_1(ι_j,τ_j),L_0),$ and $ R_1(ι_n,τ_n) $ is either equal or adjacent to $ R_1(ι',τ') $. Accordingly, the cells $R_1(ι_j,τ_j)$, $j\in \{1,\dots,n-1\}$ (and $R_1(ι_n,τ_n)$ if it differs from $ R_1(ι',τ') $) will be referred to as  \emph{diagonal steps}.

\begin{remark}
    Note that in the sequence of adjacent cells constituting the diagonal steps, the temporal coordinate is not changing in the sense that $τ = τ_j$ for all $ j\in \{1,\dots,n-1\}$. This agrees with the definition of d-paths (see Definitions \ref{def-dpath} and \ref{def-adjacent}) where diagonal moves also preclude temporal changes.
\end{remark}

For two cells $ R_{k_1}(ι_1,τ_1)$ and $ R_{k_2}(ι_2,τ_2) $ of not necessarily different scales we say that $ R_{k_1}(ι_1,τ_1)$ is \emph{diagonally connected} to $ R_{k_2}(ι_2,τ_2) $ if there exist two cells $ R_1(\tilde{ι}_1,\tilde{τ}_1)$ and $ R_1(\tilde{ι}_2,\tilde{τ}_2) $ of scale one, descendants of $ R_{k_1}(ι_1,τ_1) $ and $ R_{k_2}(ι_2,τ_2) $, respectively , so that  $ R_1(\tilde{ι}_1,\tilde{τ}_1)$ is diagonally connected to $ R_1(\tilde{ι}_2,\tilde{τ}_2) $.

\begin{deff}\label{def-Dpath}
	We define a \emph{D-path} as a sequence of
	cells of arbitrary scales, where each cell is either adjacent or diagonally connected to the next cell in the sequence.
\end{deff}

The reader will note the analogy to the definition of d-path in Definition \ref{def-dpath}. Fix a cell $ v=R_1(\iota_v,\tau_v)\in L_1 $ 
and define for any (large) $ t>0$ the set
\begin{equation}\label{def-Ωt}
	Ω_1(v\rightarrow t)
\end{equation}
of all D-paths of cells of scale one for which the first cell of the path is $v $ and the last cell is the only cell not contained in $ B_{t}(S_1(ι_v))\times [-t+τ_v,τ_v+t] $, where $B_{t}(S_1(ι_v)):= \cup_{x\in S_1(ι_v)} B_t(x)$.

The next notion of paths involves instead cells of multiple scales.
%
\begin{deff}\label{def-ScDpath}
	We define as \emph{ScD-path}  (support connected with diagonal paths) a sequence of cells of possibly different scales $ \{R_{k_1}(ι_1, τ_1),\dots, R_{k_z}(ι_z, τ_z)\} $ for some $ z\in\N, $ with the following properties:
	\begin{itemize}
		\item each pair of cells is  \emph{well-separated}, meaning that their areas of influence do not intersect; i.e.\ for any pair $ R_{\tilde{k}}(\tilde{ι},\tilde{τ}) ,$ $ R_{\hat{k}}(\hat{ι},\hat{τ} )$ we have
		      \begin{equation*}
			      R\inbase_{\tilde{k}} (\tilde{ι}, \tilde{τ})  	\cap	R\inbase_{\hat{k}} (\hat{ι}, \hat{τ})  =\emptyset,
		      \end{equation*}
		\item two consecutive cells $ R_{k_{j}} (ι_{j},τ_{j}) $ and $ R_{k_{j+1}} (ι_{j+1},τ_{j+1})  $ are either
		      \begin{equation*}
			      \text{\emph{support adjacent}}\colon \ R\Esup_{k_j}(ι_j,τ_j)			\cap 		R\Esup_{k_{j+1}}(ι_{j+1},τ_{j+1})		\neq \emptyset
		      \end{equation*}
		      or \begin{equation*}
			      \text{\emph{support connected with diagonals}}\colon
			      \begin{array}{l}
				      \text{there exist two scale-one cells, respectively}                  \\ \text{subsets of the extended supports of } R_{k_{j}}(ι_{j},τ_{j})\text{ and} \\
				      \text{$R_{k_{j+1}} (ι_{j+1},τ_{j+1}),  $ so that the first cell is} \\
				      \text{diagonally connected to the second.}
			      \end{array}
		      \end{equation*}
	\end{itemize}
\end{deff}

For $ v\in L_1$ and $ t>0$, we define
\begin{equation}\label{def-Ωtsup}
	Ω\Sup_{κ}(v\rightarrow t)
\end{equation}

as the set of all ScD-paths of cells of scale at most $ κ $ so that the extended support of the first cell of the path contains $v$ and the last cell is the only cell whose extended support is not contained in $ B_{t}(S_1(ι_v))\times [-t+τ_v,τ_v+t] $ with $ι_v,τ_v$ as before.

Define the bad cluster around $ v\in L_1$ as
\begin{equation}\label{def-K}
	K_v:=\big \{R_1(\tilde{ι},\tilde{τ})\colon \text{there exists a D-path of bad cells from $ v $ to $ R_1(\tilde{ι},\tilde{τ}) $}\big\}.
\end{equation}

We can relate D-paths and ScD-paths via the following technical lemma.

\begin{lemma}\label{lemma-Ω1Ωsup}
	For any $ t>0 $ and $ v\in L_1$, it holds that
	\begin{align*}
		 & \P \big (\exists P \in Ω_1(v\rightarrow t) \text{ of cells with bad ancestry} \big) \\
		 & \le
		\P \big(\exists P \in Ω\Sup_κ(v\rightarrow t) \text{ of multi-scale bad cells} \big).
	\end{align*}
\end{lemma}
\begin{remark}
	Note that for a path $P \in Ω_1(v\rightarrow t)$ of cells with bad ancestry, the property of having a bad ancestor is required only for the cells of $P$ and not for the cells constituting the diagonal steps in the diagonal connections of $P$. This is in line with Definition \ref{def-dpath}, where diagonal moves of d-paths do not impose any requirements on the state of the cells. The same is of course true also for $P \in Ω\Sup_κ(v\rightarrow t)$, where being multi-scale bad is not required for the cells constituting diagonal connections.
\end{remark}
\begin{proof}[Proof of Lemma \ref{lemma-Ω1Ωsup}]
	We split the proof into two steps. Defining $ Ω_{κ}(v\rightarrow t) $ as the set of D-paths of cells of scale at most $ κ $, where the first cell is an ancestor of $ v $ and the last cell is the only cell whose support is not contained in $ B_{t}(S_1(ι_v))\times [-t+τ_v,τ_v+t] $, 
	we prove in the two steps that
	\begin{align*}
		 & \P \big(\exists P \in Ω_1(v\rightarrow t) \text{ of cells with bad ancestry}\big) \\
		 & \le
		\P \big(\exists P \in Ω_κ(v\rightarrow t) \text{ of multi-scale bad cells}\big)      \\
		 & \le
		\P \big(\exists P \in Ω\Sup_κ(v\rightarrow t) \text{ of multi-scale bad cells}\big).
	\end{align*}

	\textit{Step 1.} Consider a D-path $ P=\big(R_1(ι_j,τ_j) \big)_{j=1}^{z} \in Ω_1(v\rightarrow t)$ of cells with bad ancestry. By definition, for each cell of $P$ it holds that $ A(ι_j,τ_j)=0 $, so there exists $ k_j $ such that $ A_{k_j}(π_{1}^{k_j'-1}(ι_j),   γ_{1}^{k_j'-1}(τ_j))=0, $ so that $ R_{\tilde{k}_j}(\tilde{ι}_j, \tilde{τ}_j):=R_{k_j}\big( π_{1}^{k_j'-1}(ι_j),   γ_{1}^{k_j'-1}(τ_j)\big) $ is a multi-scale bad cell.
	From the sequence $P':= \{R_{\tilde{k}_j}(\tilde{ι}_j, \tilde{τ}_j)\}_{j=1}^{z} $ construct a subsequence $ P'':=\{R_{k_j''}(ι''_j, τ''_j)\}_{j=1}^{z''}  $ taking in the same order of the cells from $ P' $ but removing all cells indexed by $ \hat{j} $ which are the descendant of some other cell in the path $P' $ with index $j_0 $, with $ j_0<\hat{j} $. Furthermore, if there is a cell $ R_{\tilde{k}_j}(\tilde{ι}_j, \tilde{τ}_j) $ before the last one whose support is not contained in $ B_{t}(S_1(ι_v))\times [-t+τ_v,τ_v+t] $, 
	we remove from $ P'' $ all following cells.

	We claim that $ P''\in Ω_{κ}(v\rightarrow t),$  which will conclude step 1.
	This path starts with an ancestor of $ v $ and by construction the last cell's support is not contained in $ B_{t}(S_1(ι_v))\times [-t+τ_v,τ_v+t] $. 
	Note that every cell in $ P $ has exactly 1 ancestor in $ P''. $ Consider now two cells $R_1(ι_j,τ_j)  $ and $ R_1(ι_{j+1}, τ_{j+1}) $ with different ancestors in $ P''. $ If $R_1(ι_j,τ_j)  $ is diagonally connected to $ R_1(ι_{j+1}, τ_{j+1}) $, then the ancestor of $R_1(ι_j,τ_j)  $ is either diagonally connected or adjacent to the ancestor of $ R_1(ι_{j+1}, τ_{j+1}) $; if $R_1(ι_j,τ_j)  $ and $R_1 (ι_{j+1}, τ_{j+1}) $ are adjacent, then their ancestors are adjacent, since two non-adjacent cells cannot have two adjacent descendants. Finally, every cell of $P''$ is multi-scale bad by how $P''$ was constructed.

	\begin{figure}[!ht]
		\centering
		\begin{subfigure}[b]{\linewidth}
			\includegraphics[width=\linewidth]{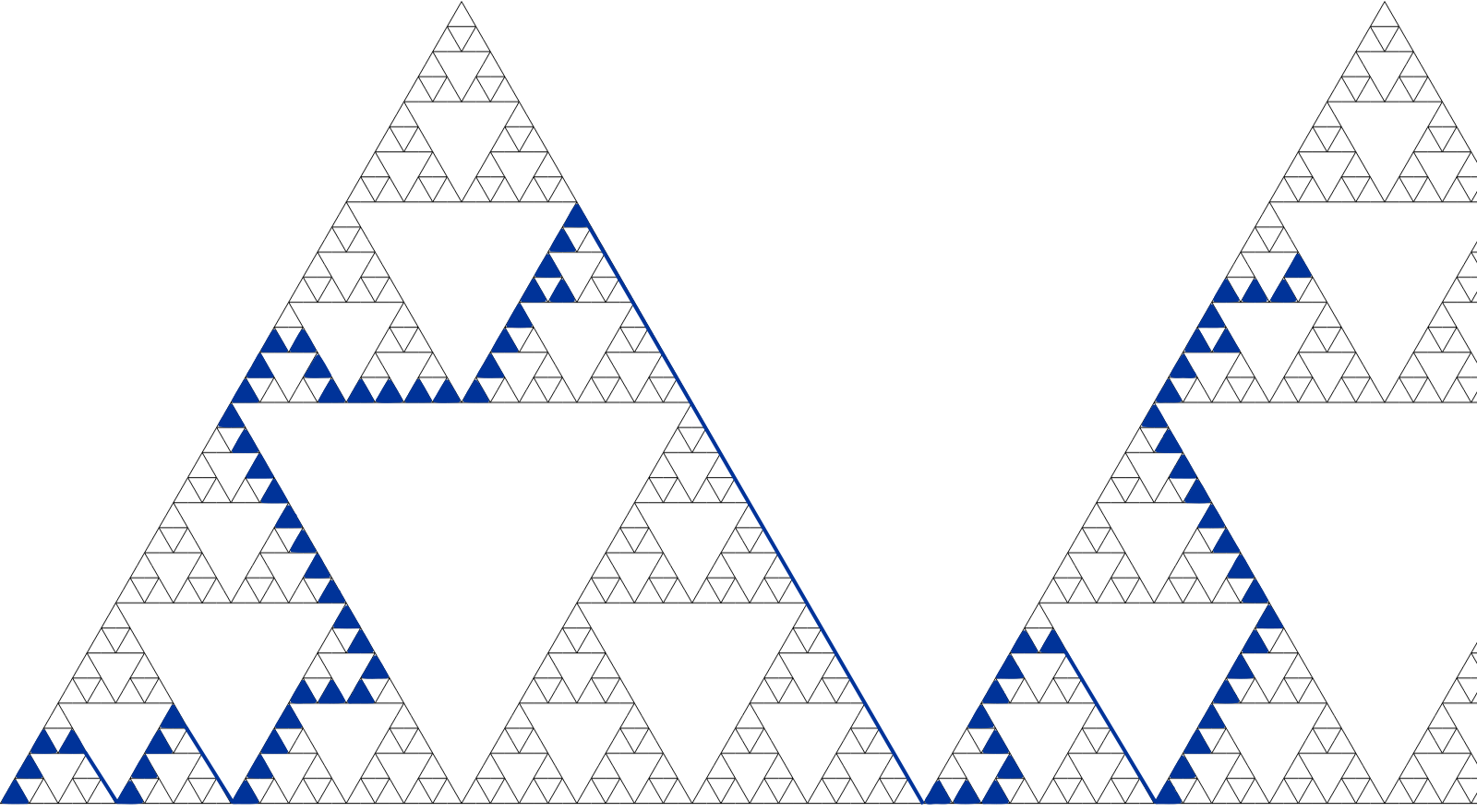}
			\subcaption{A possible D-path with adjacent and diagonally connected cells.}
		\end{subfigure}
		\begin{subfigure}[t]{0.48\linewidth}
			\includegraphics[width=\linewidth]{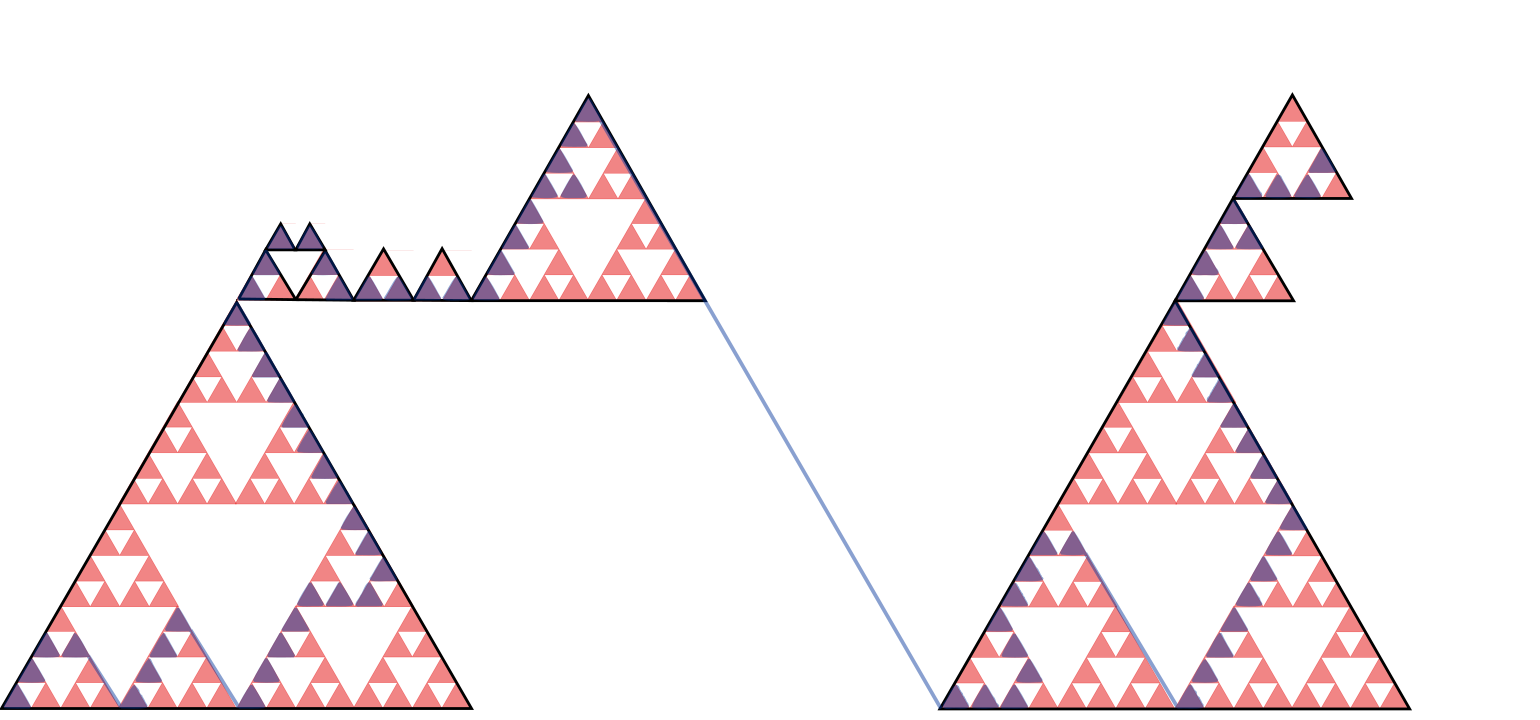}
			\subcaption{A D-path of multi-scale bad cells (in red with a thicker border) in comparison with the D-path (in blue) of the previous image. Note that many cells of scale one correspond to the same cell in this image.}
		\end{subfigure}
		\hfill
		\begin{subfigure}[t]{0.48\linewidth}
			\includegraphics[width=\linewidth]{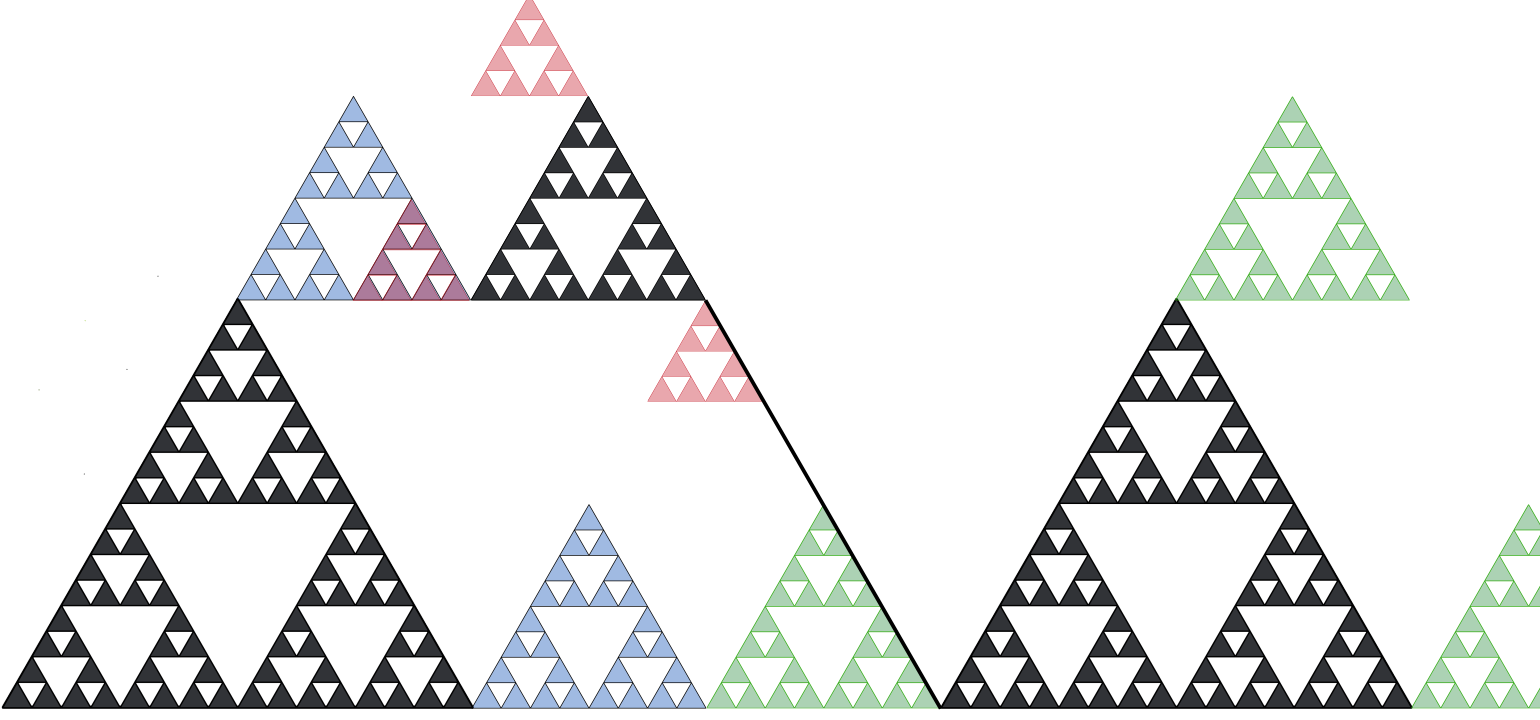}
			\subcaption{The corresponding ScD-path (in black), where some cells were discarded as they were not \emph{well-separated}. We highlight (respectively in blue, red and green) the extended supports and (in black) the diagonal of 2 cells which are \emph{support connected with diagonal}.}
		\end{subfigure}
		\caption{From D-paths to ScD-paths. Note that this example is on $\mathbb{G}$ without the time component in order to make the visualisation easier. In practice, the procedure is conducted on cells of $\mathbb{G}\times\mathbb{Z}$.}\label{fig-paths}
	\end{figure}

	\textit{Step 2.} We now prove the second inequality, that is, starting from $ P'' $ we can obtain a path $ \hat{P} $ of multi-scale bad cells which are well-separated and in which every sequential pair of cells is either support adjacent or the first cell of the pair is support connected with diagonals to the second.

	First define a sequence $ L $ of cells from $ P'' $, but where the cells are ordered in the following way: we first order cells by scale, where cells of bigger scale come first, and within cells of the same scale we maintain the original order of $ P'' $. We construct $ \hat{P} $ and create a relation between $ P'' $ and $ \hat{P}$ in the following way. Following the order of $ L $, and in particular starting with scale $ κ $, we perform the following operations. Assuming the first cell of scale $ k $ in the list $ L $ is $ R_k(\hat{ι}, \hat{τ}) $ we
	\begin{itemize}
		\item add $ R_k(\hat{ι}, \hat{τ}) $ to $ \hat{P};$
		\item remove $  R_k(\hat{ι}, \hat{τ})  $ from $ L; $
		\item associate $  R_k(\hat{ι}, \hat{τ})  $ in $ P'' $ with itself in $ \hat{P};$
		\item remove from $ L $ all cells $  R_{\tilde{k}}(\tilde{ι}, \tilde{τ})  $ which are not well-separated from $ R_k (\hat{ι}, \hat{τ})  $ and associate them all with $  R_k(\hat{ι}, \hat{τ})  $ in $ \hat{P}. $
	\end{itemize}

	Repeating this procedure until $ L$ is empty, we obtain a sequence of cells $ \hat{P} $, and all cells in $ P'' $ are associated to some cell in $ \hat{P}. $ Before proceeding, we reorder $ \hat{P} $ according to the ordering in $P''$, thus making $ \hat{P} $ a path (which we will verify below). In particular, a cell $v$ in $ \hat{P} $ appears before a different cell $u$ of $ \hat{P} $ if according to the ordering of $P''$, there exists a cell of $P''$ associated to $v$ that appears before any cell of $P''$ associated to $u$. Since the multi-scale bad property follows trivially from $ P'', $ we are only left to show that
	\begin{equation}\label{eq-GoalΩΩ}
		\hat{P}\in Ω_{κ}\Sup(v\rightarrow t).
	\end{equation}

	First, let $ R_{\hat{k}_1}(\hat{ι}_1,\hat{τ}_1) \in \hat{P}$ be the cell which $ R_{k_1''}(ι_1'',τ_1'')\in P'' $ is associated to. In the non-trivial case, $ R_{\hat{k}_1}(\hat{ι}_1,\hat{τ}_1)   $ is not associated to itself, so $R_{\hat{k}_1}(\hat{ι}_1,\hat{τ}_1)  $ and $ R_{k_1''}(ι_1'',τ_1'')$ are not well-separated and therefore their areas of influence intersect. By \eqref{eq-InfIntersSupInters} their supports intersect as well. By \eqref{eq-SuppIntersEsup}, $ R\Esup_{\hat{k}_1}(\hat{ι}_1,\hat{τ}_1)  \supseteq R\Sup _{k_1''}(ι_1'',τ_1'')$, and since $R_{k_1''}(ι_1'',τ_1'')$ contains $ v $ by definition of $ P'',  $ we obtain
	that $ R\Esup_{\hat{k}_1}(\hat{ι}_1,\hat{τ}_1)$ contains $ v $ as desired.

	Secondly, we can argue in the same way to show that the extended support of the cell which $ R_{k''_{z''}}( ι''_{z''}, τ''_{z''}) $ is associated to is not contained in the space-time ball $ B_{t}(S_1(ι_v))\times [-t+τ_v,τ_v+t] $.

	Finally, we need to show that sequential pairs of cells of $\hat{P}$ are either support adjacent or the first cell of the pair is support connected with diagonals to the second. Consider $ R_{\hat{k}_j}(\hat{ι}_j,\hat{τ}_j)\in \hat{P}$, and let $ R_{k_{j''},}(ι_{j''},τ_{j''}) $ be the first cell of $ P'' $ (in the original ordering of $ P'' $) which is associated to $R_{\hat{k}_j} (\hat{ι}_j,\hat{τ}_j) $. Next, take $ R_{k_{j''-1}}(ι_{j''-1},τ_{j''-1})\in P'' $ and let $ R_{\hat{k}_{j-1}}(\hat{ι}_{j-1},\hat{τ}_{j-1})\in \hat{P} $ be the cell which it is associated to. We claim that $ R_{\hat{k}_j,}(\hat{ι}_j,\hat{τ}_j)$ and $ R_{\hat{k}_{j-1}}(\hat{ι}_{j-1},\hat{τ}_{j-1}) $ are either support adjacent or $ R_{\hat{k}_{j-1}}$ is support connected with diagonals to $ R_{\hat{k}_j,}(\hat{ι}_j,\hat{τ}_j)$ based on whether  $ R_{k_{j''}}(ι_{j''},τ_{j''}) $ and $ R_{k_{j''-1}}(ι_{j''-1},τ_{j''-1})\in P'' $ are adjacent or whether $ R_{k_{j''-1}}(ι_{j''-1},τ_{j''-1})$ is connected with diagonals to $ R_{k_{j''}}(ι_{j''},τ_{j''}) .$

	If $ R_{k_{j''-1}}(ι_{j''-1},τ_{j''-1}) $ and $ R_{k_{j''}}(ι_{j''},τ_{j''}) $ are adjacent, we can suppose without loss of generality that $ k_{j''-1}\le k_{j''} $, and by definition there exists a cell $ R_{k_{j''}}(\tilde{ι}_{j''-1},\tilde{τ}_{j''-1}) $, which is an ancestor of $ R_{k_{j''-1}}(ι_{j''-1},τ_{j''-1}) $ and adjacent to $ R_{k_{j''}}(ι_{j''},τ_{j''}). $ Hence applying $ \eqref{eq-RInfinSup} $ twice we obtain that
	\begin{equation}\label{eq-proof-InfInSup}
		\begin{split}
			R\inbase_{k_{j''}}(ι_{j''},τ_{j''})&\subseteq R\Sup_{k_{j''}}(\tilde{ι}_{j''-1},\tilde{τ}_{j''-1})
			\\
			&\text{and}\\
			R\inbase_{k_{j''-1}}(ι_{j''-1},τ_{j''-1})&\subseteq
			R\Sup_{k_{j''}}(\tilde{ι}_{j''-1},\tilde{τ}_{j''-1}) .
		\end{split}
	\end{equation}
	Since $ R_{k_{j''}}(ι_{j''},τ_{j''})  $ is associated to $ R_{\hat{k}_j}(\hat{ι}_j,\hat{τ}_j) $, they are not well-separated and thus their areas of influence intersect. Therefore \eqref{eq-proof-InfInSup}  implies that $ R\Sup_{k_{j''}}(\tilde{ι}_{j''-1},\tilde{τ}_{j''-1})$ intersects $R\inbase_{\hat{k}_{j}}(\hat{ι}_{j},\hat{τ}_{j})$ and by \eqref{eq-RInfinSup} intersects $R\Sup_{\hat{k}_{j}}(\hat{ι}_{j},\hat{τ}_{j})$; since $ \hat{k}\ge k_{j''} $, applying \eqref{eq-SuppIntersEsup}, we have $ R\Esup_{\hat{k}_{j}}(\hat{ι}_{j},\hat{τ}_{j}) \supseteq $ $ R\Sup_{k_{j''}}(\tilde{ι}_{j''-1},\tilde{τ}_{j''-1}) $$\supseteq R\inbase_{k_{j''-1}}(ι_{j''-1},τ_{j''-1})$ where the last inclusion is due to \eqref{eq-proof-InfInSup}.
		Since the cells $ R_{k_{j''-1}}(ι_{j''-1},τ_{j''-1})$ and $ R_{\hat{k}_{j-1}}(\hat{ι}_{j-1},\hat{τ}_{j-1})$ are not well-separated, repeating the same argument below \eqref{eq-GoalΩΩ} we have
	$ R\Esup_{\hat{k}_{j-1}} (\hat{ι}_{j-1},\hat{τ}_{j-1})$
	$\supseteq$ $R\Sup_{k_{j''-1}}({ι}_{j''-1},{τ}_{j''-1})$ $ \supseteq $ $ R\inbase_{k_{j''-1}}(ι_{j''-1},τ_{j''-1}) $, where the last inclusion follows from \eqref{eq-RInfinSup}. This shows that the two extended supports intersect.

		If instead $ R_{k_{j''-1}}(ι_{j''-1},τ_{j''-1}) $ is connected with diagonals to $R_{k_{j''}}(ι_{j''},τ_{j''}) $, then by definition they contain respectively two cells $ R_1(\tilde{ι}_{j''-1},\tilde{τ}_{j''-1})$ and
	$ R_1(\tilde{ι}_{j''},\tilde{τ}_{j''}) $ such that $ R_1(\tilde{ι}_{j''-1},\tilde{τ}_{j''-1})$ is connected with diagonals to $ R_1(\tilde{ι}_{j''},\tilde{τ}_{j''}) $.
		Additionally, since $ R_{k_{j''}} (ι_{j''},τ_{j''}) $ is associated to $ R_{\hat{k}_j}(\hat{ι}_j,\hat{τ}_j)$, they are not well-separated and by the argument below \eqref{eq-GoalΩΩ} we have $ R_1(\tilde{ι}_{j''},\tilde{τ}_{j''}) \subseteq R\Sup_{{k}_{j''}}(ι_{j''},τ_{j''})  \subseteq R\Esup_{\hat{k}_j}(\hat{ι}_j,\hat{τ}_j).$ With the same argument, $ R_1 (\tilde{ι}_{j''-1},\tilde{τ}_{j''-1}) \subseteq  R\Esup_{\hat{k}_{j-1}}(\hat{ι}_{j-1},\hat{τ}_{j-1})$. This shows that $ R_{\hat{k}_{j-1}}(\hat{ι}_{j-1},\hat{τ}_{j-1}) $ is support connected with diagonals to  $ R_{\hat{k}_j}(\hat{ι}_j,\hat{τ}_j)$, which concludes the proof.
\end{proof}

\section{Multi-scale analysis}\label{SEC-MultiScaleAnalysis}

We will now use the multi-scale set-up introduced above in order to bound the probability of having paths of multi-scale bad cells. Recall that $ζ\in(0,\infty)$ as defined in Theorem \ref{thm-main} plays the role of 
imposing the confinement width of the particle movement at the 
scale-one tessellation.
We now define what will essentially be the ``weight'' of a cell as
\begin{equation}\label{def-ψ}
	\begin{aligned}
		ψ_1 (ε,μ_0,\ell) & :=	\min \Big\{	\frac{ε^2μ_0 2^{\dv\ell}	} {C_λ}				,		-\log \Big(	1- ν_E\big( (1-ε)λ, S_1^η, \displ_{ζ\ell}, ηβ	 \big)
		\Big)	\Big\},                                                                                                                             \\
		ψ_k (ε,μ_0,\ell) & :=\frac{ε^2μ_02^{\dv\ell_{k-1}}	} { k^4}, \qquad k\ge 2,
	\end{aligned}
\end{equation}
which we will use as a reference for both the probability of a cell of scale $k$ to be bad, and for the number of ScD-paths which contain a cell of scale $k.$

\subsection{Probability of a multi-scale bad ScD-path}\label{SUBSEC-ProbMultiBad}

We want to estimate the probability for a cell to be multi-scale bad. While close cells are heavily dependent on each other, we still want to obtain a bound conditioning on cells which are ``not too close'', in a spatial or temporal sense. Recall the definitions of $S\inbase_k $ and $T\inbase_k$ in \eqref{def-Sinbase} and \eqref{def-Tinbase}. We define $ \mathcal{F}_k (ι,τ) $ to be the $ σ $-algebra generated by all the $A_{k'}(ι',τ')  $ for which either:

\begin{enumerate}[label=(\alph{enumi}), ref=\alph{enumi}]
	\item \label{item-Fka}
	      $T_{k'}\inbase (τ')$
	      $\cap$ $ 			 [γ_{k}^{(1)}(τ)β_{k+1},∞)	= \emptyset 	$,  or
	\item \label{item-Fkb}$ τ'β_{k'}\le τβ_k $ and $ S\inbase_k(ι) \cap S\inbase_{k'}(ι') =\emptyset. $
\end{enumerate}
Intuitively, this is information about the behaviour of particles in space-time cells that are either far enough in the past so that we can ignore them due to the starting assumptions guaranteed by $\{D_k^{\textrm{base}}(\iota,\tau)=1\},$ or which are happening roughly simultaneously with the time interval indexed by $\tau$ or later (i.e.\ during a time interval occurs after the time interval $\tau$, regardless of scale), but far enough away not to be able to influence the occurrence of the event $\{A_k(ι,τ)=0\}$ due to the confinement of the random walks under consideration. Recall that the intensity of the Poisson point process is $μ_x=μ_0λ_x.$

\begin{lemma}\label{lemma-Mixing}
	Let $\ell, β, ε,ζ,η$ be as in Theorem \ref{thm-main} with
	\begin{equation}\label{def-ζ}
		ζ \ge	\frac{1}{\ell}		\Big(	\big(   \Cr{conf1}\log \big( \tfrac{8\Cr{conf1}}{3ε} \big) \big)^{\dw-1}	ηβ	\Big)^{\frac{1}{\dw}}			.
	\end{equation}
	If $a$ and $m $  in \eqref{def-ell}  are large enough, then there exist $ C_ψ \in (0,\infty)$ and $ α_0 = α_0(ε,\beta/2^\ell,μ_0)>0 $  such that if $ ψ_1 >α_0,$ then for all $ k=1,\dots,κ $, all cells $ R_k(ι,τ) $ and any $F\in\mathcal{F}_k (ι,τ) $ with $\P(F)>0$ we have
	\begin{equation*}
		\P(A_k(ι,τ)=0\given F ) \le e^{ -C_ψ ψ_k }.
	\end{equation*}
	Furthermore, we have for scale $\kappa$ that
	\begin{equation*}
		\P(A_κ(ι,τ)=0 ) \le e^{ -C_ψ ψ_κ }.
	\end{equation*}
\end{lemma}
\begin{proof}
	We start by proving the result for $ 2\le k \le κ-1. $ Let $ F\in\mathcal{F}_k (ι,τ)$ with $\P(F)>0.$ Since
	\begin{align*}
		\P(A_k(ι,τ)=0\given F) & = \P(D\ext_k(ι,τ)=0, D\base_k(ι,τ)=1 \given F),
		\intertext{if $ \{ D\base_k(ι,τ)=1\}\cap F=\emptyset $, such probability is 0 and the lemma trivially holds, so we can assume $ \{ D\base_k(ι,τ)=1\}\cap F\neq \emptyset $ and obtain}
		\P(A_k(ι,τ)=0\given F) & \le \P(D\ext_k(ι,τ)=0 \given F,  D\base_k(ι,τ)=1).
	\end{align*}

	Recall that the event $ D\base_k(ι,τ)=1 $ (see \eqref{def-Dbase}) ensures that there are enough particles in $ S_k\base(ι) $ confined in $ \displ_{b(k)2^{\ell_{k}}} $ during $ [γ_{k}^{(1)}(τ)β_{k+1},τβ_k]  $. By definition, $ F  $ does not reveal further information about those particles because either

	\begin{itemize}
		\item by \eqref{item-Fka}, $ (τ'+2)β_{k'} \le γ_k^{(1)}(τ) β_{k+1}$ and so the time interval relevant to $ A_{k'}(ι',τ') $ does not intersect the time interval occurring in the definition of $   A_{k}(ι,τ)$, or
		\item by \eqref{item-Fkb}, $ S \inbase_{k'}(ι') \cap S\inbase_k(ι)		=\emptyset, $ so the particles in $ S_{k'}\base(ι') $ confined in $ B_{b(k)2^{\ell_k}} $ cannot leave $ S_{k'}\inbase(ι')  $ and thus cannot enter $ S_{k}\inbase(ι)  $ before $ τ'β_{k'}. $
	\end{itemize}

	Conditioned on the event $ D\base_k(ι,τ)=1 $ (defined in \eqref{def-Dbase}), we apply Theorem \ref{thm-mixing} to $ S_k\base(ι,τ) $, with the choices
	\begin{align*}
		K          & :=\text{side length of }S_k\base(ι)  = 2b(k) 2^{\ell_k}    		+2^{\ell_k},                                                  \\
		K'         & \text{ such that }K-K'=b(k)2^{\ell_k},                                                                                     \\
		l          & := 2^{\ell_k},                                                                                                             \\
		δ          & := (1-\mathfrak{d}_{k+1}) μ_0,                                                                                             \\
		\mathbf{Δ} & := \text{length}\big([γ_{k}^{(1)}(τ)β_{k+1},τβ_k]\big)=τβ_k-γ_{k}^{(1)}(τ)β_{k+1}\in[\beta_{k+1},2\beta_{k+1}],\text{ and} \\
		\bar{ε}    & :=\frac{ε}{8k^2}.
	\end{align*}

	We check now that the conditions of Theorem \ref{thm-mixing} are satisfied, starting with checking that $ K-K' \ge \Cr{MixingRoot} \big(\mathbf{Δ}(\log_2 \mathbf{Δ})^{\dw-1}\big)^{\frac{1}{\dw}} $. Since $  K-K'
		=    b(k)2^{\ell_k}	 $ and $ \mathbf{Δ}\le 2 β_{k+1} $ we need to verify that
	\begin{align*}
		b(k)2^{\ell_k}			\ge \Cr{MixingRoot}	\big(\tfrac{β_{k+1}}{2} (\log_2 \tfrac{β_{k+1}}{2})^{\dw-1}  \big)^{\frac{1}{\dw}},
	\end{align*}
	which by definition of $ β_k $ in \eqref{def-β} is implied by $ b(k)2^{\ell_k} 		\ge		\Cl{ccc} 2^{\ell_k}\ell_k k^{\frac{8}{Θ\dw}}$ for some constant $ \Cr{ccc} $. Comparing it to the definition of $ \ell_k $ in \eqref{def-ell} it holds if we set
	\begin{equation}\label{def-b(k)}
		b(k):=ak^{2+ \frac{8}{Θ\dw}}m2^{m},
	\end{equation}
	and assume $ a $ and $ m $ are large enough.
	To check that $ \mathbf{Δ}\ge \Cr{MixingDelta} l^{\dw} \bar{ε}^{\,-4/Θ} $, we use that $ \mathbf{Δ}\ge β_{k+1}=  \Cmix\big(\frac{(k+1)^2}{ε}\big)^{4/Θ} \big(2^{\ell_{k}}\big)^{\dw} $ by definition of $ β_{k+1} $ in \eqref{def-β}, and the inequality holds as $ \Cmix \ge \Cr{MixingDelta} 8^{4/Θ}$.
	We finally note that
	\begin{align*}
		K' & =K-b(k)2^{\ell_k}               
		    =b(k)2^{\ell_k}+2^{\ell_k}      
		    \geq \big(2b(k-1)\big) 2^{\ell_{k-1}}+2^{\ell_k},
	\end{align*}
	which is the side length of $S_k\ext(ι)$.

	We can therefore apply  Theorem \ref{thm-mixing} in order to obtain a coupling between the particle system at time $τβ_k$ inside $S_k\ext(ι)$  and a Poisson point process $ Ξ $ with intensity $\big( 1- \mathfrak{d}_{k+1}  \big)μ_0(1-ε)λ_y $ where the inclusion of Theorem \ref{thm-mixing} holds with probability at least
	\begin{equation*}
		1- \sum_{y\in S_{k}\ext(ι)}				e^{-\Cr{MixingProb}( 1- \mathfrak{d}_{k}  )μ_0 λ_y 		\bar{ε}^2 \mathbf{Δ}^{\dv/\dw}}.
	\end{equation*}
	Using that $ \mathbf{Δ}\ge β_{k+1} >\Cmix 2^{\dw \ell_{k}} $ and the definition of $ β_k $ from  \eqref{def-β}, the quantity in the previous display is bigger than
	\begin{equation}\label{eq-MixingLemma-CouplingProbability}
		\begin{aligned}
			 & 1-\sum_{y\in S_{k}\ext(ι)}	e^{-\C( 1- \mathfrak{d}_{k}  )μ_0 λ_y \bar{ε}^2 2^{\dv \ell_{k}} }                                             \\
			 & \ge 1-\big(2b(k-1)2^{\ell_{k-1}}+2^{\ell_k} \big)^{\dv}			e^{-\C ( 1- \mathfrak{d}_{k}  )μ_0 C_λ^{-1} \frac{ε^2}{k^4}  2^{\dv \ell_{k}} } \\
			 & \ge 1-2^{\dv(1+\ell_{k})}	e^{-\C (1-ε)μ_0 C_λ^{-1}\frac{ε^2}{k^4} 2^{\dv \ell_{k}} }                                                      \\
			 & \ge 1- \frac{1}{2} e^{-\Cpsi ψ_k}.
		\end{aligned}
	\end{equation}
	The last step holds for $ k=2 $ since $ ψ_1(ε,μ_0,\ell) $ and therefore also $ ψ_2(ε,μ_0,\ell) $ is large enough by assumption; the inequality for $ k>2 $ follows from it by setting $ a,m $ large enough.

	To obtain $ D_k\ext(κ,ι)=1 $ we need to check the confinement requirement. To this end, define a Poisson point process $ Ξ' $ made of the particles of $ Ξ $ at their positions at time $(τ+2)β_k$ that are confined during the time $ [τβ_k,(τ+2)β_k] $ inside $ \displ_{b(k-1)2^{\ell_{k-1}}} $. Using the definition of confinement from Lemma \ref{lemma-confining}, this happens for each particle independently with probability $\P(\Conf(\displ_{b(k-1)2^{\ell_{k-1}}},	2β_{k}	)) $. By the thinning property of Poisson processes, $ Ξ' $ is therefore a Poisson point process with intensity measure
	\begin{align*}
		 & \P\big(\Conf(\displ_{b(k-1)2^{\ell_{k-1}}},	2β_{k}	)\big)	(1-\mathfrak{d}_{k+1 })μ_0(1-\bar{ε})λ_y
	\end{align*}
	which we can estimate using \eqref{eq-confining} as being bigger than
	\begin{align*}
		 & \Big(	1-\Cr{conf1}e^{-\Cr{conf1}^{-1}		\big(\frac{( b(k-1)2^{\ell_{k-1}} )^{\dw}	}{2β_{k}}	\big)^{\frac{1}{\dw-1}}		 } 	\Big)
		(1-\mathfrak{d}_{k+1 })	μ_0 (1-\tfrac{ε}{8k^2})λ_y                                                                                                                           \\
		 & \leftstackrel{\eqref{def-β}}{=}\Big(	1-\Cr{conf1}e^{- \Cr{conf1}^{-1}		\big(\frac{b(k-1)^{\dw}}	{2\Cmix } ( \frac{ε}{k^2}	)^{4/Θ}\big)^{\frac{1}{\dw-1}}
		}\Big)			(1-\mathfrak{d}_{k+1 })	μ_0 (1-\tfrac{ε}{8k^2}) λ_y                                                                                                                 \\
		 & \leftstackrel{\eqref{def-b(k)}}{=}\Big(	1-\Cr{conf1}e^{- \C	\big(\frac{\big(a(k-1)^{2+\frac{8}{Θ\dw}}m2^{m}\big)^{\dw }   ε^{4/Θ}}	{k^{8/Θ}\Cmix }\big)^{\frac{1}{\dw-1}}
		}\Big)			(1-\mathfrak{d}_{k+1 })	μ_0 (1-\tfrac{ε}{8k^2}) λ_y.
	\end{align*}
	Next, taking advantage of $\Cmix=\tfrac{β}{2^{\dw\ell}}ε^{4/Θ}2^{m \dw }$ which can be obtained by setting $β_1=β $ in \eqref{def-β}, the right-hand side of the previous display is bigger than
	\begin{align*}
		\Big(	1-\Cr{conf1}e^{- \C	\big(
		\frac{2^{\dw\ell}}{β}a^{\dw } (k-1)^{2\dw }m 	\big)^{\frac{1}{\dw-1}}
		}\Big)			(1-\mathfrak{d}_{k+1 })	μ_0 (1-\tfrac{ε}{8k^2}) λ_y.
	\end{align*}
	Setting $ m $ large enough with respect to $\varepsilon$, $\ell$ and $\beta$, this again is then larger than
	\begin{align*}
		 & (1-\tfrac{ε}{8k^2}) (1-\mathfrak{d}_{k+1 })	μ_0 (1-\tfrac{ε}{8k^2})  λ_y 
	 \ge (1-\tfrac{ε}{4k^2}) (1-\mathfrak{d}_{k+1 })	μ_0 λ_y.
	\end{align*}
	Conditioning on the success of the above coupling, we obtain using a union bound that the probability that all $ S_{k-1}(i') $ contained in $ S_k\ext(ι) $ have at least $(1- \mathfrak{d}_{k})μ_0\sum_{y\in S_{k-1}(i')}λ_y  $ particles which are confined during $ [τβ_k,(τ+2)β_k] $ inside $ \displ_{b(k-1)2^{\ell_{k-1}}}, $  is at least
	\begin{equation}\label{eq-MixChernoff1}
		1-\sum_{S_{k-1}(i')\subseteq S_k\ext(ι)}		\mathbb{Q}\big(Ξ'(S_{k-1}(i')) \le (1-\mathfrak{d}_{k})μ_0\textstyle\sum_{y\in S_{k-1}(i')}λ_y\big).
	\end{equation}
	Using the Chernoff bound \eqref{eq-Chernoff} with $ χ $ given by
	\begin{align*}
		 & 1-\frac{	( 1-\mathfrak{d}_{k})	μ_0	\sum_{y\in S_{k-1}(i')}λ_y}
		{	(1-\tfrac{ε}{4k^2}) (1-\mathfrak{d}_{k+1 }) μ_0	\sum_{y\in S_{k-1}(i')}λ_y	}  \\
		 & = \frac{	(1-\tfrac{ε}{4k^2}) (1-\mathfrak{d}_{k+1 })-( 1-\mathfrak{d}_{k})	}
		{	(1-\tfrac{ε}{4k^2}) (1-\mathfrak{d}_{k+1 }) 	}                                \\
		 & \ge (1-\tfrac{ε}{4k^2}) (1-\mathfrak{d}_{k+1 })		-		( 1-\mathfrak{d}_k)      
		  \ge( \mathfrak{d}_{k} - \mathfrak{d}_{k+1})- \frac{ε}{4k^2} =\frac{ε}{4k^2},
	\end{align*}
	we obtain the following lower bound for \eqref{eq-MixChernoff1}:
	\begin{equation}\label{eq-MixingLemma-DisplProbability}
		\begin{aligned}
			 & 1- \sum_{S_{k-1}(i')\subseteq S_k\ext(ι)}       \exp
			\Big\lbrace		-\frac 12	(\tfrac{ε}{4k^2})^2(1-\tfrac{ε}{4k^2}) (1-\mathfrak{d}_{k+1 })	μ_0 \textstyle\sum_{y\in S_{k-1}(i')}λ_y\Big\rbrace \\
			 & \leftstackrel{\eqref{eq-SideLengthSk}}{\ge}	   1- \sum_{S_{k-1}(i')\subseteq S_k\ext(ι)}       \exp
			\Big\lbrace		-\frac{ε^2}{32k^4}(1-\tfrac{ε}{4}) (1-\mathfrak{d}_{2})	μ_0 		{C_{λ}}^{-1} (2^{\ell_{k-1}})^{\dv}\Big\rbrace                 \\
			 & \ge 1-  \CVol\big(b(k-1)+2^{\ell_k - \ell_{k-1}}\big)^{\dv}     \exp
			\Big\lbrace		-\frac{ε^2}{32k^4}(1-\tfrac{ε}{4}) (1-\tfrac{ε}{2})	μ_0 		{C_{λ}}^{-1} 2^{\dv\ell_{k-1}}\Big\rbrace                          \\
			 & \ge 1-\frac{1}{2}e^{-C_ψ ψ_k},
		\end{aligned}
	\end{equation}
	where the last inequality follows from the same argument as after \eqref{eq-MixingLemma-CouplingProbability} since $ ψ_1 $ is assumed large enough.

	Combining \eqref{eq-MixingLemma-CouplingProbability} and \eqref{eq-MixingLemma-DisplProbability} proves the claim for $ 1<k<κ. $

	For $ k=κ $ the argument is easier, as there is no need to use the decoupling theorem and one can simply use \eqref{eq-MixingLemma-DisplProbability}, and prove both the conditional and unconditional statements.

	For $ k=1, $ we recall that the event $ A_1(ι,τ) $ was defined differently (cf.\ \eqref{def-A1}) We use again the decoupling Theorem to obtain a coupling with a Poisson point process $ Ξ $ which succeeds with probability  \eqref{eq-MixingLemma-CouplingProbability} with the choice $ k=1. $  To obtain $ \ind_{E(ι,τ)}=1, $ we recall that the event $ E(ι,τ) $ is measurable with respect to the σ-algebra of particles inside $ S_1^η(ι) $, which is contained in $ S_1\base(ι)$ by Remark \ref{rmk-SηSbase}, and the particles are confined in $ B_{ζ\ell_1} $ during $ [τβ_1, (τ+η)β_1] $. Using Lemma \ref{lemma-confining} we obtain
	\begin{equation*}
		\P\big(\Conf (B_{ζ\ell_1},	ηβ_1)\big)
		\ge 1-\Cr{conf1}\exp \Big\{		-\Cr{conf1}^{-1}\big(	 (ζ\ell_1)^{\dw}/(ηβ_1)		\big)	^{\frac{1}{\dw-1}}		\Big\}
		\stackrel{\eqref{def-ζ}}{\ge} 1-\tfrac{3ε}{8}.
	\end{equation*}
	Hence, the Poisson point process $ \Xi' $ of particles satisfying $ \Conf (B_{ζ\ell_1},	ηβ_1) $ has intensity at least
	\begin{equation*}
		\P\big(\Conf(\displ_{b(k-1)2^{\ell_{k-1}}},	2β_{k}	)\big)	(1-\mathfrak{d}_{2})μ_0(1-\bar{ε})λ_y
		\ge (1-\tfrac{3ε}{8}) (1-\tfrac{ε}{2})  	μ_0 (1-\tfrac{ε}{8}) λ_y \ge (1-ε)μ_0λ_y,
	\end{equation*}
	and since $ E(ι,τ) $ is increasing, we have
	\begin{equation*}
		\P(\ind_{E(i,τ)}=1\given F,  D\base_1(ι,τ)=1)
		\le 1-ν_E\big((1-ε)λ, S_1^η, \displ_{ζ\ell}, ηβ_1\big) \le e^{-α_0},
	\end{equation*}
	which concludes the proof.
\end{proof}

Now that we have a bound on the probability that a single cell $ R_k(ι,τ) $ is multi-scale bad, we can obtain an upper bound on the probability that all multi-scale cells in a given ScD-path are multi-scale bad. Recall the definition of the weights $ ψ_k $ in \eqref{def-ψ} and the value $ α_0 $ defined in Lemma \ref{lemma-Mixing}.

\begin{corol} \label{corol-MixingPaths} Let $ ζ $ as in \eqref{def-ζ}, $ ψ_1>α_0$ and consider an ScD-path $ \{ R_{k_1}(ι_1,τ_1), \dots, R_{k_z}(ι_z,τ_z)\} $. Then
	\begin{equation*}
		\P \Big(	\bigcap_{j=1}^{z} \{A_{k_j} (ι_j,τ_j)=0 \}	\Big)				\le 			e^{-C_ψ\sum_{j=1}^{z} ψ_{k_j}}.
	\end{equation*}
	where $ C_ψ $ is the constant from Lemma \ref{lemma-Mixing}.
\end{corol}

\begin{proof}
	We first need to order the cells in a temporal order. To this end, consider any order $ \prec $ of the indices of the cells $ {1,\dots ,z} $ such that if $ j_1 \prec j_2 $ then $ τ_{j_1} β_{k_{j_1}}  \le   τ_{j_2} β_{k_{j_2}}  $. The corollary will be a simple consequence of Lemma \ref{lemma-Mixing} once we prove that for every $ 1\le \bar{j} \le z $, the cells $  R_{k_j}(ι_j,τ_j) $ with $ j\prec \bar{j} $ are $ \mathcal{F}_{k_{\bar{j}}}(ι_{\bar{j}}	,	τ_{\bar{j}}) $-measurable.

	We therefore consider two cells $ R_{k_{j_1}}(ι_{j_1},τ_{j_1}) $ and $R_{k_{j_1}} (ι_{j_2},τ_{j_2})  $ with $ j_1\prec j_2 $, so that $ τ_{j_1}β_{k_{j_1} }		\le τ_{j_2} β_{k_{j_2}} $. By definition of an ScD-path cells are well-separated, so $ R\inbase_{k_{j_1}} (ι_{j_1}, τ_{j_1})  	\cap	R\inbase_{k_{j_2}} (i_{j_2}, τ_{j_2})  =\emptyset,  $ meaning that:
	\begin{itemize}
		\item  either $ T_{k_{j_1}}\inbase (τ_{j_1})	\cap T_{k_{j_2}}\inbase (τ_{j_2})	=\emptyset  $ and thus \eqref{item-Fka} is satisfied;
		\item or $ S\inbase_{k_{j_1}} (ι_{j_1}) \cap S\inbase_{k_{j_2}} (ι_{j_2}) =\emptyset$ and thus \eqref{item-Fkb} is satisfied.
	\end{itemize}
	Here, \eqref{item-Fka} and \eqref{item-Fkb} are as they appear at the beginning of this subsection.
	Hence, using the standard chain conditioning and applying Lemma \ref{lemma-Mixing}  $ z $-many times we obtain that
	\begin{equation*}
		\P \Big(	\bigcap_{j=1}^{z}	\big\{A_{k_j} (ι_j, τ_j) =0\big\} \Big)
		\le 		\prod_{j=1}^{z}		\P \Big(A_{k_j} (ι_j, τ_j)=0
		\Given 	 	\bigcap_{\bar{j}\prec j}	\big\{A_{k_{\bar{j}}} (ι_{\bar{j}}, τ_{\bar{j}}) =0	\}		\Big)
		\le e^{-C_ψ\sum_{j=1}^{z} ψ_{k_j}},
	\end{equation*}
	which is the desired claim.
\end{proof}

\subsection{Number of ScD-paths}

In the previous section we established the probability for a given path of $ z $ cells of scales $ k_1,\dots,k_z $ to be made of multi-scale bad cells. We want now to count the number of such paths. Recall the definition of ScD-path in Definition \ref{def-ScDpath}, and of $ Ω_κ\Sup(v\rightarrow t) $ in \eqref{def-Ωtsup}. We will now give an upper bound for the number of paths in $ Ω_κ\Sup(v\rightarrow t) $, given a fixed number of cells and their scales. As we will see, $κ$ and $t$ are going to be linked with each other; for the moment, we work with given scales and so we omit stating either $\kappa$ or $t$ in our first bound below.

\begin{lemma}\label{lemma-Combinatorial}
	For a fixed length $ z \in\N$, fixed scales $ k_1,\dots, k_z $ and $ v\in L_1 $, the number of ScD-paths 
	of cells with scales $k_1,\dots,k_z,$ where the extended support of the first cell contains $v ,$ is at most
	\begin{equation*}
		\exp \Big\{{\dfrac{C_ψ}{2} \sum_{j=1}^{z} ψ_{k_j} 	}	\Big\},
	\end{equation*}
	where $ C_ψ $ is the same constant as in Lemma \ref{lemma-Mixing}.
\end{lemma}
\begin{proof}
	Recall that two consecutive cells $ R_{k_1}(ι_1,τ_1) $ and $ R_{k_2}(ι_2, τ_2 )$ in a ScD-path are either support adjacent or $ R_{k_1}(ι_1,τ_1) $ is support connected with diagonals to $ R_{k_2}(ι_2, τ_2 )$. We will prove the result in three steps: first, we will bound the number of ScD-paths where each cell is support adjacent to the next one, i.e.\ we do not allow diagonal connections. In the second step, we will show the result for the case in which the beginning and end of the (scale-one) diagonal steps are fixed relative to each other; in the third step we will obtain the bound where this last restriction is removed.

	\textit{Step 1.} We define the maximum number of scale $ k' $ cells which are support adjacent to a cell of scale $ k $
	\begin{equation}\label{def-Phi}
		Φ_{k,k'}:= \max_{(ι,τ)} \big|
		\{R_{k'}(ι',τ')\colon		R_k(ι,τ) \text{ is support adjacent to }R_{k'} (ι',τ')\}\big|
	\end{equation}
	and the number of cells of scale $ k $ whose extended support (defined in \eqref{def-SEsup} and \eqref{def-TEsup}) contains $v $ as
	\begin{equation}\label{def-Chi}
		χ_k:= \big| \{R_k(ι,τ)\colon R_k\Esup(i,τ) \supseteq v\}\big|.
	\end{equation}
	Hence, clearly the number of support adjacent only D-paths in $Ω_κ\Sup(v\rightarrow t) $ of cells with scales $ k_1,\dots, k_z $ is upper bounded by
	\begin{equation*}
		χ_{k_1} \prod_{j=2}^{z} Φ_{k_{j-1},k_j}.
	\end{equation*}

	We start by deriving a bound for $ χ_k $. Since the extended support of a cell of scale $ k $ contains at most $27\CVol(3m+1)^{\dv}  $ cells of scale $ k+1 $, there exist at most $27 \CVol(3m+1)^{\dv} $ different extended supports of a cell of scale $ k $ that contain the distinct cell of scale $k+1$ containing $ v $, and thus $v$ itself. By \eqref{eq-EnumSmallerCells} and \eqref{eq-rateBeta} each cell of scale $ k+1 $ contains $ \frac{β_{k+1}}{β_k} 2^{\dv (\ell_{k+1}-\ell_{k})}\le 	2^{8+\dw (2ak-3a+m)+\dv (2ak-a+m)}$ cells of scale $ k $, which is therefore also the number of scale $k$ cells that share the same extended support. We therefore have
	\begin{equation}\label{eq-BoundChi}
		χ_k \le 27\CVol(3m+1)^{\dv}		2^{8+\dw (2ak-3a+m)+\dv (2ak-a+m)}\le \exp \Big	\{\frac{C_ψ}{16} ψ_k	\Big\},
	\end{equation}
	where the last inequality holds trivially for $ m,a $ and $ α_0 $ large enough.

	We now bound $ Φ_{k,k'}. $ A cell of scale $ k' $ can only be support adjacent to a cell $R_k (ι_1,τ_1) $ if it is inside $B_r(p)\times A$, where $p\in S_k(ι_1) $, $r:= (3m+2)2^{\ell_{k+1}}+(3m+2)2^{\ell_{k'+1}} $ and $A$ an interval centered around $ T_k(τ_1) $ of width $ 28(β_{k+1}+β_{k'+1}). $ Consequently, $ Φ_{k,k'} $ can be bounded by the number of scale $ k' $ cells inside this Cartesian product. If $ k\ge k' $ then the terms $ 2^{\ell_{k'+1}} $ and $ β_{k'+1} $ are negligible (or of the same size) in comparison to $ 2^{\ell_{k+1}} $ and $ β_{k+1} $, and the spatial region contains at most $ \CVol(2(3m+2))^{\dv} $ cells of scale $ k+1 $, and by \eqref{eq-EnumSmallerCells}, each one of those contains exactly $ 2^{\dv(\ell_{k+1}-\ell_{k'})} $ cells of scale $k'$, so
	\begin{align*}
		\text{if } k\ge k', \text{ then } &  & \Phi_{k,k'}\le \Big(\CVol(2(3m+2))^{\dv}		2^{\dv (\ell_{k+1}-\ell_{k'})}\Big)\Big(56\frac{β_{k+1}}{β_{k'}}\Big).
	\end{align*}
	If instead $ k<k' $ we have similarly
	\begin{align*}
		\text{if } k<k', \text{ then } &  & \Phi_{k,k'}\le \Big(\CVol(2(3m+2))^{\dv}			2^{\dv (\ell_{k'+1}-\ell_{k'})}\Big)\Big(56\frac{β_{k'+1}}{β_{k'}}\Big).
	\end{align*}
	Combining the two and using \eqref{eq-rateBeta} we have that
	\begin{align*}
		\Phi_{k,k'} \le & \C 2^{\dv(6m+4)} 			2^{\dv \left(a(k \vee k')^2+m(k \vee k')\right)} 2^{\dw 2a(k \vee k')+\dw m }
	\end{align*}
	and for $ a,m,α_0 $ large it holds trivially that this is further smaller than
	\begin{align*}
		\exp\Big\{\frac{C_ψ}{16}ψ_{(k \vee k')} \Big\}.
	\end{align*}
	Hence we obtain with \eqref{eq-BoundChi}
	\begin{equation*}
		χ_{k_1} \prod_{j=2}^{z} Φ_{k_{j-1},k_j}
		\le \prod_{j=1}^z \Big(	e^{\frac{C_ψ}{16} ψ_{k_j}}	\Big)^2
		\le e^{\frac{C_ψ}{8} \sum_{j=1}^{z}ψ_{k_j}}.
	\end{equation*}

	\textit{Step 2.} In this step, we consider $ R_{k_1}(ι_1,τ_1) $ to be support connected with diagonals to $ R_{k_2}(ι_2,τ_2)$, which, as defined in Definition \ref{def-ScDpath}, means that there exist two cells $ R_1(\tilde{ι}_1,\tilde{τ}_1) $ and $ R_1(\tilde{ι}_2,\tilde{τ}_2)  $ contained in their respective extended supports such that $ R_1(\tilde{ι}_1,\tilde{τ}_1) $ is diagonally connected to $ R_1(\tilde{ι}_2,\tilde{τ}_2)  $. We denote by $(\tilde{ι}_1- \tilde{ι}_2,\tilde{τ}_1-\tilde{τ}_2) $ the relative position of the cell $ R_1(\tilde{ι}_1,\tilde{τ}_1) $ with respect to $ R_1(\tilde{ι}_2,\tilde{τ}_2)$ and write $(0,0)$ for the relative position of the cells $ R_{k_1}(ι_1,τ_1) $ and $ R_{k_2}(ι_2,τ_2)  $ when they are adjacent. In this step we consider the relative positions to be fixed, and we will show a bound for the number of different possible relative positions in the next step. In analogy with step 1, we define
	\begin{equation}\label{def-PhiStar}
		Φ^*_{k_1,k_2}:= \max_{(ι_1,τ_1)}
		\Bigg|	\Bigg\{
		R_{k_2}(ι_2,τ_2)\colon
		\begin{array}{l}
			R_{k_1}(ι_1,τ_1) \text{ is support adjacent or support connected} \\
			\text{with diagonals to }			R_{k_2}(ι_2,τ_2)
			\text{ with fixed relative }                                      \\
			\text{position of $ R_1 (\tilde{ι}_1,\tilde{τ}_1) $ with respect to $ R_1(\tilde{ι}_2,\tilde{τ}_2) $}
		\end{array}
		\Bigg\}\Bigg|.
	\end{equation}

	The case when the relative position is $(0,0)$ was treated in the previous step, so in that case we have
	\begin{equation*}
		Φ^*_{k_1,k_2} \le e^{\frac{C_ψ}{16}ψ_{k_1\vee k_2}	}.
	\end{equation*}

	In the case of diagonally connected cells, since the relative position is fixed, the possible combinations are determined by the product of all the possible positions of the cell $ R_1(\tilde{ι}_1,\tilde{τ}_1) $ inside the extended support of the cell $ R_{k_1}(ι_1,τ_1 )$ and the number of cells of scale one contained in the extended support of $   R_{k_2}(ι_2,τ_2 )$. Using the bound from the previous step we have
	\begin{equation*}
		Φ^*_{k_1,k_2} \le e^{\frac{C_ψ}{16}ψ_{k_1}	}	 e^{\frac{C_ψ}{16}ψ_{k_2}	}.
	\end{equation*}
	Combining the two equations yields
	\begin{equation*}
		Φ^*_{k_1,k_2} \le e^{\frac{C_ψ} {16}ψ_{k_1}	}	 e^{\frac{C_ψ}{16}ψ_{k_2}	}+e^{\frac{C_ψ}{16}ψ_{k_1\vee k_2}	}.
	\end{equation*}
	Hence the number of ScD-path where the $ z $ cells have fixed relative position is bounded by 
	\begin{equation}\label{eq-DpathFRelPos}
		χ_{k_1} \prod_{j=2}^{z} Φ^*_{k_{j-1},k_j} \le \exp\Big\{ \frac{C_ψ}{4} \sum_{j=1}^{z} ψ_{k_j} \Big\}
	\end{equation}
	\textit{Step 3.} In the final step, we bound the number of combinations of different relative positions in a ScD-path.
	For two given cells of scales $k_j$ and $k_{j+1}$ where the first is support connected with diagonals to the second, let $ R_1(ι_1,τ_1) $ and $ R_1(ι_2,τ_2) $ be the corresponding two scale-one cells for which $ R_1(ι_1,τ_1) $ is diagonally connected to $ R_1(ι_2,τ_2) $ with relative position $(ι_1-ι_2,τ_1-τ_2)$.
	Let $ h $ be the (absolute) difference between the distances of $ R_1(ι_1,τ_1) $ and $ R_1(ι_2,τ_2) $ from $ L_0 $, which we refer to as ``difference in height''; see the discussion below \eqref{def-L1}.
	Define $A(h)$ to be the number of cells that $ R_1(ι_1,τ_1) $ can be diagonally connected to, where the ``difference in height'' is $h$. More precisely, define
	\[
		A(h):=\max_{(ι_1,τ_1)}\bigg| \Big\{R_1(ι_2, τ_2)		\colon \begin{array}{l}
			R_1(ι_1,τ_1)	\text{ is diagonally connected to }R_1(ι_2,τ_2) \\
			\text{ with } \big|d\left(L_0, R_1(ι_1,τ_1)\right)  -  d\left(L_0,R_1(ι_2,τ_2)\right)\big|= h
		\end{array}			\Big\}\bigg|.
	\]
	As defined, $A(h)$ is also an upper bound on the number of different relative positions $(ι_1-ι_2,τ_1-τ_2)$ which result in a height difference of $h$.


	We next note that, by definition of the diagonal steps, we can bound $ A(h) $ by the number of cells of scale one at distance $h$ from a given cell of scale one. Recalling  \eqref{def-dv}, we can therefore use the very generous bound
	\begin{equation}\label{eq-A(x)estimate}
		A(h) <  \CVol h^{\dv+1},
	\end{equation}
	where the $+1$ term comes from having to also consider the time dimension.

	Recall from Subsection \ref{SUBSEC-Paths} that when a scale-one cell is diagonally connected to another scale-one cell, the height of the second cell can be at most that of the first cell. We can thus obtain easily an upper bound on the number of diagonal steps and equivalently on the total height difference. Define $ H_k $ as the side length of $ S\Esup_k $ divided by the side length of $ S_1$, that is
	\begin{equation}\label{def-Hk}
		H_k:=(3m+1)2^{ak^2+mk}.
	\end{equation}

	Then, using that a diagonal step by definition leads to a decrease of the distance to $L_0$, the maximum number of diagonal steps in an ScD-path of cells of scales $k_1,\dots,k_z$ is at most the combined distance from $L_0$ that the cells of scales $k_1,\dots,k_z$ can contribute to an ScD-path, i.e.
	\begin{equation*}
		H=\sum_{i=1}^{z} H_{k_i}.
	\end{equation*}

	Hence, the number of different configurations of the diagonal steps, and in particular different relative positions, is at most
	\begin{align*}
		 & \sum_{l=0}^{H}\sum_{\substack{h_2,\dots h_z \\ h_2+\dots+h_z=l}} A(h_2+1)A(h_3+1)\dots A(h_z+1),
	\end{align*}
	where $ h_i $ represent the (absolute) height difference between the $ i$-th and $ (i-1 $)-th cell; the $+1$ accounts for the fact that the final scale-one cell of a diagonal connection might be adjacent and not equal to the next cell of the path, as per definition of being diagonally connected. Using the method of Lagrange multipliers, this is smaller than
	\begin{align*}
		 & \sum_{l=0}^{H}\sum_{\substack{h_2,\dots h_z \\ h_2+\dots+h_z=l}} \bigg(	A\Big( \frac{l}{z-1}+1 \Big)	\bigg)^{z-1}.
	\end{align*}
	Using \eqref{eq-A(x)estimate} and that the total number of combinations of $z-1$ values $ h_i \ge 0$ which sum to $ l $ is $\binom{l+z-2}{ z-2}$, this is smaller still than
	\begin{align*}
		 & \sum_{l=0}^{H}			\binom{l+z-2} { z-2} \CVol	\big( \tfrac{l}{z-1}+1 \big)^{(z-1)(\dv+1)} \\
		 & \le \sum_{l=0}^{H}			\binom{l+z-1}{ z-1 }
		\CVol \big( \tfrac{l}{z-1}+1 \big)^{(z-1)(\dv+1)}
	\end{align*}
	and using repeatedly Pascal's rule we can further bound this by
	\begin{align*}
		 & \binom{z+H}{ z}
		\CVol\big( \tfrac{H}{z-1}+1 \big)^{(z-1)(\dv+1)} \\
		 & \le 	\frac{(z+H)^z}{z!}
		\Cl{Pasc}	\big( \tfrac{H}{z-1}+1 \big)	^{(z-1)(\dv+1)}.
	\end{align*}
	Since $ \frac{H}{z} $ is big by the assumption that $\psi_1$ is large enough,
	we finally get that this is smaller than
	\begin{align*}
		 & \frac{(z+H)^z}{(z/3)^z}
		\Cr{Pasc} \big( \tfrac{3H}{z} \big)	^{(z-1)(\dv+1)}            \\
		 & \le		(3+3H/z)^z	\Cr{Pasc}\big( \tfrac{3H}{z} \big)^{z(d+1)} \\
		 & \le		\big(	\Cl{p}\tfrac{H}{z}	\big)^{2z(d+1)}
	\end{align*}
	for some constant $\Cr{p}>0$ depending only on $ d $; we used in the first inequality that $\dv\le d$, which is a simple consequence of the fact that the graph can be embedded into the $d$-dimensional triangular lattice which has volume growth dimension $d$. To obtain that $ (\Cr{p}H/z)^{2z(d+1)} \le \exp(\frac{C_ψ}{8} \sum_{j=1}^{z}ψ_{k_j}) $ and thus to conclude \emph{Step 3} and the proof, we can equivalently show that
	\begin{equation}\label{eq-Combinatorial-h<ψ}
		(d+1)(\log(\Cr{p}H/z) \le \frac{1}{z}\frac{C_ψ}{8} \sum_{j=1}^{z}ψ_{k_j}.
	\end{equation}
	Comparing $ H_k $ from \eqref{def-Hk} and $ ψ_k $ from  \eqref{def-ψ} and setting $ m$ and $α_0 $ (and thus $ \ell $) large enough we can obtain $ H_k\le \frac{C_ψ}{8(d+1)\Cr{p}}ψ_k $ for all $ k $, and therefore \eqref{eq-Combinatorial-h<ψ} holds.
\end{proof}

In the previous two lemmas, we showed the relationship between ScD-paths and the sum of the weights $ ψ_k. $ We show now that if we consider an ScD-path in $  Ω\Sup_{κ}(v\rightarrow t)  $ (defined in \eqref{def-Ωtsup}) of cells of scales $ k_1,\dots,k_z $ for some $ t>0 $, then the sum of the weights $ ψ_k $ is at least of order $ t^{c_s} $.

\begin{lemma} \label{lemma-φandt}
	Suppose that the largest scale $κ$ we consider satisfies $ κ=\mathcal{O}\big(\sqrt{\log(t)}\big).$ Then, if $ψ_1$ is large enough, there exist $t_0 $ and $\Cl{t}>0$ 
	such that for any $ t> t_0 $, $ v\in L_1 $ and any path $ 
		\big\{ R_{k_j}(ι_j,τ_j)\big\}_{j=1}^{z} \in Ω_{k}\Sup(v\rightarrow t)$
	\begin{equation*}
		\sum_{j=1}^{z} ψ_{k_j} \ge		 \Cr{t}t^{c_s},
	\end{equation*}
	where the positive constant $c_s$ is as defined in Theorem \ref{thm-main2}.
\end{lemma}
\begin{proof}
	Let $ \diam_k $ denote the diameter of the extended support of a cell of scale $ k $.

	The key observation to prove the lemma is that
	\begin{equation}\label{eq-diam>t}
		\sum_{j=1}^{z} \diam_{k_j}
		\ge \frac{t}{2},
	\end{equation}
	since by definition of $ Ω_κ\Sup(v\rightarrow t) $ in \eqref{def-Ωtsup} the path exits from $ B_{t}(S_1(ι_v))\times [-t+τ_v,τ_v+t] $ 
	and with an argument similar to the one surrounding \eqref{def-Hk}, the distance that can be covered by diagonal steps is at most the sum of the side lengths of the cells. Therefore, we only need to compare $ \diam_{k} $ with $ ψ_k. $

	For the geometry of the fractal, the diameter of the tile is equal to the side length; hence, for $ 1\le k\le κ $, we note that
	\begin{align*}
		\diam_{k} & \le  (6m+3) 2^{\ell_{k+1}} + 27β_{k+1}                      \\
		          & \le (6m+3)2^{ak+a+m}2^{\ell_{k}}+(\Cmix 2^{\ell_{k}})^{\dw} \\
		          & \le \Cl{diago} 2^{2m} \ 2^{ak}2^{\dw\ell_{k}}               \\
		          & \le \Cr{diago} 2^{2m+ak}2^{(\dv+1)\ell_{k}},
	\end{align*}
	where in the last step we made use of \eqref{eq-dvdw}.
	For $ k\ge 2, $
	\begin{align*}
		ψ_k & =\frac{ε^2 μ_0 2^{\dv \ell_{k-1}}	} { k^4}                                                 
		 =\frac{ε^2 μ_0 2^{\dv \ell_{k}}	} { k^4	2^{\dv (ak-a+m)}}                                                \\
		    & =\frac{ε^2 μ_0 } { k^4 2^{\dv( ak-a+m)}}\frac{1}{\big( \Cr{diago} 2^{2m+ak}  \big)^{\frac{\dv}{\dv+1}}	}
		\big( 	\Cr{diago} 2^{2m+ak}	2^{(\dv+1)\ell_{k}}\big)^{\frac{\dv}{\dv+1}}                                       \\
		    & \ge\frac{ε^2 μ_0 } { \Cl{diago2}k^42^{\dv(ak+2m)}}
		\big(\diam_k\big)^{\frac{\dv}{\dv+1}}	.
	\end{align*}

	For $ k=1 $ we can fix a constant $ \Cl[s]{diagoScale1} >0$ depending on $ ε,μ_0,a,m, \ell $ and $ ν_E $, but crucially not on $t$, such that $ ψ_1\ge \Cr{diagoScale1} (\diam_1)^{\dv/(\dv+1)}.$

	Since we assumed that $ κ=\mathcal{O}(\sqrt{\log(t)}) $, we have that there exists $ \Cl[s]{bigO} $ such that $ k\le \Cr{bigO} \sqrt{\log(t)}$ for all $ k\le κ $ and thus summing over all cells of the path, \eqref{eq-diam>t} gives
	\begin{align*}
		\sum _{j=1}^z ψ_j \ge \C  \frac{ε^2 μ_0 } { 	\log^2(t)  2^{\dv a\sqrt{\log(t)}	}		}
		t^{   \frac{\dv}{\dv+1} }
	\end{align*}
	which for $ t $ large is larger than $ \Cr{t}t^{c_s}$.
\end{proof}

\subsection{Size of bad clusters}

Let $ t>0 $ be large, $ v\in L_1 $ and define 
\begin{equation*}
	\mathbf{S}_k^t(v):=\left\{S_k(ι')\colon ι'\in \mathbb{B}^d,\ S_k(ι') \cap B_{t}(S_1(ι_v)) \neq \emptyset\right\},
\end{equation*}
\begin{equation*}
	\mathbf{T}_k^t(v):=\left\{T_κ(τ')\colon τ'\in \mathbb{Z},\ \exists \bar{τ}\in \Z \colon γ_{1}^{(k-1)}(\bar{τ})=τ',\ T_1(\bar{τ})\cap [τ,τ+t]\neq \emptyset\right\}
\end{equation*}
and
\begin{equation*}
	\mathbf{R}_k^t(v):= \left\{ S\times T\colon S\in \mathbf{S}_k^t (v), T\in  \mathbf{T}_k^t (v)\right\},
\end{equation*}
where $ι_v, τ_v$ and $B_t(S_1(ι_v))$ are as defined previously below \eqref{def-Ωt}.
Recall also the definition of the bad cluster $ K_v $ from \eqref{def-K}.

\begin{prop}\label{prop-KnotinR}
	Let $ ζ $ as in \eqref{def-ζ}, $ α_0 $ as in Lemma \ref{lemma-Mixing} and $ t_0 $ as in Lemma \ref{lemma-φandt}. Then there exists a constant $ \Cr{K} $ independent of $ t $ such that for any $ v\in L_1 $ we have
	\begin{equation}\label{eq-KnotinR}
		\P\big(K_v \nsubseteq \mathbf{R}_1^t(v)\big) \le  e^{-\Cl{K} t^{c_s} },
	\end{equation}
	for all $ t>t_0. $
\end{prop}
\begin{proof}
	Using Lemma \ref{lemma-Bad=>BadAnc} implies
	\begin{align*}
		\P\big (K_v \nsubseteq \mathbf{R}_1^t(v) \big)
		 & \le \P\big (\exists P\in  Ω_1(v\rightarrow t) \text{ of bad cells}\big)              \\
		 & \le \P\big(\exists P\in  Ω_1(v\rightarrow t) \text{ of cells with bad ancestry}\big),
	\end{align*}
	and by Lemma \ref{lemma-Ω1Ωsup} this is smaller than
	\begin{align*}
		\P \big(\exists P \in Ω_κ\Sup(v\rightarrow t) \text{ of multi-scale bad cells}\big),
	\end{align*}
	for any arbitrary choice of $κ$; we will fix it momentarily.

	Define now the event $H_\kappa$ to be the event that $A_\kappa(\iota,\tau)$ holds for all cells in $\mathbf{R}_\kappa^t(v)$, i.e.
	\[
		H_\kappa:=\bigcap_{R_κ(ι,τ)\in \mathbf{R}_κ^t(v)} \{A_κ(ι,τ)=1 \}
	\]
	Recalling how the event $ A_κ(ι,τ) $ is defined in \eqref{def-Aκ} for the largest scale $ κ $, using a union bound and Lemma \ref{lemma-Mixing} we obtain directly that
	\begin{align*}
		\P\big(H_κ(v)\big)
		\ge 1 - \big|\mathbf{R}_κ^t(v)\big|   e^{-C_ψ 	ψ_κ	}.
	\end{align*}
	We choose now $ κ $ to be the smallest integer such that $  ψ_κ \ge t.$ Using the definition of $ ψ_k $ in \eqref{def-ψ} one can see that $ κ=\mathcal{O}(\sqrt{\log(t)}) $; note that this choice satisfies the assumption of Lemma \ref{lemma-φandt}.
	Since the cardinality of $ \mathbf{R}_κ^t(v) $  satisfies
	\begin{equation*}
		\big|\mathbf{R}_κ^t(v)\big|			\le \C
		\Big( \frac{t}{2^{\ell_k}} \Big)^{\dv}
		\Big( \frac{t}{β_k} \Big),
	\end{equation*}
	we can use this to find some constant $\Cr{κgood} $ such that
	\begin{equation*}
		\P\big(H_κ(v)\big) \ge 1- e ^{\Cl[s]{κgood}	t}.
	\end{equation*}
	We now continue the previous chain of inequalities
	\begin{equation}
		\begin{aligned}
			 & \P( \exists P\in Ω_{κ}\Sup(v\rightarrow t) \text{ of multi-scale bad cells})                                                        \\
			 & \le   \P( \exists P\in Ω_κ\Sup(v\rightarrow t) \text{ of multi-scale bad cells}		\cap 	H_κ (v)	) + \P \big(H_κ(v)^{\mathsf{c}}\big)
			\\
			 & \le  \P( \exists P\in Ω_{κ-1}\Sup(v\rightarrow t) \text{ of multi-scale bad cells})   + e^{-\Cr{κgood} t }.
		\end{aligned}
	\end{equation}
	Since $c_s<\tfrac{\dv}{\dv+1}-\frac{1}{2}<1$, the term $e^{-\Cr{κgood} t }$ is of a smaller order than the claimed bound of $e^{-\Cr{K} t^{c_s} }$ and we can ignore it going forward.

	We now want to bound the remaining probability. If we fix the length of the path $ z\in \N $ and the scales $ k_1,\dots, k_z $ we can use Corollary \ref{corol-MixingPaths} and Lemma \ref{lemma-Combinatorial} to obtain
	\begin{align*}
		 & \P\big ( \exists P\in Ω_{κ-1}\Sup(v\rightarrow t) \text{ of $ z $ multi-scale bad cells of scales }k_1,\dots, k_z\big)                                           \\
		 & \le \ e^{-C_ψ \sum_{j=1}^{z} ψ_{k_j}}e^{\frac{C_ψ}{2} \sum_{j=1}^{z} ψ_{k_j}} = e^{-\frac{C_ψ}{2} \sum_{j=1}^{z} ψ_{k_j}} \le e^{- \frac{C_ψ}{2}\Cr{t}t^{c_s} },
	\end{align*}
	where the last step follows from Lemma \ref{lemma-φandt} since $κ$ and therefore also $κ-1=\mathcal{O}(\sqrt{\log(t)})$.

	It only remains to estimate the number of different possible lengths and weights of a path. We rewrite the weight of a path as the sum of the weights of cells of different scales, namely $ \sum_{j=1}^{z} ψ_{k_j}  = \sum_{k=1}^{κ-1} h_k ψ_k$, where $ h_k $ is the number of cells of scale $ k. $ Hence, for fixed $ h_1,\dots, h_{κ-1} $, the number of possible ways to order the cells is
	\begin{equation}\label{eq-multichoose}
		\frac{(h_1+\dots + h_{κ-1})!}{h_1! h_2!\dots h_{κ-1}!}
		= \binom{h_1+\dots +h_{κ-1}}{h_1}
		\binom{h_2+\dots +h_{κ-1}}{h_{2}}
		\dots \binom{h_{κ-1}}{ h_{κ-1}}.
	\end{equation}
	By the bounds provided by Lemma \ref{lemma-φandt}, there exists  $ k\in \{1,\dots, {κ-1}\} $ such that  $h_k\ge \frac{  \Cr{t}t^{c_s} } {(κ-1)ψ_{k}} $. Define now
	\begin{equation*}
		\mathcal{H}:=\Big\{(h_1,\dots,,h_{κ-1})\in (\N_0)^{κ-1}\colon \exists l\in \{1,\dots, {κ-1}\}\ h_l\ge \tfrac{  \Cr{t}t^{c_s} } {(κ-1)ψ_{l}} \Big\}.
	\end{equation*}
	We can then write
	\begin{align*}
		 & \P( \exists P\in Ω_{κ-1}\Sup(v\rightarrow t) \text{ of multi-scale bad cells}) \\
		 & \le	\sum_{	\mathcal{H}}
		\P\Big(     \exists P\in Ω_{κ-1}\Sup(v\rightarrow t)\colon
		\begin{array}{l}
			\text{such that for each $ k=1,\dots κ-1$, $P$ is made } \\ \text{of $ h_k $ multi-scale bad cells of scale $ k $}
		\end{array}
		\Big)                                                                             \\
		 & \le \sum_{	\mathcal{H}}
		e^{-\frac{C_ψ }{2}\sum_{k=1}^{κ-1} h_k ψ_{k}}			\frac{(h_1+\dots + h_{κ-1})!}{h_1! h_2!\dots h_{κ-1}!}.
	\end{align*}
	Applying \eqref{eq-chooses} $ κ-1 $ times to the right-hand side of \eqref{eq-multichoose} we can bound this further by
	\begin{align*}
		\sum_{	\mathcal{H}}
		e^{-\sum_{k=1}^{κ-1} h_k (\frac{C_ψ }{2}ψ_{k}-k)},
	\end{align*}
	and using that $ α_0,a,m  $ are large enough twice, this is finally smaller than
	\begin{align*}
		\sum_{	\mathcal{H} }
		e^{-		\frac{C_ψ }{3}  \sum_{k=1}^{κ-1} h_k ψ_{k}}		\le  e^{-\Cr{K} t^{c_s} },
	\end{align*}
	which concludes the proof.
\end{proof}

\subsection{Proof of Theorem \ref{thm-main}}\label{SUBSEC-ProofThm1}

\begin{proof}
	By Proposition \ref{prop-RadHill} we need to show for all $ v\in L_1 $ that
	\begin{equation*}
		\sum_{r\ge 1} r^{\dv+1} P(\rad_v(\Hill_v) > r) <∞.
	\end{equation*}

	Recalling the definition of $\mathbf{R}^t_k(v)$ above Proposition \ref{prop-KnotinR} and letting $v=R_1(ι_v,τ_v)$, we note that $ \mathbf{R}_1^t (v)$ contains only cells $R_1(ι',τ')$   with $ d(S_1(ι_v),S_1(ι'))\le \frac{t}{2^\ell } $ and $ |τ_v-τ'|\le \frac{t}{β}$. Hence, if $ r,T $ satisfy
	\begin{equation*}
		T \Big(	\frac{1}{2^\ell} +\frac{1}{\Cl{T} 2^{\dw \ell}}	\Big) \le r
	\end{equation*}
	for some constant $\Cr{T}$, it holds that
	\begin{equation*}
		\mathbf{R}_1^{T} (v)\subseteq \big\{R_1(ι',τ')\colon (ι',τ')\in \mathbb{B}^d \times \Z	,\ d\big( R_1(ι',τ'),v \big)		\le r	\big\}.
	\end{equation*}
	Define therefore $ T(r) :=$ $  (	\frac{1}{2^\ell} +\frac{1}{\Cr{T} 2^{\dw \ell}}	) ^{-1}  r $, and let $ t_0 $ be as in Lemma \ref{lemma-φandt} and $ r_0 $ such that $ T(r_0)>t_0. $
	Then
	\begin{align*}
		\sum_{r\ge r_0} r^{\dv+1} P(\rad_v(\Hill_v) > r)
		 & \le \sum_{r\ge r_0} r^{\dv+1} P\left(\Hill_v \nsubseteq \mathbf{R}_1^{T(r)}(v)  \right) \\
		 & \le  \sum_{r\ge r_0} r^{\dv+1} P\left(K_v \nsubseteq \mathbf{R}_1^{T(r)}(v)   \right)   \\
		 & \leftstackrel{\eqref{eq-KnotinR}}{\le}
		\sum_{r\ge r_0}		r^{\dv+1} 			\exp\big\{-\Cr{K} T(r)^{c_s}	\big\}.
	\end{align*}
	Since this series converges, the \Lip{} exists almost surely as stated in Proposition \ref{prop-RadHill}.
\end{proof}

\section{Proof of Theorem \ref{thm-main2}}\label{SEC-ProofThm2}
\newcommand{\DD}	{D\rotatebox[origin=c]{180}{D}-path}
\newcommand{\SDD}	{ScD\rotatebox[origin=c]{180}{D}-path}
\newcommand{\OmegaD}{Ω\rotatebox[origin=c]{180}{D}}

The main tool for the proof of Theorem \ref{thm-main2} are \DD s, which we define next. They are in essence a symmetric version of D-paths, in the sense that diagonal connections can go ``backwards''; equivalently, being connected by a \DD{} is a symmetric relationship unlike before.

\subsection{\texorpdfstring{\DD s}{DD}	}\label{SUBSEC-DD}

For the Reader's convenience we recall here  from Subsection \ref{SUBSEC-Paths} some definitions that we will build on next. We say for two scale-one cells $ R_1(ι,τ) $ and $ R_1(ι',τ') $ that $ R_1(ι,τ) $ is \emph{diagonally connected} to $ R_1(ι',τ') $ if there exists a sequence of scale-one cells $ \{R_1(ι_1,τ_1),\dots, R_1(ι_n,τ_n)\} $ such that $R_1(ι,τ)=R_1(ι_1,τ_1)$, for all $ j\in \{1,\dots,n-1\} $,
$d( R_1(ι_{j+1},τ_{j+1}),L_0)< d(R_1(ι_j,τ_j),L_0)$ and $ R_1(ι_n,τ_n) $ is either equal or adjacent to $R_1(ι',τ'). $
In addition, we define here two cells to be \emph{diagonally linked} if the first case occurs, i.e.\ if $ R_1(ι_n,τ_n) = R_1(ι',τ'). $

We say that two scale-one cells $ R_1(ι,τ) $ and $R_1 (ι',τ') $ are \emph{single diagonally connected} if $ R_1(ι,τ) $ is diagonally connected to $ R_1(ι',τ') $ or if $ R_1(ι',τ') $ is diagonally connected to $R_1 (ι,τ) $. We say that two scale-one cells $ R_1(ι,τ) $ and $ R_1(ι',τ') $ are double diagonally connected if there exists $ R_1(\tilde{ι},\tilde{τ}) $ such that $ R_1(ι,τ) $ is diagonally connected to $ R_1(\tilde{ι},\tilde{τ}) $,
$ R_1(ι',τ') $ is diagonally connected to $ R_1(\tilde{ι},\tilde{τ}) $,
and either $ R_1(ι,τ) $ or $ R_1(ι',τ') $ is diagonally linked to $ R_1(\tilde{ι},\tilde{τ}) $.
Note that being single diagonally connected or double diagonally connected is a symmetric relationship.

As done in Subsection \ref{SUBSEC-Paths}, we extend these new definitions to cells of arbitrary scale $ R_{k_1}(ι_1,τ_1) $ and $ R_{k_2}(ι_2,τ_2) $ by requiring that they respectively contain two scale-one cells which satisfy the corresponding definition of the connectedness above. In analogy to Definition \ref{def-Dpath} we introduce a new type of paths.

\begin{deff}\label{def-DD}
	We define a \DD \ as a sequence $ \{R_{k_j} (ι_j,τ_j)\}_{j=1}^{n} $ of cells where for each $ j \in \{2,\dots,n \}$, the cells $ R_{k_{j-1}}(ι_{j-1},τ_{j-1}) $ and $ R_{k_{j}}(ι_j,τ_j) $ are either adjacent, single diagonally connected or double-diagonally connected.
\end{deff}

Similarly to \eqref{def-Ωt}, for some $ t>0$ and $v\in L_1$, we define
\begin{equation}\label{def-ΩDt}
	\OmegaD_{1}(v\rightarrow t)
\end{equation}
to be the set of all \DD s of cells of scale one for which the first cell of the path is $ v $ or $ v $ is single diagonally connected to the first cell, and the last cell is the only cell not contained in the space-time ball $B_{t}(S_1(ι_v))\times [-t+τ_v,τ_v+t]$. 
We stress that, contrary to $ Ω_1(v\rightarrow t) $, $ v $ must not necessarily be part of the \DD{}; it can be that $v$ is only single diagonally connected to the path and not an actual cell of the \DD{}.

We define now Sc\DD{s}, the support connected version of \DD{s}. Recall the definition of well-separated cells and support adjacent cells from Definition \ref{def-ScDpath}. We say that two cells $R_{k_1}( ι_1,τ_1 )$ and $ R_{k_2}(ι_2,τ_2) $ are \emph{support connected with single diagonal} if there exist two scale-one cells respectively contained in the extended supports of $R_{k_1}(ι_1,τ_1 )$ and $ R_{k_2}(ι_2,τ_2) $ which are single diagonally connected. Similarly, we say that two cells $ R_{ k_1}(ι_1,τ_1 )$ and $ R_{k_2}(ι_2,τ_2) $ are \emph{support connected with double diagonal} if there exist two scale-one cells respectively contained in the extended supports of $R_{k_1}(ι_1,τ_1 )$ and $ R_{k_2}(ι_2,τ_2) $ which are double diagonally connected.

\begin{deff}\label{def-ScDD}
	We define as \SDD{} (support connected \DD) a sequence  of well-separated cells $ \{R_{k_j}(ι_j,τ_j)\}_{j=1}^{z} $ for some $z\in \N$ where for all $ j=2,\dots,z $ the cells $ R_{k_{j-1}}(ι_{j-1},τ_{j-1}) $ and $ R_{k_{j}} (ι_j,τ_j) $ are either support adjacent, support connected with single diagonal or support connected with double-diagonals.
\end{deff}
For  $ t>0$ and $v\in L_1$, we define
\begin{equation*}\label{def-ΩDtsup}
	\OmegaD\Sup_{κ}(v\rightarrow t)
\end{equation*}
the set of all Sc\DD s of cells of scale at most $ κ $ so that the extended support of the first cell of the path contains $ v $ or $ v $ is single diagonally connected to a scale-one cell that is contained in the extended support of the first cell, and the last cell is the only cell whose extended support is not contained in the space-time ball $B_{t}(S_1(ι_v))\times [-t+τ_v,τ_v+t]$.
Again, we highlight the difference with $ Ω\Sup_κ (v\rightarrow t)$, where instead $ v $ must be contained in the extended support, whereas here it can be only single diagonally connected to it.

Finally we define the analogue of the bad cluster $ K_v $ from \eqref{def-K}:
\begin{equation}\label{def-K*}
	K_{v}^*:=\{R_1(ι',τ')\colon \text{there exists a \DD\ of bad cells from $ v $ to $ R_1(ι',τ') $}\}
\end{equation}

Repeating the arguments of Lemma \ref{lemma-Ω1Ωsup}, we can easily obtain its analogue for \DD s.
\begin{lemma}\label{lemma-Ω1ΩsupDD}
	It holds that
	\begin{align*}
		 & \P \big(\exists P \in \OmegaD_1(v\rightarrow t) \text{ of cells with bad ancestry}\big) \\
		 & \le
		\P \big(\exists P \in \OmegaD\Sup_κ(v\rightarrow t) \text{ of multi-scale bad cells}\big).
	\end{align*}
\end{lemma}

\subsection{Multi-scale analysis of \texorpdfstring{\DD  s}{DD}}
We want to show that the \Lip{} intersects the \emph{base} $ L_0 $ within distance $ r $ from the origin with high probability. If the opposite was true, then we would be able to find a nearest-neighbour path in $ L_1 \setminus F $ which leaves a ball of radius $ r.  $ We will show that this implies the existence of a \DD\ from the origin that exits such a ball and we will use similar arguments to before to prove such \DD s are improbable.

We follow the structure of Section \ref{SEC-MultiScaleAnalysis} and write in detail only the parts where the proofs for \DD s differ from the ones for D-paths. Lemma \ref{lemma-Mixing} and Corollary \ref{corol-MixingPaths} still hold and and can be applied unchanged. We need to show the analogue of Lemma \ref{lemma-Combinatorial}.

\begin{lemma}\label{lemma-CombinatorialDD}
	For a fixed length $ z \in\N$, fixed scales $ k_1,\dots, k_z $ and $ v\in L_1 $, the number of \SDD s of cells of scale $k_1,\dots,k_z$ where the first cells either contains $v $ or is $v$ is single diagonally connected to a scale-one cell contained in the extended support of the first cell, is at most
	\begin{equation*}
		\exp \Big\{{\dfrac{C_ψ}{2} \sum_{j=1}^{z} ψ_{k_j} 	}	\Big\},
	\end{equation*}
	where $ C_ψ $ is the same constant as in Lemma \ref{lemma-Mixing}.
\end{lemma}
\begin{proof}
	We follow the proof of Lemma \ref{lemma-Combinatorial}.  For \textit{Step 1}, we need to make a small change. Compare the definitions of $ Ω_κ\Sup(v\rightarrow t) $ and $ \OmegaD_κ\Sup(v\rightarrow t) $: in the latter we also allow $ v $ to be single diagonally connected to a scale-one cell contained in the extended support of the first cell in the \DD. To account for this, note that we can fix the relative position of $ v$ and the scale-one cell in the extended support of the first cell in the \DD, and we are only left to control the number of the possible relative position which is done in \textit{Step 3.}

	\textit{Step 2} remains unchanged, and we can turn to \textit{Step 3.}

	Consider two consecutive cells in the \DD\ which are single diagonally connected. We can define similarly to before
	\begin{equation*}
		A(h):=\max_{(ι_1,τ_1)}\bigg| \Big\{  R_1(ι_2, τ_2)							\colon
		\begin{array}{l}
			R_1(ι_1,τ_1)	\text{ is single diagonally connected to }R_1(ι_2,τ_2) \\
			\text{ with } \big|d(L_0, R_1(ι_1,τ_1))  -  d(L_0,R_1(ι_2,τ_2))\big|= h
		\end{array}			\Big\}\bigg|  .
	\end{equation*}
	For two cells $ R_1(ι_1,τ_1) $ and $ R_1(ι_2,τ_2) $ in the \DD\ which are double diagonally connected, let $R_1 (\tilde{ι}, \tilde{τ}) $ be the cell of the double diagonal that $R_1 (ι_1,τ_1) $ or $ R_1(ι_2,τ_2) $ is diagonally linked to. Letting $ h_1 $ be the height difference between $ R_1(ι_1,τ_1) $ and $ R_1(\tilde{ι}, \tilde{τ}) $ and $ h_2 $ the height difference between $R_1 (ι_2,τ_2) $ and $R_1(\tilde{ι}, \tilde{τ}) $, we can upper bound the number of different relative positions between $ R_1(ι_1,τ_1) $ and $R_1 (ι_2,τ_2) $ for which the respective height differences to $ R_1(\tilde{ι},\tilde{τ}) $ are $ h_1 $ and $ h_2 $ by $ A(h_1+1)A(h_2+1). $

	Let $ H_k $ be as in \eqref{def-Hk}; similarly to what was done for D-paths, we can bound the total number of diagonal steps in a \DD{} with the maximal attainable distance from $ L_0, $ within the path, i.e.\ by
	\begin{equation*}
		H=2\sum_{i=1}^{z} H_{k_i},
	\end{equation*}
	where we added the factor 2 to account for the diagonal step to the \emph{previous} and the \emph{following} cell. For simplicity, when two cells are double diagonally connected we consider also the cell $ R_1(\tilde{ι},\tilde{τ}) $, to which both cells are diagonally connected as part of the path. So, letting $ h_i $, $ i=1,\dots,2z-1 $ be the height difference between two diagonally connected cells, the number of diagonal steps is at most
	\begin{equation*}
		\sum_{l=0}^{H}\sum_{\substack{h_1,\dots h_{2z-1}\\ h_1+\dots+h_{2z-1}=l}} A(h_1+1)A(h_2+1)\dots A(h_{2z-1}+1).
	\end{equation*}
	We can then repeat the remaining calculations as in Lemma \ref{lemma-Combinatorial}, substituting $ z $ with $ 2z $ and obtain the same result.
\end{proof}

We also have the analogue of Lemma \ref{lemma-φandt}:
\begin{lemma} \label{lemma-φandtDD}
	Suppose that the largest scale $κ$ satisfies $ κ=\mathcal{O}(\sqrt{\log(t)}).$ Then if $ψ_1$ is large enough, there exist $t_0 $ and $\Cr{t}>0$
	such that for any $ t> t_0 $ and any $ v\in L_1 $ and any \SDD\ $ \{R_{k_j}(ι_j,τ_j)\}_{j=1}^{z} \in \OmegaD_{k}\Sup(v\rightarrow t)$
	\begin{equation*}
		\sum_{j=1}^{z} ψ_{k_j} \ge			 \Cr{t}t^{c_s} .
	\end{equation*}
\end{lemma}
\begin{proof}The proof is unchanged from the one of Lemma \ref{lemma-φandt} except that in \eqref{eq-diam>t} we have to substitute $ t/2 $ with $ t/3 $ since we now consider 2 diagonals for each cell instead of only one. The rest remains identical.
\end{proof}

Recall now the definition of $ K_{v}^* $ in \eqref{def-K*}. The analogue of Proposition \ref{prop-KnotinR} is then argued in the same way.
\begin{prop}\label{prop-K*notinR}
	Let $ ζ $ as in \eqref{def-ζ}, $ α_0 $ as in Lemma \ref{lemma-Mixing} and $ t_0 $ as in Lemma \ref{lemma-φandt}. Then there exists a constant $ \Cr{K} $ independent of $ t $ such that for any $v\in L_1$
	\begin{equation*}
		\P(K_{v}^* \nsubseteq \mathbf{R}_1^t(v)) \le e^{-\Cr{K} t^{c_s} },
	\end{equation*}
	for all $ t>t_0, $ with $c_s$ as in Theorem \ref{thm-main2}.
\end{prop}

Recall the concept of hills from Definition \ref{def-HillMountain}. In the following, we will say that two hills $ \Hill_{v_1} $ and $ \Hill_{v_2} $ are adjacent if there exist $ v'_j \in \Hill_{v_j} $, $ j=1,2 $ that are adjacent, and call them \emph{intersecting} if there exists $ \tilde{v} \in \Hill_{v_1}\cap \Hill_{v_2}.$

\begin{lemma}\label{lemma-HillsPath}
	Let $F$ be the \Lip{} from Theorem \ref{thm-main}. Let $ π =\{u_j\}_{j=0}^{n}$ with $ u_j \in L_1\setminus F $ be a sequence of sequentially pairwise adjacent cells.

	Then there exists a sequence of hills $\mathcal{H}:= \{\Hill_{v_j}\}_{j=0}^{k} $, $ k\le n $, such that every $ u_j $ is contained in some hill $ \Hill_{j'} $ and two consecutive hills of the sequence are either adjacent or intersecting.

	Furthermore there exists a \DD\ which starts in $ u_0 $ and ends in $ u_n. $
\end{lemma}
\begin{proof}
	We start with the first claim. For each $ u_j \in π$, we have by assumption that $ u_j \notin F $, so there exists a hill $ \Hill_{v_j}\ni u_j. $ Furthermore, for all $ j=1,\dots, n $, $ u_{j-1} $ and $ u_j $ are adjacent and so the respective hills $ \Hill_{v_{j-1}} $ and $ \Hill_{v_{j}} $ are either adjacent or they intersect.
	The sequence of hills $ \{\Hill_{v_j}\}_{j=0}^{n} $ may contain repetitions of the same hills, so by removing all but the first appearance of those which appears multiple times, we end up with a sequence of $ k\le n $ different elements.

	We prove now the existence of the \DD. Consider the sequence of hills $ \{\Hill_{v_j}\}_{j=0}^{k} $ from the previous step, and denote with $ \overleftarrow{v}_{j} \in \Hill_{v_j} $ for $ j=1,\dots, n $ the cell (chosen in some arbitrary manner, for example lexicographically) that is either contained in or adjacent to a cell contained in $ \Hill_{v_{j-1}}. $ By definition of a hill, there exist a d-path from $ v_0$ to $ u_0 $ and a d-path from $ v_0$ to either $ \overleftarrow{v}_{1} $ or to a cell adjacent to it. Similarly, there exist a d-path from $ v_j $ to $ \overleftarrow{v}_j $ and a d-path from $ v_j $ to $ \overleftarrow{v}_{j+1} $ (or a cell adjacent to it). Repeating this, we obtain a sequence of cells
	\begin{equation*}
		u_0 ,	v_0,		\overleftarrow{v}_1,
		v_1,		\overleftarrow{v}_2,
		\dots,
		v_k,   u_n
	\end{equation*}
	where for each pair of consecutive cells there exists a d-path from the first to the second or from the second to the first.

	Note that, just like D-paths, d-paths are also \DD{s}. Secondly, if a certain sequence is a \DD, then the reverse sequence is also a \DD{}, as a simple consequence of the fact that being adjacent, single diagonally connected or double diagonally connected is a symmetric relation. Thirdly, if there exist a \DD{} from a cell $ u_1 $ to $ u_2 $ and one from $ u_2 $ to $ u_3 $ we can concatenate them and obtain a \DD{} from $ u_1 $ to $ u_3. $

	We can thus construct a \DD\ for the sequence $ u_0 ,	v_0,		\overleftarrow{v}_1,v_1,\overleftarrow{v}_2,\dots,v_k,   u_n	 $, concluding the lemma.
\end{proof}

We can now prove Theorem \ref{thm-main2}.

\begin{proof}[Proof of Theorem \ref{thm-main2}]
	By Theorem \ref{thm-main}, a \Lip{} $ F $ exists \as{}, so we need to show that it surrounds the origin at some distance $ r. $ Suppose the converse.

	This means that there exists a sequence of cells $ \{u_j\}_{j=0}^{n} $ with $ u_j:=R_1(ι_j,τ_j)\in L_1\setminus F $ and such that $ u_0=R_1(0,0) $ and $ d(u_n, u_0) > r$. Applying Lemma \ref{lemma-HillsPath} we obtain the existence of a \DD\ from $R_1 (0,0) $ to $ u_n. $

	By Proposition \ref{prop-K*notinR}, for $ t>t_0 $, the probability that such a path exists is smaller than
	\begin{equation*}
		\P(K_{(0,0)}^* \nsubseteq \mathbf{R}_1^t(0,0)) \le e^{-\Cr{K} t^{c_s} }.
	\end{equation*}
	Setting again $ t= (	\frac{1}{2^\ell} +\frac{1}{\Cr{T} 2^{\dw \ell}}	) ^{-1}  r $ as in the proof in Subsection \ref{SUBSEC-ProofThm1} concludes the proof for
	$r_0 := (	\frac{1}{2^\ell} +\frac{1}{\Cr{T}2^{\dw \ell}}	) t_0.$
\end{proof}

\section{Generalised \Sier{} carpets}\label{SEC-Carpets}

In this section we show how to adapt the previous arguments for the \Sier{} gasket to a further class of fractal graphs, the \Sier{} carpets. We start by introducing the graph and then stating the results. As we will see, other than changes to constants and parameters, the work done for the gasket can be applied mostly without further changes necessary, so we will only highlight selected statements to show how they work in the carpet case.

\begin{figure}[!ht]
	\centering
	\begin{subfigure}[t]{0.48\linewidth}
		\includegraphics[height=\linewidth]{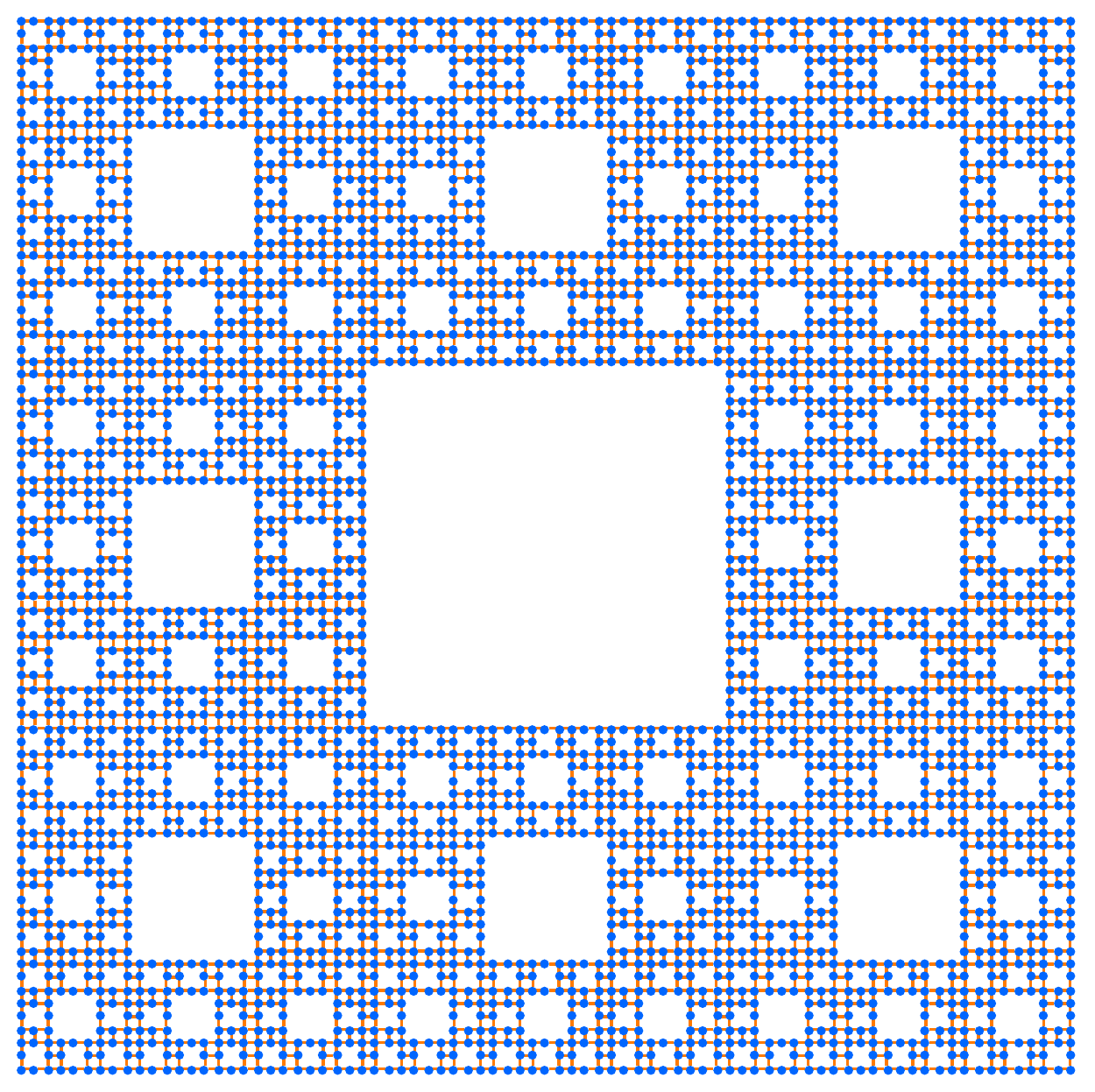}
		\subcaption{$d=2$}
	\end{subfigure}
	\hfill
	\begin{subfigure}[t]{0.48\linewidth}
		\includegraphics[height=\linewidth]{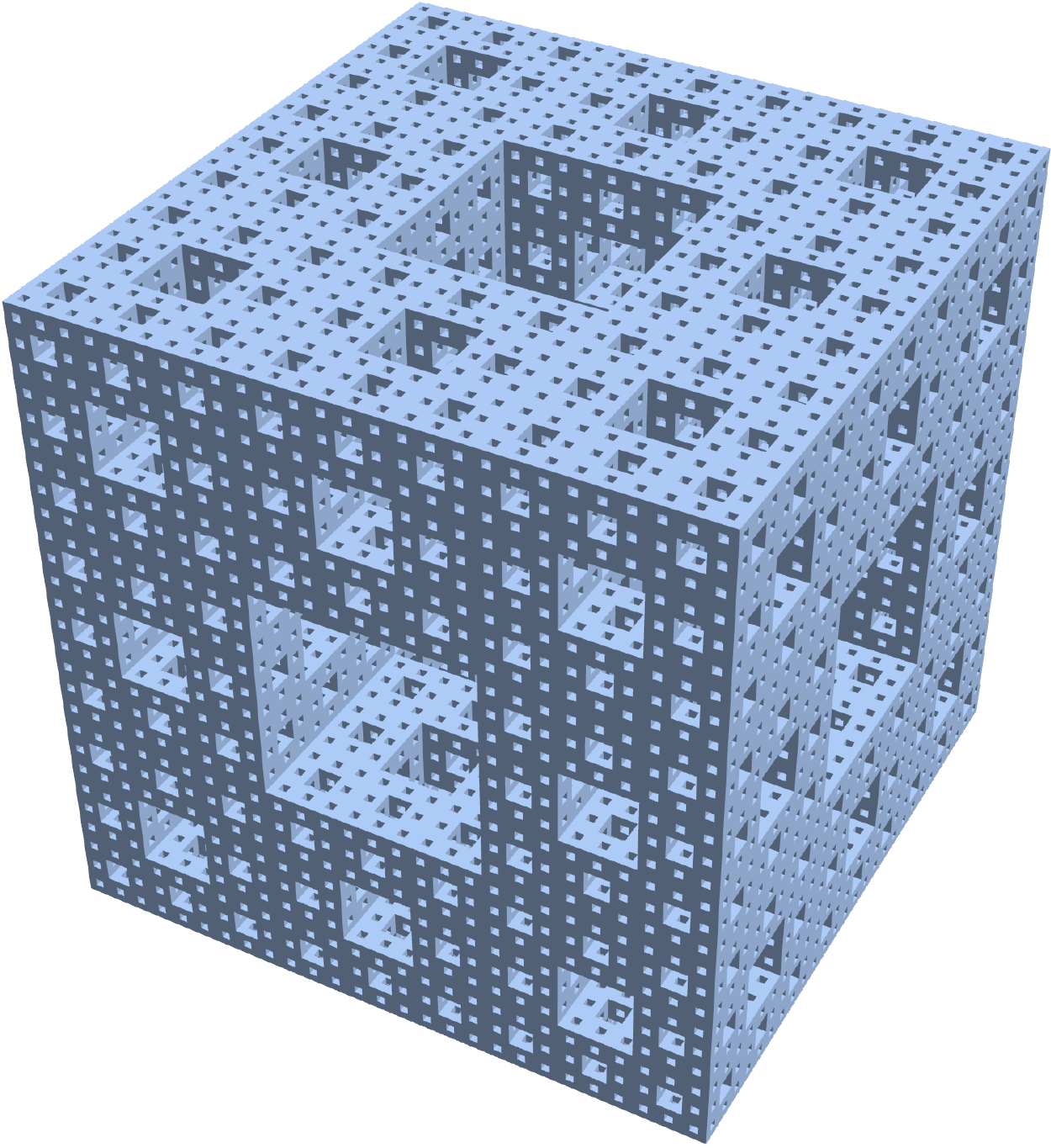}
		\subcaption{$d=3$}
	\end{subfigure}
	\caption{Examples of generalised \Sier{} Carpet. }\label{fig-SierpCarpet}
\end{figure}

\subsection{Setup and statement}

We consider the class of fractal graphs of \cite{BarlowBassCarpets}. We state the definition for completeness and refer to \cite{BarlowBassCarpets} for more details.

Let $ d\ge 2 $, $ \unit\ge 3,  $ and $ 1\le \remain \le (\unit)^d. $
Next, let $ F_0 :=[0,1]^d $ and for $n\in\Z$ $ \mathcal{S}_n $ be the collection of closed cubes of side $ (\unit)^n $ and corner vertices in the lattice $ (\unit )^n \Z^d$. For $ A\subseteq\R^d $ let $ \mathcal{S}_n(A):= \{S\in \mathcal{S}_n\colon S \subseteq A\} $. For $ S\in\mathcal{S}_n $, let $ Ψ_S $ be the orientation preserving affine map which maps $ F_0 $ onto $ S. $

Let $ F_1 $ be the union of $ \remain $ distinct cubes of $ \mathcal{S}_{-1} (F_0) $ satisfying the following conditions:

\begin{enumerate}[(H1)]
	\item \textit{Symmetry}: $ F_1 $ is preserved by all the isometries of $ F_0. $
	\item \textit{Connectedness}: the interior $ \text{Int}(F_1) $ is connected, and contains a path connecting the hyperplanes $ \{x_1=0\} $ and $ \{x_1=1\} $.
	\item \textit{Non-diagonality}: For any cube $ B $ in $ F_0 $ which is the union of $ 2^d $ distinct elements of $ \mathcal{S}_{-1} $, if $ \text{Int}(F_1 \cap B) $ is non-empty, it is connected.
	\item \textit{Borders included:} $ F_1 $ contains the segment $ \{x\colon 0\le x_1\le 1, x_2=\dots =x_d=0\} $.
\end{enumerate}

Given $ F_n,$ $ F_{n+1} $ is obtained by removing the same pattern from each of the squares in $ \mathcal{S}_{-n}(F_n) $, so that $ F_{n+1} $ is the union of $ (\remain) ^n$ squares in $ \mathcal{S}_{-n}(F_0) $; formally
\begin{equation*}
	F_{n+1} := \bigcup_{S\in \mathcal{S}_{-n}(F_n)}  Ψ_S (F_1)
\end{equation*}
and $ F := \bigcap_{n=0}^{∞} F_n $ is called a \emph{generalised \Sier{} carpet.} The Hausdorff dimension of $ F $ is $ \dv:= \frac{\log (\remain)}{\log(\unit)} $ (see \cite{BarlowBassCarpets} and references therein). We now define the pre-fractal graph.

For any cube $ S_{-n} $, call the \emph{lower-left corner} the vertex $ x $ with $ x_i \le y_i $ for each $ i=1,\dots, d $ and $ y\in S_{-n} $. Let $ \Box_n $ be the collection of lower-left corners of the cubes in $ (\unit)^n  F_n $, and
\begin{equation*}
	V:= \bigcup_{n=0}^{∞} \Box_n,
\end{equation*}
see Figure \ref{fig-CarpetBoxes}.

We define the \emph{generalised \Sier{} carpet graph} $ \mathbb{SC}^d:=\mathbb{SC}^d(\unit,\remain) $ as the graph with vertex set $ V $ and edges $ E:=\{\{x,y\}	\in V\times V\colon \|x-y\|_1=1\}. $

\begin{figure}[ht]
	\centering
	\includegraphics[width=0.3\linewidth]{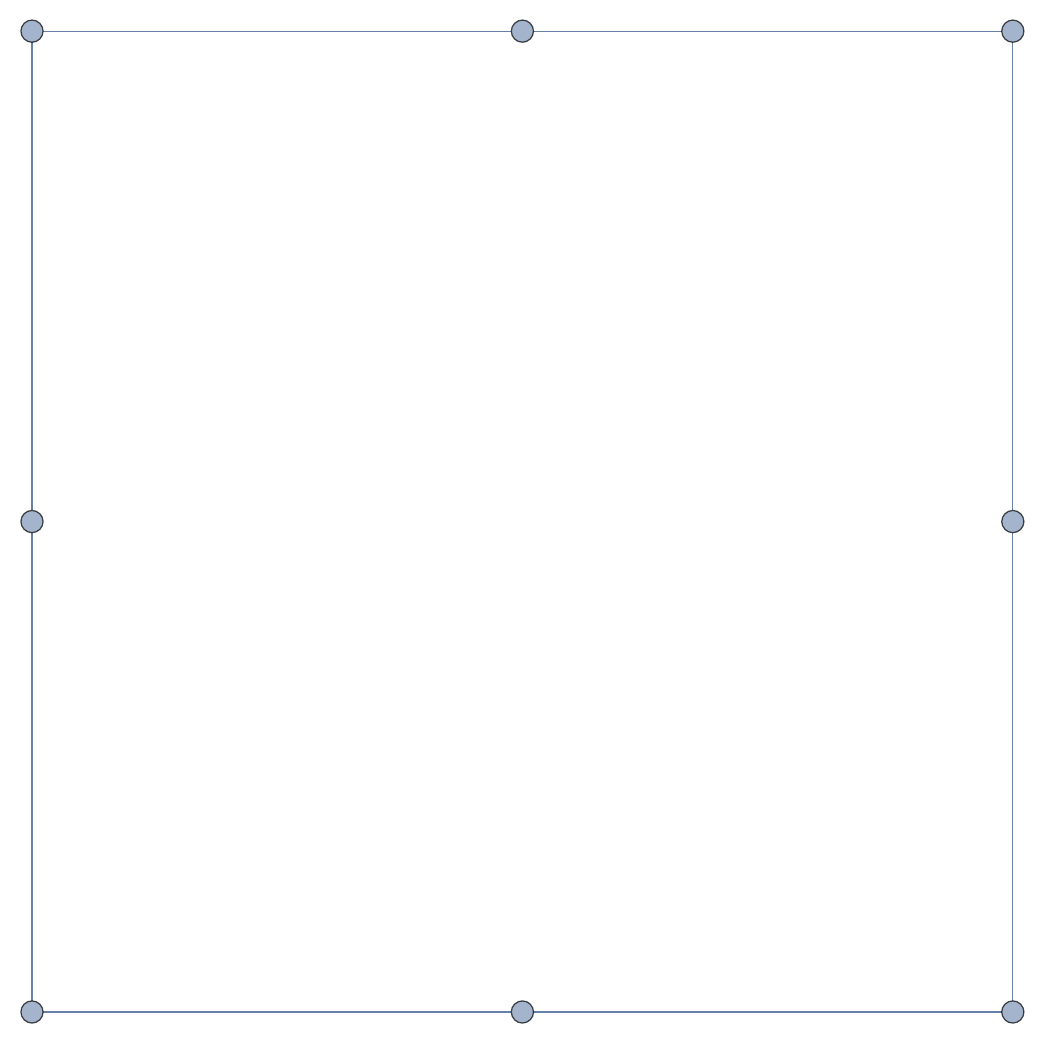} \hfill
	\includegraphics[width=0.3\linewidth]{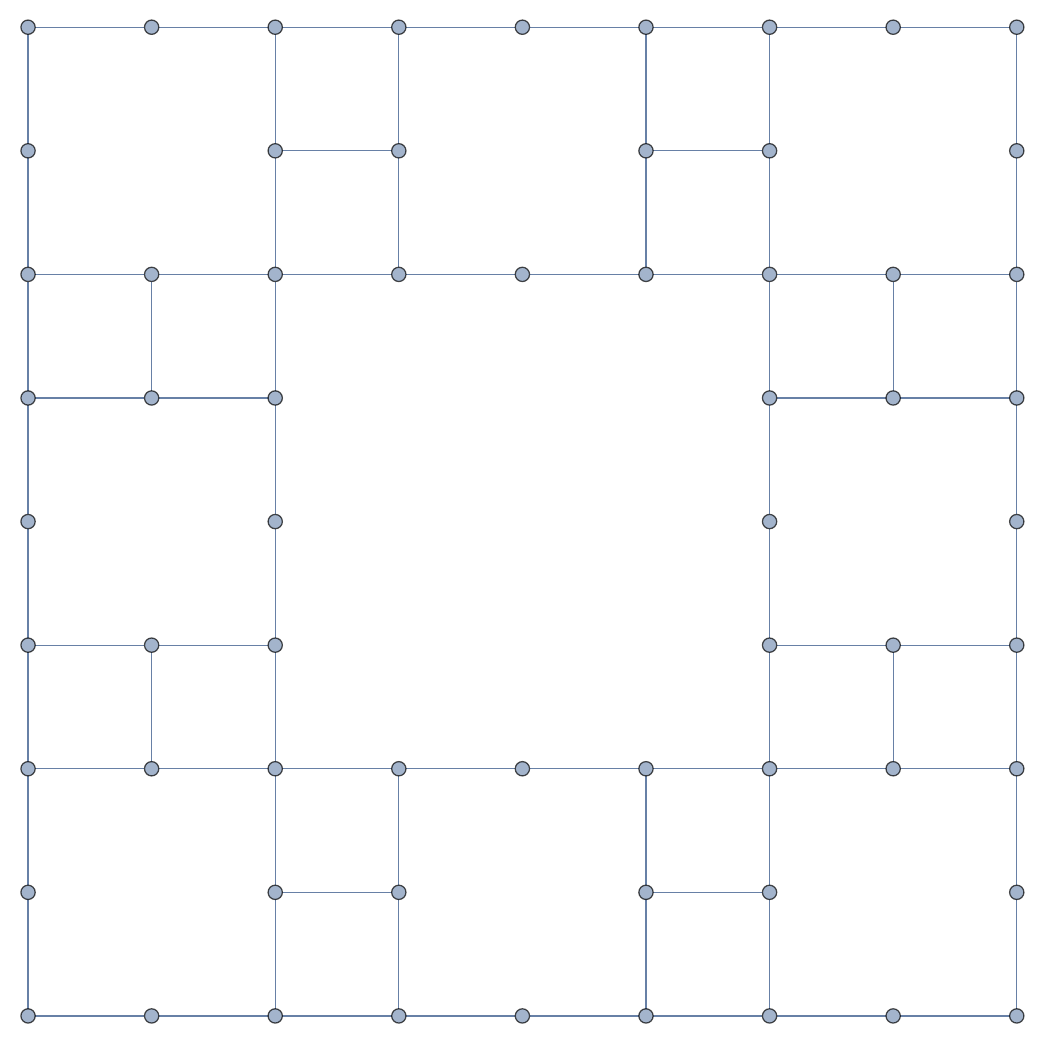} \hfill
	\includegraphics[width=0.3\linewidth]{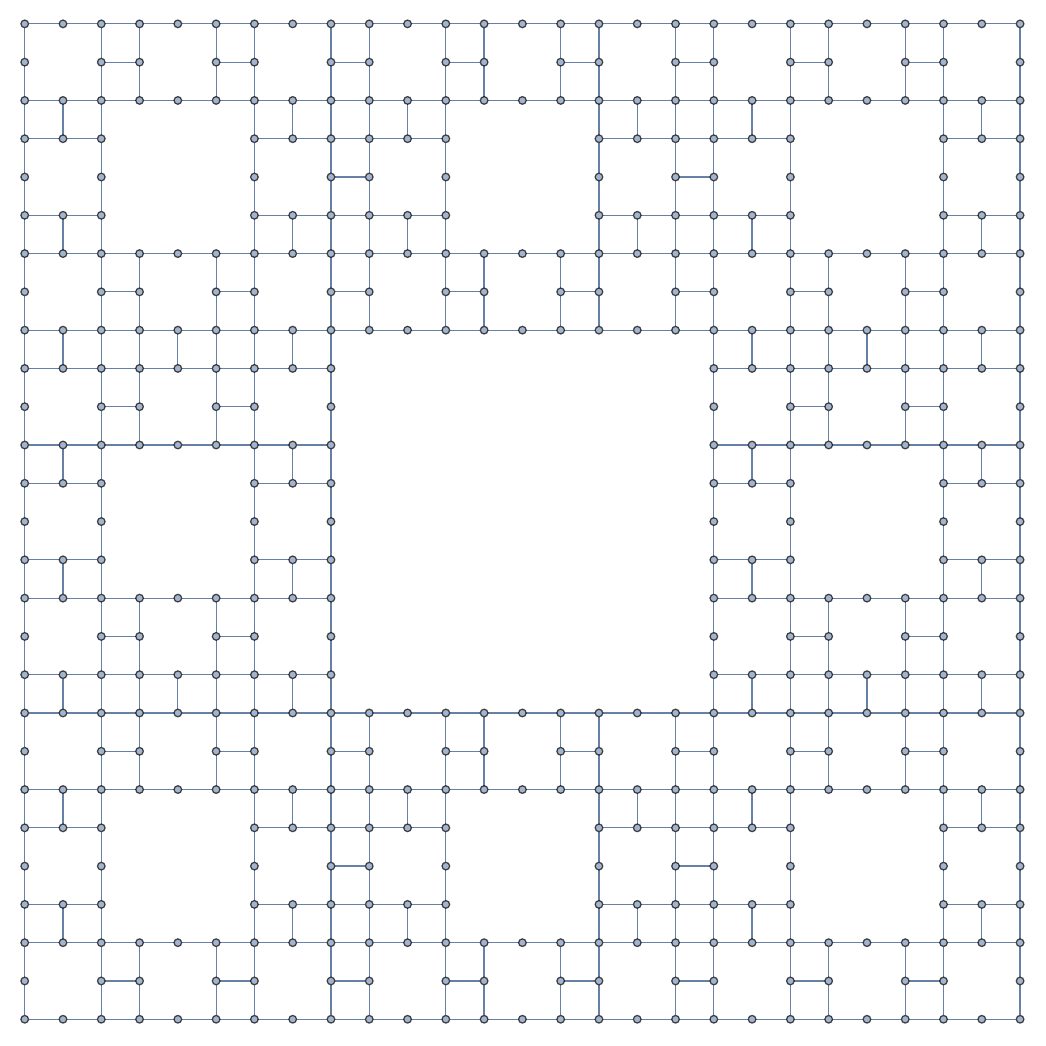}
	\caption{$\Box_1$, $\Box_2$ and  $\Box_3$ with $d=2, \unit=3$, $\remain= 8$ and corresponding edges. Note that the first 2 pictures are scaled up by a factor of $3^2$ and $3^1$ respectively.}\label{fig-CarpetBoxes}
\end{figure}

Similarly to \Sier{} gaskets, one can easily prove the volume estimate
\begin{equation}\label{def-dvSC}
	\cvol \, r^{\dv} \le \vol_r(x)\le \CVol\,r^{\dv},
\end{equation}
with $ \dv:= \frac{\log(\remain)}{\log(\unit)}$. Theorem 1.5 in \cite{BarlowBassCarpets} shows that upper and lower bounds for the heat kernels \eqref{eq-HKBound} hold for some value $ \dw$ (which to the best of our knowledge is not known explicitly). Similarly to gaskets, applying \cite[Theorem 3.1]{Grigor02} gives that the mean exit time satisfies
\begin{equation}\label{def-dwSC}
	E_x[H_{B_r(x)^\mathsf{c}}] \asymp r^{\dw},
\end{equation}
and that the parabolic Harnack inequality \eqref{eq-PHI} with parameter $ \dw $ holds. Furthermore Lemma \ref{lemma-confining} also holds due to the above.

Now that the graph has been defined, we can define the tessellation of the carpets, in order to formulate the analogues of Theorems \ref{thm-main} and \ref{thm-main2}. 
We define the tiles $ S_k(ι) $ as
\begin{equation*}
	S_k(ι) := ι(\unit)^{\ell_k} + \Box_{\ell_k},
\end{equation*}
$ι\in \mathbb{SC}^d$ which is the union of $ \unit^{\dv (\ell_{k}-\ell_{k-1})} $-many $ (k-1) $ tiles.

Define just like before $ β_k $ to be
\begin{equation*}
	β_k:=\Cmix ( \tfrac{k^2}{ε})^{\frac{4}{Θ}} 	\big( \unit^{\ell_{k-1}} \big)^{\dw}
\end{equation*}
with the walk dimension $ \dw $ from \eqref{def-dwSC} and we define the time interval $T_k(τ)$, $τ\in\Z$, as before. Similarly, we define space-time cells as the cross product of spatial tiles with the time intervals.

Like in the gasket case, we define $L_0$ and $L_1$ as in \eqref{def-L0} and \eqref{def-L1} to be the ``hyperplane'' subgraph and its corresponding collection of cells.
Note that in order to define $L_0$, one needs to consider a subgraph $\mathbb{SC}^{d-1}(\unit,\remain)$ with the same $\unit$ but an appropriately changed $\remain$. As an example, in the case of the 3 dimensional \Sier{} carpet from Figure \ref{fig-SierpCarpet}, $\remain$ must be changed from $20$ in $d=3$ to $\remain=8$ in $d=2$.

We define two scale-one cells $ R_1(ι_1,τ_1)$ and $ R_1(ι_2,τ_2) $ to be \emph{adjacent} if $ d(ι_1,ι_2)  + |τ_1-τ_2|= 1.$ With this adaptation, we can define the \Lip{} $ F $ as in Definition \ref{def-LipschitzSurf}, and state the main theorem.

\begin{remark}\label{rmk-AdjacentCarpet}
	The change in how adjacency is defined is due to the ``disjoint'' nature of how the pre-fractal is constructed (recall that with the gasket, the corners of the triangles were shared). With this new definition of adjacency, we recover the same behaviour in the sense that two cells are adjacent if either they are spatially the same and only one time interval away from each other, or if they share the time interval and are spatially nearest neighbours, i.e.\ have $L^1$ distance equal to 1.
\end{remark}

\begin{thm}\label{thm-main3}
	Let $d\ge 2$, $\unit\ge3,$ $1\le \remain \le (\unit)^d$ and $ \mathbb{SC}^d (\unit, \remain)$ be a $ d- $dimensional generalised \Sier{} carpet. Let $\ell\in \N$ and let $β\in\N$ be large enough. Furthermore, let $η\in\N$, $ ε\in(0,1)$ and $ζ \in (0,\infty)$ such that 
	\begin{equation*}
		ζ \ge	\frac{1}{\ell}		\Big(	\big(   \tfrac{1}{\Cr{conf1}^{-1}}\log \big( \tfrac{8\Cr{conf1}}{3ε} \big) \big)^{\dw-1}	ηβ	 \Big)^{\frac{1}{\dw}}		.
	\end{equation*}
	Tessellate $ \graph $ into space-time cells as described above, and let $ E(ι,τ) $ be an increasing event restricted to the super cell $R_1^η(ι,τ)$ whose associated probability $ν_E\big((1-ε)μ, S_1^η(ι,τ), \displ_{ζ\ell}, ηβ\big)$ has a uniform lower bound across all $(ι,τ)\in \mathbb{SC}^d\times \Z$ denoted with
	\begin{equation*}
		ν_E\big((1-ε)μ, S_1^η, \displ_{ζ\ell}, ηβ\big).
	\end{equation*}

	Then there exists $ α_0 >0$ such that if
	\begin{equation*}
		ψ_1 (ε,μ_0,\ell) :=	\min \Big\{	\frac{ε^2μ_0 2^{\dv \ell}	} {C_λ}				,		-\log \Big(	1- ν_E\big( (1-ε)λ, S_1^η, \displ_{ζ\ell}, ηβ	 \big)
		\Big)	\Big\}  \ge α_0
	\end{equation*}
	there exists almost surely a \Lip{} $ F $ where the event $ E(ι,τ) $ holds for all $ (ι,τ)\in F $.

	Furthermore there exists $ \Cr{thm2a}>0 $ such that for $ r_0 $ large enough
	\begin{equation*}
		\P \big( S(F,r_0)^{\mathsf{c}} \big)		\le \sum_{r\ge r_0}  r^{\dv+1} e^{		-\Cr{thm2a}  r^{c_s} 		},
	\end{equation*}
	for $c_s\in \big(0, \frac{\dv}{\dv+1}-\frac{1}{2}\big)$ and $S(F,r_0)$ was defined above Theorem \ref{thm-main2}.
\end{thm}

The application of Theorem \ref{thm-main3} is similar to how Theorem \ref{thm-main} is applied, including the order in which the various quantities are fixed. As such, we remind the reader of Remark \ref{rem-order} where details about this can be found.

\subsection{Proof of Theorem \ref{thm-main3}}

To adapt the proof, only a single notable change beyond the changes in the preceding definitions is necessary. Similar to those, this change is essentially substituting the base 2 that appeared in the gasket case with $\unit$, as we have seen in the definitions of $S_k(ι)$ and $β_k$. From here onward we will repeatedly:
\begin{equation}\label{substitution}\tag{Subst}
	\text{Substitute every base 2 exponential with a base $\unit$ exponential. }
\end{equation}

Recall the definition of adjacent scale-one cells above Remark \ref{rmk-AdjacentCarpet}. We generalise this to cells of arbitrary scale: two cells $ R_k(ι_1,τ_1)$ and $ R_k(ι_2,τ_2) $ of the same scale are called \emph{adjacent} if $ d(ι_1,ι_2)+ |τ_1-τ_2|\le 1$, where $d(\cdot,\cdot)$ is as before the graph distance. Seeing $ \mathbb{SC}^d\times \Z $ as a subgraph of $ \Z^{d+1} $, we define two cells $ R_{k}(ι_1,τ_1) $ and $ R_k(ι_2,τ_2) $ to be \emph{$ \ast $-neighbours} if $ \|  (ι_1,τ_1) - (ι_2,τ_2) \|_∞ \le 1.$
We next define d-paths for carpets.

\begin{deff}[d-path]\label{def-dpathCarpet}
	A \emph{d-path} in $ \graph$ is a  sequence $\{u_k\}_{k=0}^{n}$ of $ \ast-$neighbouring cells in $\mathbb{SC}^d\times \R $ from a bad cell $u_0\in L_1$ such that for each $u_k$ and $u_{k+1}$ one of the following holds:
	\begin{itemize}
		\item \emph{increasing move}: $u_{k+1}$ is bad and $d(L_0,u_{k+1})\ge d(L_0,u_{k})$
		\item \emph{diagonal move}:\ \ \ $d(L_0,u_{k+1})< d(L_0,u_{k})$
	\end{itemize}
\end{deff}

We say for two scale-one cells $R_1(ι,τ) $ and $ R_1(ι',τ') $ that $R_1(ι,τ) $ is \emph{diagonally connected} to $ R_1(ι',τ') $ if there exists a sequence of $ \ast $-neighbour scale-one cells $ \{R_{1}(ι_1,τ_1),\dots, R_1(ι_n,τ_n)\} $ such that $R_1(ι,τ)=R_1(ι_1,τ_1)$, for all $ j\in \{1,\dots,n-1\} $,
$d(R_1(ι_{j+1},τ_{j+1}),L_0)< d(R_1(ι_j,τ_j),L_0)$ and $ R_1(ι_n,τ_n) $ is either equal or adjacent to $R_1 (ι',τ'). $

Lemma \ref{lemma-Ω1Ωsup} then still applies using the change \eqref{substitution}. Similarly, the definition of $ ψ_k $ is subject to \eqref{substitution}. In this way, Lemma \ref{lemma-Mixing} can be proven in the same way by again applying the decoupling Theorem \ref{thm-mixing} with the choices
\begin{align*}
	K          & :=\text{side length of }S_k\base(ι)  = 2b(k) \unit^{\ell_k}    		+\unit^{\ell_k},      \\
	K'         & \text{ such that }K-K'=b(k)\unit^{\ell_k},                                             \\
	l          & := \unit^{\ell_k},                                                                     \\
	δ          & := (1-\mathfrak{d}_{k+1}) μ_0,                                                         \\
	\mathbf{Δ} & := \text{length}([γ_{k}^{(1)}(τ)β_{k+1},τβ_k])=τβ_k-γ_{k}^{(1)}(τ)β_{k+1},	\text{ and} \\
	\bar{ε}    & :=\frac{ε}{8k^2}.
\end{align*}

Lemmas \ref{lemma-Combinatorial} and \ref{lemma-φandt}, Proposition \ref{prop-KnotinR}, and the proof of Theorem \ref{thm-main} go through by applying \eqref{substitution}, and therefore the first half of Theorem \ref{thm-main3} is shown.

Similarly, Section \ref{SEC-ProofThm2} can be proven in the same way after using \eqref{substitution} and in particular we obtain the bound
\begin{equation*}
	\P \big( S(F,r_0)^{\mathsf{c}} \big)		\le \sum_{r\ge r_0}  r^{\dv+1} e^{		-\Cr{thm2a} r^{c_s 		}}.
\end{equation*}

\section{Survival of the infection}\label{SEC-applications}

We now give an application of the \Lip{} framework to show that for an infection with recovery on a particle system as defined in Subsection \ref{subsec-PRW}, the infection survives indefinitely with positive probability, if the  recovery rate is small enough.

Consider either the \Sier{} gasket $\mathbb{G}^d$ or a generalised \Sier{} carpet $\mathbb{SC}^d(\unit,\remain)$ and the particle system defined in Subsection \ref{subsec-PRW} given by a Poisson point process with intensity $μ(x):=μ_0 λ_x$. Assume furthermore that at time 0, there is an infected particle at the origin of the graph\footnote{The choice of the site where the infection starts is arbitrary as all of the bounds we use are uniform across the graph. Note however that the local geometry of the origin is in fact different from that of any other site in the graph.}. We next describe the dynamics of the infection.

Any particle of the process gets instantaneously infected when it shares a site with an infected particle. For a second parameter $γ>0$, suppose that an infected particle recovers independently at rate $γ$, but can get infected again afterwards. In particular, we allow for a particle to get immediately reinfected if it recovers while sharing a site with an infected particle, i.e.\ recovery is impossible when a particle shares a site with a different particle. However, our application works also in the case where infections can only occur when particles change sites, i.e.\ when a healthy particle jumps to a site with an infected particle or vice versa. To model recovery, consider a collection of Poisson point processes $(R_γ^{x,n})_{x\in \mathbb{G}^d,n\in\N }$ on $\R^+$ with intensity $γ$, which we refer to as the \emph{recovery marks}. As in \cite{BaldassoStauffer}, we view the process $R_γ^{x,n}$ as the recovery marks of the random walk $(X^{x,n}_t)_{t\geq 0}$, where $X^{x,n}_t$ is the location  at time $t$ of $n$-th particle started from $x$ at time $0$. A particle $(X^x_t)_t$ recovers at time $s$ if it is alone, i.e.\ $∏_s(X^x_s)=1$, and $s\in R_γ^{x,n}.$

We say that the infection survives if for every $t>0$ there exists at least one infected particle at time $t$ somewhere on the graph. We denote with $\P_μ^γ$ the distribution of the process with intensity $μ$ and recovery rate $γ.$

\begin{prop}\label{prop-application}
	For any $μ_0>0$ there exists $γ_0>0$ such that for all $0<γ<γ_0$ the infection survives with positive probability.
\end{prop}

Note that Theorem \ref{thm-application} is a special case of the above Proposition, and as such it remains for us to prove the latter.
We will follow the approach introduced in \cite{Pete19b} and refined in \cite{BaldassoStauffer}.
To prove the result we will define a suitable event $E(ι,τ)$ and apply Theorem \ref{thm-main}. We will then be able to infer from the definition of $E(ι,τ)$ and the connectivity properties of the \Lip{} that the infection survives indefinitely almost surely once the infection has entered the \Lip{}, therefore surviving indefinitely as long as the infection does not recover before this. The event $E(ι,τ)$ will then consist of two phases:  in the first phase we will use (some) of the already infected particles to infect a sufficiently large number of the particles in the cell $R_1(ι,τ)$. In the second phase, we will use these newly infected particles to propagate the infection to the surrounding cells.

Fix the value $\ell\in\N$ and consider a value $β$, depending on $\ell$, so that the ratio $\frac{2^{\dw\ell}}{β}$ is fixed. We define $T:= 2^{\ell(\dw -1/3)}$ to play the role of time buffer between the two phases.

Define the following condition: we say that a cell $R_1(ι,τ)$ is \emph{acceptable} if
\begin{enumerate}[label=(A\theenumi), ref=A\theenumi]
	\item \label{item-acceptable1} for every $x\in S_1(ι,τ)$ with $∏_{τβ}(x)>0$ there exists a path denoted with $π^x$, which starts at $x$, and does not exit the super-tile $S_1^3(ι)$ and has no recovery marks up to time $τβ+T$.

	\item \label{item-acceptable2}for each $S_1(ι')\subseteq S_1^3(ι)$ and each $x\in S_1(ι)$ with $∏_{τβ}(x)>0$, there exists a particle which stays inside the super-tile $S_1^3(ι)$ and does not have any recovery marks up to time $(τ+1)β$, is inside $S_1(ι')$ at time $(τ+1)β$ and
	      intersects\footnote{We say that a particle intersects a path if the path and the particle path intersect in space-time, i.e.\ have the same position at the same time at least once.}
	      the path $π^x$ during the time interval $[τβ,τβ+T].$
\end{enumerate}

We now claim 
\begin{equation}\label{eq-accetable}
	\P_μ^γ\big(R_1(ι,τ) \text{ satisfies }\eqref{item-acceptable1},\eqref{item-acceptable2}\big)
	\ge 1-\exp\Big\{ \C μ_0e^{-γβ}2^{\tfrac{\ell/3}{\dw-1}}  \Big\},
\end{equation}
the proof of which we relegate to Appendix \ref{appendix-AccepDecent} since it is an easy adaptation of the work done in \cite{Pete19b,BaldassoStauffer}.

\begin{remark}
	One might be tempted to think that using the event \[E(ι,τ):=\{ R_1(ι,τ) \text{ is acceptable}\} \] and Theorem \ref{thm-main} would yield our claim. This would be true if the infection were to enter the \Lip{} from the time
	dimension\footnote{The infection enters a cell $R_1(ι,τ)$ from the time dimension if there is an infected particle in $S_1(ι)$ at time $τβ$. We say that the infection enters the cell $R_1(ι,τ)$ from the spatial dimension, if there are no infected particles inside $S_1(ι)$ at time $τβ$ and there is an infected particle which enters $S_1(ι)$ at some time $t\in (τβ,(τ+1)β)$.}.
	Then by definition of acceptable, the infection enters from the time dimension in all cells in $R_1^3(ι,τ)$ appearing in \eqref{item-acceptable2},
	including the one in the \Lip{} due to Corollary \ref{corol-LipschitzInTime}, and thus survives indefinitely.
	The next definition takes care of the case in which the infection does not enter from the time dimension when it first enters the \Lip{}.
\end{remark}

For each cell $R_1(ι,τ)$ and each $x\in S_1(ι)$ fix an independent realisation of a random walk path $(π^x_s)_{s\in [0,τβ]}$ with $π^x_{0}=x$. We say that a cell $R_1(ι,τ)$ is \emph{decent} if
\begin{enumerate}[label=(D3), ref=D3]
	\item \label{item-decent0} for every $x\in S_1(ι)$ the path $π^x_s$ has no recovery marks and for every jump time $t$ of $(π^x_s)_{s\in [0,τβ]}$ there exists a tile $S_1(ι')\subseteq S_1^1(ι)$ such that
	      \begin{enumerate}[label=({D3}a), ref=\theenumi\theenumii]
		      \item \label{item-decent1} if $t< (τ+1)β-T$ there exists a particle which has no recovery marks and stays inside $R_1^1(ι,τ)$, is at time $(τ+1)β$ inside $S_1(ι')$ and intersects the path $(π^{x}_{s-t})_{s\in [t,t+T]}$ during the time interval $[t,t+T]$;
		      \item \label{item-decent2} if $(τ+1)β-T \le t \le (τ+1)β$ it holds $π^{x}_{(τ+1)β-t}\in S_1(ι')$.
	      \end{enumerate}
\end{enumerate}
We refer again to Appendix \ref{appendix-AccepDecent} for the proof of
\begin{equation}\label{eq-decent}
	\P_μ^γ(R_1(ι,τ) \text{ is decent})
	\ge
	1- \exp\{-\C β\}- \exp\{-\Cγβ\}
	-\exp \big\{ \C μ_0e^{-γβ} 2^{\tfrac{\ell/3}{\dw-1}}
	\big\},
\end{equation}
as the arguments remain very similar to \cite{BaldassoStauffer}.

\begin{remark}
	We note that unlike done in \cite{BaldassoStauffer}, where the authors introduce a single random walk path $π^0$ for each space-time cell, which they then translate to $x$ as needed, our graphs lack translation invariance and we must therefore consider different paths for each $x$. This however has no bearing on the rest of the argument.
\end{remark}

\begin{proof}[Proof of Theorem \ref{prop-application}]
	We introduce an alternative construction of the process using the additional paths $(π^x_s)_{s}$. We fix the tessellation and observe a cell $R_1(ι,τ).$ If at time $τβ$ there are infected particles inside $S_1(ι)$, we do not use the paths $(π^x_s)$, $x\in S_1(ι).$
	If instead there are no infected particles in $S_1(ι)$ at $τβ$, we observe the process on adjacent tiles and consider the first infected particle which enters the tile $S_1(ι)$ at some site $y$ during $T_1(τ)$, if it exists, and let this particle follow the path $π_s^y$ until $(τ+1)β$ or until it the same rule applies for some adjacent cell, whichever happens first. Then, as simple concatenations of random walks, with this new construction the process maintains the same distribution as the original process.

	We can now define the event
	\begin{equation*}
		E(ι,τ):=\{ \text{all cells } R_1(ι',τ') \text{ adjacent to }R_1(ι,τ)
		\text{ are acceptable and decent}
		\}.
	\end{equation*}
	Then the event $E(ι,τ)$ is increasing, restricted to the super-cell $R_1^4(ι,τ)$ and using the volume estimates \eqref{def-dv} for $\ell$ large enough and $γ$ small enough we can find $α_0>0$ such that  $\P^γ_μ(E(ι,τ)) \ge 1-e^{-α_0}$.

	Then Theorem \ref{thm-main} provides the existence of a \Lip{} $F^{\mathrm{o}}$  such that the event $E(ι,τ)$ holds for all $(ι,τ) \in F^{\mathrm{o}}$ and Theorem \ref{thm-main2} entails that it surrounds the origin at some finite distance $r$ almost surely, hence an initially infected particle starting at the origin has a positive probability of entering a cell in $F^\mathrm{o}$ before recovery.

	Suppose that this infected particle enters the \Lip{} from the time dimension: then it suffices to consider \eqref{item-acceptable1} and \eqref{item-acceptable2} to obtain that the infection spreads to all cells in $R_1^1(ι,τ)$. Since by Corollary \ref{corol-LipschitzInTime} for every cell $R_1(ι,τ)$ in $F^\mathrm{o}$ there exists a cell $R_1(ι',τ+1) \subseteq F^\mathrm{o}$ with  $d(S_1(ι),S_1(ι'))=0$, by definition of \emph{acceptable} cells once the infection enters the \Lip{} it spreads to neighbouring cells inside $F^\mathrm{o}$. Since this observation can then be inductively repeated, the infection now survives almost surely by spreading along cells of $F^\mathrm{o}$.

	Suppose instead that the infected particle enters a \emph{decent} cell $R_1(ι,τ)$ from the spatial dimension. Since the cell is \emph{decent}, the infection spreads to at least one cell $R_1(ι',τ')\subseteq R_1^1(ι,τ)$ which is \emph{acceptable} by the definition of $E(ι,τ)$. Note that this cell might not necessarily be part of $F^\mathrm{o}.$ However, since it is \emph{acceptable} it spreads the infection to all cells $R_1(ι'',τ'') \subseteq R_1^3(ι',τ')$. By Corollary \ref{corol-LipschitzInTime} and since $\eta=3$ there exists in particular at least cell $R_1(ι'',τ'') \subseteq R_1^3(ι',τ')$ that is inside $F^{\mathrm{o}}$. By definition of acceptable cells, the infection enters this cell from the time dimension, and the infection survives indefinitely by the previous argument.

	Since every cell of $F^\mathrm{o}$ is acceptable and decent by construction and the \Lip{} surrounds the origin at almost surely finite distance, this yields the claim.
\end{proof}

\section{Further work}\label{sec:further}

As outlined in the introduction, this work's main contribution is adapting the Lipschitz surface framework of \cite{Pete19} from Euclidean lattices to the sub-diffusive \Sier{} fractal graphs. As such, the application from Section \ref{SEC-applications} represents only the first of many possible problems that can be studied with the help of the Lipschitz cutset framework we have developed.

\paragraph{Further results about the survival of the infection.}
In \cite{BaldassoStauffer} the authors use the Lipschitz surface framework on $\Z^d$ to prove that the infection survives locally with probability 1, conditionally on the infection surviving in the first place. They also show that if the particle intensity $\mu$ is high enough, then the infection has a positive probability of surviving for all recovery rates $\gamma>0$. We conjecture that the same holds also for fractal graphs. The result does not follow directly by just replacing the Lipschitz surface framework in \cite{BaldassoStauffer} with the Lipschitz cutset framework due to the weaker connectivity properties of the cutset. We do however believe  that Theorem \ref{thm-main2} (resp.\ Theorem \ref{thm-main3}) provides sufficient structure to still be able to deduce similar statements for fractal graphs.

\paragraph{Linear speed of the infection.}
Similarly to the above, one cannot immediately recover positive speed of the infection from the Lipschitz cutset, as was done in the case of the Lipschitz surface on $\Z^d$ in \cite{Pete19b}. While the Lipschitz cutset retains the Lipschitz property in the temporal direction (see Corollary \ref{corol-LipschitzInTime}), the Lipschitz property along the spatial axes of the cutset can only be inferred when the cutset intersects with $L_1$. We conjecture that this property can due to Theorem \ref{thm-main2} (resp. Theorem \ref{thm-main3}) be recovered sufficiently often so that positive speed of the infection should also hold for \Sier{} gaskets and carpets.

\paragraph{Shape theorem for the spread of infection.}
After proving positive speed of an infection on $\Z^d$ in \cite{KestenSid05}, Kesten and Sidoravicius used their result to derive a shape theorem for the spread of an infection in their seminal paper \cite{KestenSid08}. Baring that our conjecture above holds, a natural followup would be to try and prove a corresponding result for sub-diffusive graphs such as the fractal graphs in this paper.

\appendix
\section{Probability of acceptable and decent cells}\label{appendix-AccepDecent}
In this appendix we prove equations \eqref{eq-accetable} and \eqref{eq-decent} adapting the proofs of \cite{Pete19b,BaldassoStauffer}. Recall that the ratio $\frac{2^{\dw\ell}}{β}$ is fixed and that $T:=2^{\ell(\dw -\tfrac{1}{3})}.$

\paragraph{Acceptable.}
We start by showing \eqref{eq-accetable}.

\begin{lemma}[{\cite[Lemma 2]{Pete19b}}]\label{lemma-2}
	Assume that the particles in $S_1(ι)$ are a Poisson point process of intensity $\Cl[s]{int} μ_0λ_x$ for some $\Cr{int}>0$. For $x\in S_1(ι)$, let $π^x$ a path of an (infected) particle which starts in $x$ and stays inside $S_1^3(ι)$ during $[τβ,τβ+T]$. Then, for $\ell$ large enough, the number of particles in $S_1^3(ι)$ at time $τβ$ which intersect $π^x$by time $τβ+T$ is a Poisson random variable with mean at least $\Cl{Lemma2} μ_0 2^{\ell(\tfrac{1/3}{\dw-1})}.$
\end{lemma}
\begin{proof}
	The proof is a simple adaptation of {\cite[Lemma 2]{Pete19b}}, using \eqref{eq-HKBound} and splitting time into sub-intervals of length $W:=2^{\ell(\dw-\tfrac13-\tfrac{1/3}{\dw-1})}$.
\end{proof}

\begin{lemma}[{\cite[Lemma 3]{Pete19b}}]\label{lemma-3}
	Given a set of $N\in\mathbb{N}$ particles in $S_1^3(ι)$ at time $τβ+T$ and a tile $S_1(ι')\subseteq S_1^3(ι)$, the probability that at least one of the $N$ particles is in $S_1(ι')$ at time $(τ+1)β$ is at least $1-\exp\{ -N c_p\}$ for some constant $c_p>0$ and $\ell$ large enough.
\end{lemma}
\begin{proof}
	One can define a suitable binomial variable $B$ with parameters $N$ and $p\in (0,1), $ the latter being the minimal probability for a particle to be in $S_1(ι')$ after moving for $β-T$ amount of time, so that the probability in the statement is at least $\P(B\ge 1) \ge 1- \exp\{-N p\}.$
	The estimate $p>c_p$ then follows from  applying \eqref{eq-HKBound} in the time interval $[T,(τ+1)β]$.
\end{proof}

With the help of Lemma \ref{lemma-confining}, we can combine the previous two statements with the help of Chernoff's bound into the following result.
\begin{lemma}[{\cite[Lemma 4]{Pete19b}}]\label{lemma-4}
	Assume that the particles inside $S_1^3(ι)$ at time $τβ$ are a Poisson process of intensity $\Cr{int}μ_0 λ_x$ and let $π^x$ be the path from Lemma \ref{lemma-2}.
	The probability that at time $(τ+1)β$ there is at least one particle in every tile $S_1(ι')\subseteq S_1^3(ι)$ which intersected $π^x$ during $[τβ,τβ+T]$ is at least $1-\exp\{- \Cl{lemma4} μ_02^{\tfrac{\ell/3}{\dw-1}} \}$.
\end{lemma}

Lemma \ref{lemma-4} with the use of a simple union bound across all paths $π^x$ for $x\in S_1^3(ι)$ and \eqref{eq-confining} for \eqref{item-acceptable1} yields
\begin{align*}
	\P_μ^0(R_1(ι,τ) \text{ satisfies }\eqref{item-acceptable1},\eqref{item-acceptable2})
	 & \ge
	1-\sum_{x\in S_1^3(ι)}
	\Big(
	\Cr{conf1}\exp\Big\{-\Cr{conf1}^{-1}2^{\tfrac{\ell/3}{\dw-1}}	\Big\}+
	\exp\Big\{ \Cr{lemma4} μ_02^{\tfrac{\ell/3}{\dw-1}}  \Big\}
	\Big)                                                         \\
	 & \ge 1-\exp\Big\{ -\C μ_02^{\tfrac{\ell/3}{\dw-1}}  \Big\}.
\end{align*}


Applying a further thinning on all of the particles appearing in the previous arguments (as done in detail in \cite[Lemma 3.1]{BaldassoStauffer}), preventing them from recovering during the time interval $[τβ,(τ+1)β]$, one obtains the analogous result with recovery \eqref{eq-accetable}.

\paragraph{Decent.}
We now bound the probability of a cell to be decent and show \eqref{eq-decent}.
The probability that a path has no recovery marks during an interval of length $β$ is $e^{-γβ}$ and it holds
for any random walk that
\begin{equation}\label{eq-applicationConfinment}\P\big(\Conf(\displ_{R},Δ)\big)
	\ge 1- \C R^{\dv} \exp\Big\{-\C \frac{R^2}{Δ} 	\Big\},
\end{equation}
(see for example \cite[(4.1)]{Grigor01}).

We now evaluate the probability of \eqref{item-decent2} for fixed $x,t$. We observe the time interval $[t,(τ+1)β]$: if the length $(τ+1)β-t$ is bigger then $2^\ell$ we can apply Lemma \ref{lemma-confining}; if instead $(τ+1)β-t<2^\ell$ then we can apply \eqref{eq-applicationConfinment} with $R=2^\ell$ and $Δ\le 2^\ell$, which yields a lower bound of $1-\exp\{-\Cl{acc2} 2^{\ell}\}.$ All together
\begin{equation*}
	\P \big(\text{the pair }π^x, t\text{ satisfy }\eqref{item-decent2}\big)
	\ge 1-
	\Cr{conf1} \exp\{
	-\Cr{conf1}^{-1}  2^{\frac{\ell /3}{\dw-1}} \}
	-\exp\{-\Cr{acc2} 2^{\ell}\}.
\end{equation*}

For \eqref{item-decent1}, we adapt a strategy similar to acceptable cells. Lemma \ref{lemma-2} still applies. Lemma \ref{lemma-3} still holds as before if $(τ+1)β-t-T>2^\ell$, if instead $(τ+1)β-t-T<2^\ell$, we need to use \eqref{eq-applicationConfinment} instead of \eqref{eq-HKBound} in the proof of Lemma \ref{lemma-3}. Then Lemma \ref{lemma-4} applies with appropriately modified exponential bounds. Hence, for fixed $x$ and $t$ the probability of \eqref{item-decent1} under $\P^0_μ$ is at least $\exp\{ \Cl{acc1}μ_0 2^{\tfrac{\ell/3}{\dw-1}}\}. $

Note now that the probability that a path has no recovery marks during an interval of length $β$ is $e^{-γβ}$. The probability that a path jumps more than $3β$ times during a time interval of length $β$ is bounded by $e^{-β}$ by a simple Poisson bound. Combined, we obtain

\begin{align*}
	\P_μ^0(R_1(ι,τ) \text{ is decent})
	\ge
	1-\sum_{x\in S_1(ι)}
	\Big( & e^{-γβ}    +
	e^{-β}          +
	3β\exp\big\{ -\Cr{acc1} μ_02^{\tfrac{\ell/3}{\dw-1}}\big\}            \\
	      & +3β\Cr{conf1} \exp\{-\Cr{conf1}^{-1}  2^{\frac{\ell /3}{\dw-1}} \}
	+3β\exp\{-\Cr{acc2} 2^{\ell}\}
	\Big).
\end{align*}

With the thinning property of Poisson point processes we can adapt the calculation for the recovery marks as in \cite{BaldassoStauffer}, and \eqref{def-dv} then yields  \eqref{eq-decent} for $\ell$ large enough since the ratio $\frac{\ell^{\dw}}{β}$ is fixed.

\section{Standard Results}

\begin{lemma}[Chernoff Bound] Let $ P $ be a Poisson random variable with parameter $ λ. $ Then, for $ χ\in(0,1) $
	\begin{equation}\label{eq-Chernoff}
		\P(P<(1-χ)λ)< e^{- λ\frac{χ^2}{2}}.
	\end{equation}
\end{lemma}

A proof of the following result is in \cite[Lemma A.2]{Pete19}
\begin{lemma}
	Let $ x,y\in\N $. Then, for any $ a,b>1 $
	\begin{equation}\label{eq-chooses}
		\binom{x+y}{y} e^{-ax-by}			\le 	e^{-(a-1)x-(b-1)y}.
	\end{equation}
\end{lemma}

\paragraph{Acknowledgements.} AD and PG acknowledge support from DFG through the scientific network {\em Stochastic
		Processes on Evolving Networks.}

\printbibliography

\end{document}